\RequirePackage{fix-cm}
\documentclass[smallextended,twocolumn]{svjour3}          

\smartqed  
\usepackage{graphicx}
\usepackage[disable]{todonotes} 
\usepackage{amssymb}
\usepackage{amsmath}

\usepackage{amsthm}

\usepackage{color}
\usepackage{booktabs}
\DeclareMathAlphabet\gothic{U}{euf}{m}{n}

\newcommand{\R}{\mathbb{R}}
\newcommand{\bbS}{\mathbb{S}}
\newcommand{\bbM}{\mathbb{M}}
\newcommand{\bx}{\mathbf{x}}
\newcommand{\bq}{\mathbf{q}}
\newcommand{\bv}{\mathbf{v}}
\newcommand{\bw}{\mathbf{w}}

\newcommand{\bp}{\mathbf{p}}

\newcommand{\ul}{\mathbf}
\newcommand{\bgamma}{\boldsymbol{\gamma}}

\newcommand{\desda}{\Leftrightarrow}
\newcommand{\ba}{\mathbf{a}}
\newcommand{\gF}{\gothic{F}}
\newcommand{\cC}{\mathcal{C}}
\newcommand{\rmd}{\mathrm{d}}
\newcommand{\mG}{\mathcal{G}}
\newcommand{\cK}{{\mathcal K}}
\newcommand{\bfR}{{\mathbf{R}}}

\DeclareMathOperator\Lip{Lip}

\DeclareMathOperator\Length{Length}

\DeclareMathOperator\Hull{Hull}
\DeclareMathOperator\Id{Id}
\def\bM{{\mathbb M}}
\def\bR{{\mathbb R}}
\def\bS{{\mathbb S}}

\def\bX{{\mathbb X}}
\def\bZ{{\mathbb Z}}
\def\bE{{\mathbb E}}
\def\bN{{\mathbb N}}

\def\cC{{\mathcal C}}
\def\cF{{\mathcal F}}
\def\cG{{\mathcal G}}

\def\cB{{\mathcal B}}

\def\cI{{\mathcal I}}
\def\cH{{\mathcal H}}
\def\cO{{\mathcal O}}

\def\by{{\mathbf y}}
\def\bn{{\mathbf n}}

\def\gL{{\mathfrak L}}
\def\gH{{\mathfrak H}}
\def\gB{{\mathfrak B}}
\def\gC{{\mathfrak C}}

\def\vp{\varphi}
\def\ve{\varepsilon}
\def\sm{\setminus}
\def\<{\langle}
\def\>{\rangle}
\def\trans{{\mathrm T}}

\def\diff{{\rm d}}

\makeatletter
\newcommand{\osetH}[3]{%
  {\mathop{#3}\limits^{\vbox to #1\ex@{\kern-2\ex@
   \hbox{\scriptsize #2}\vss}}}}
\makeatother

\renewcommand\cot[1]{\hat #1}

\def\pTarget{\bp_{\rm T}}
\def\pSource{\bp_{\rm S}}

\begin{document}\sloppy

\title{Optimal Paths for Variants of the 2D and 3D Reeds-Shepp Car with Applications in Image Analysis
\thanks{The research leading to the results of this article has received funding from the European Research Council under the European Community’s 7th Framework Programme \
(FP7/20072014)/ERC grant agreement No. 335555 (Lie Analysis). This work was partly funded by ANR grant NS-LBR. ANR-13-JS01-0003-01.}} 


\author{R. Duits $^*$        \and
        S.P.L. Meesters $^*$\and
        J-M. Mirebeau $^*$\and
        J.M. Portegies $^*$ 
}

\authorrunning{R. Duits, S.P.L. Meesters, J-M. Mirebeau, J.M. Portegies} 

\institute{$^*$ Joint main authors           \and \\
              R. Duits, S.P.L. Meesters, J.M. Portegies\at
              CASA, Eindhoven University of Technology, The Netherlands \\
              Tel.: +31-40-2472859 \\
              \email{$\{$r.duits, s.p.l.meesters, j.m.portegies$\}$@tue.nl}           
           \and
           J-M. Mirebeau \at
           University Paris-Sud, CNRS, University Paris-Saclay, 91405 Orsay, France\\
           \email{jean-marie.mirebeau@math.u-psud.fr}
}



\maketitle
\begin{abstract}
We present a PDE-based approach for finding optimal \todo{R5.1: terminology} paths for the Reeds-Shepp car. In our model we minimize a (data-driven) functional involving both curvature and length penalization, with several generalizations. Our approach encompasses the two and three dimensional variants of this model, state dependent costs, and moreover, the possibility of removing the reverse gear of the vehicle. We prove both global and local controllability results of the models.


Via eikonal equations on the manifold $\bR^d \times \mathbb{S}^{d-1}$ we compute distance maps w.r.t. highly anisotropic Finsler metrics, which approximate the singular (quasi)-distances \todo{R1.3: terminology} underlying the model. This is achieved using a Fast-Marching (FM) method, building on Mirebeau \cite{mirebeau_anisotropic_2014,mirebeau_efficient_2013}. The FM method is based on specific discretization stencils which are adapted to the preferred directions of the Finsler metric and obey a generalized acuteness property. The shortest paths can be found with a gradient descent method on the distance map, which we formalize in a theorem. We justify the use of our approximating metrics by proving convergence results.

Our curve optimization model in $\bR^{d} \times \mathbb{S}^{d-1}$ with data-driven cost allows to extract complex tubular structures from medical images, e.g. crossings, and incomplete data due to occlusions or low contrast. Our work extends the results of  Sanguinetti et al. \cite{sanguinetti_sub-riemannian_2015} on numerical sub-Riemannian eikonal equations and the Reeds-Shepp Car to 3D, with comparisons to exact solutions by Duits et al. \cite{duits_sub-riemannian_2016}.

Numerical experiments show the high potential of our method in two applications: vessel tracking in retinal images for the case $d=2$, and brain connectivity measures from diffusion weighted MRI-data for the case $d=3$, extending the work of Bekkers et al \cite{bekkers_pde_2015}. We demonstrate how the new model without reverse gear better handles bifurcations.

\end{abstract}

\listoftodos[Changes to document]
\section{Introduction}
\label{sec:intro}

Shortest paths in position and orientation space are central in this paper. Dubins describes in \cite{dubins_curves_1957} the problem of finding shortest paths for a car in the plane between initial and final points and direction, with a penalization on the radius of curvature, for a car that has no reverse gear. Reeds and Shepp consider in \cite{reeds_optimal_1990} the same problem, but then for a car that does have the possibility for backward motion. In both papers, the focus lies on describing and proving the general shape of the optimal paths, without giving explicit solutions for the shortest paths.

This can be considered a curve optimization problem in the space $\R^2 \times (\bR/2 \pi\bZ)$, equipped with the natural Euclidean metric but only among curves $\gamma(t) = (x(t),y(t),\theta(t))$ subject to the constraint that $(\dot x(t), \dot y(t))$ is proportional to $(\cos \theta(t), \sin \theta(t))$. Formulating the problem this way, it becomes one of the simplest examples of sub-Riemannian (SR) geometry: the tangent vector $\dot \gamma(t)$ is constrained to remain in the span of $(\cos \theta(t),\sin \theta(t),0)$ and $(0,0,1)$, see Fig. \ref{fig:intro_lifting}\todo{R1.1}. The SR curve optimization problem and the properties of its geodesics in $\R^2 \times \bS^1$ have been studied and applied in image analysis by \cite{petitot_neurogeometry_2003,citti_cortical_2006,duits_association_2013,boscain_curve_2014,mashtakov_parallel_2013,agrachev_control_2004}, and in particular for modelling the Reeds-Shepp car in \cite{moiseev_maxwell_2010,boscain_existence_2010,sachkov_cut_2011}, whereas the latter presented a complete and optimal synthesis for the geometric control problem on $\R^{2}\times \bS^{1}$ with uniform cost. Properties of SR geodesics in $\R^d \times \bS^{d-1}$ with $d=3$ have been studied in \cite{duits_sub-riemannian_2016} and for general $d$ in \cite{duits_cuspless_2014}.  Apart from the Reeds-Shepp car problem, there are other examples relating optimal control theory and SR geometry, see for example the books by Agrachev and Sachkov \cite{agrachev_control_2004} and Montgomery \cite{montgomery_tour_2002}. Applications in robotics and visual modeling of SR geometry and control theory can be found in e.g. \cite{stefani_applications_2014}. \todo{R1.2: literature}.

\begin{figure}[t]\label{fig:intro_lifting}
\centerline{
\includegraphics[width= \hsize]{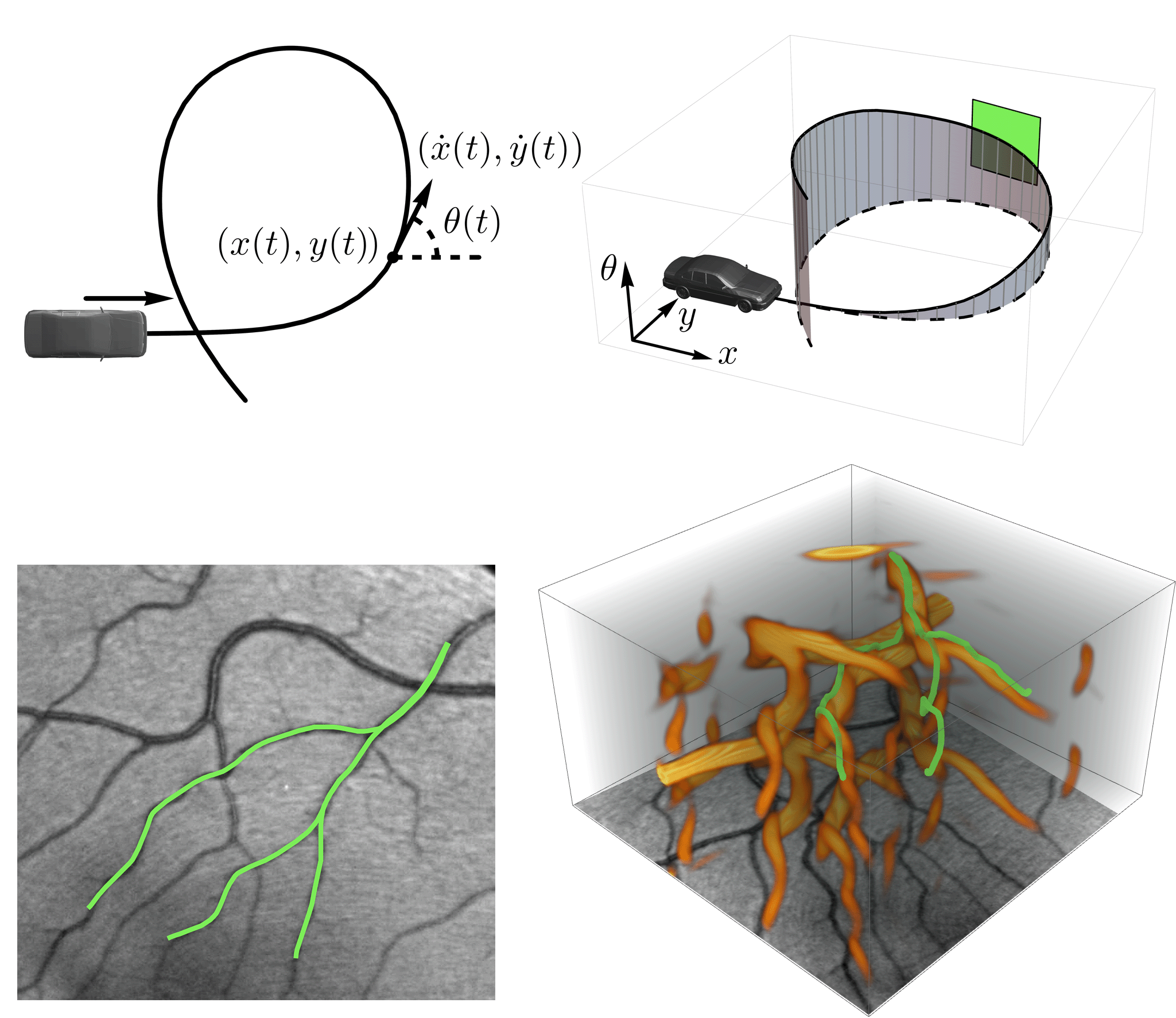}}

\caption{Top: A car can only move in its current orientation or change its current orientation. In other words, when the path $\gamma(t) = (x(t),y(t),\theta(t))$ is considered as indicated in the left figure, the tangent $\dot{\gamma}(t)$ is restricted to the span of $(\cos \theta (t), \sin \theta(t),0)$ and $(0,0,1)$, of which the green plane on the right is an example. Bottom: the meaning of shortest path between points in an image is determined by a combination of a cost computed from the data, the restriction above, and a curvature penalization. The path optimization problem is formulated on the position-orientation domain such as in the image on the right. The cost for moving through the orange parts is lower than elsewhere.}
\end{figure}

\todo[size=\tiny]{R5.2: Adapted Fig. 1} On the left in Fig. \ref{fig:freevspositive_introfig_full}, we show an example of an optimal path between two points in $\R^2 \times \bS^1$. The projection on $\R^2$ of this curve has two parts where the car moves in reverse (the red parts of the line), resulting in two cusps. From the perspective of image analysis applications this is undesirable and it is a valid question what the optimal paths are if cusps and reverse gear are not allowed. In this paper, similar to the difference between the Dubins car and the Reeds-Shepp car, we also consider this variant: it can be accounted for by requiring that the spatial propagation is forward. This variant falls outside the SR framework and requires asymmetric Finsler geometry instead.

Furthermore, we would like to extend the Finsler metric using two data-driven factors that can vary with position and orientation. This can be used to compute shortest paths for a car, where for example road conditions and obstacles are taken into account. In \cite{bekkers_pde_2015} it is shown this approach is useful for tracking vessels in retinal images. Likewise, the 3D variant of the problem provides a basis for algorithms for blood vessel detection in 3D Magnetic Resonance Angiography (MRA) data, or detection of shortest paths and quantification of structural connectivity in 5D diffusion weighted Magnetic Resonance Imaging (MRI) data of the brain.

\begin{figure*}[t]
\centerline{
\includegraphics[width= \hsize]{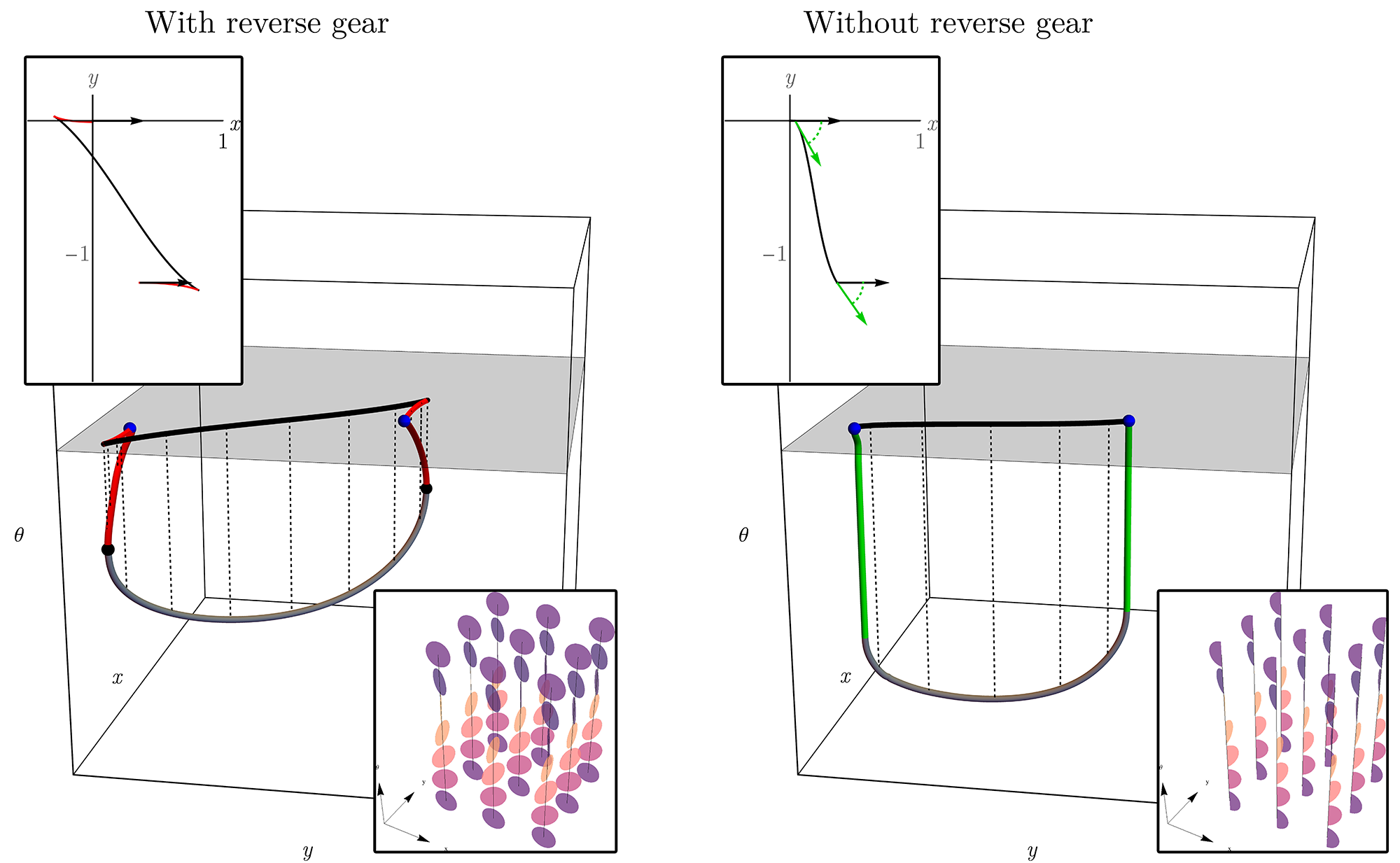}}
\caption{Top: Example of a shortest path with (left) and without (right) reverse gear in $\R^2 \times S$ and its projection on $\R^2$. 
The black arrows indicate the begin and end condition in the plane, corresponding to the blue dots in $\R^2 \times S$. The paths in the lifted space are smooth, but vertical tangents appear in both cases. 
In the left figure, the projection of the path has two cusps, and the first and last part of the path is traversed backwards (the red parts). On the right, backward motion is not possible. Instead, according to our model, the shortest path is a concatenation of an in-place rotation (green), a SR geodesic, and again an in-place rotation.
Bottom: corresponding control sets as defined in (\ref{controlset}) for the allowed velocities at each position and orientation, with $B_{\cF_0}$ on the left and $B_{\cF_0^+}$ on the right.
}\label{fig:freevspositive_introfig_full}
\end{figure*}

\subsection{A distance function and the corresponding shortest paths on $\bR^d \times \bS^{d-1}$}\label{sec:introdistance}
We fix the dimension $d\in \{2,3\}$ 
, and let $\bM := \bR^d\times \bS^{d-1}$ be the $2d-1$ dimensional manifold of positions and orientations. We use a Finsler metric on the tangent bundle of $\bbM$, $\cF : T(\bM) \to [0,+\infty]$ , of which specific properties are discussed later, to define a geometry on $\bM$. Any such Finsler metric $\cF$ induces a measure of length $\Length_\cF$ on the class of paths with Lipschitz regularity, defined as\footnote{In contrast to previous works \cite{duits_sub-riemannian_2016,boscain_curve_2014,bekkers_pde_2015,mashtakov_tracking_2017,duits_association_2013} we parameterize such that the time integration stays on $[0,1]$, and $t>0$ is \emph{not} a priori reserved (unless explicitly stated otherwise) for arc length parametrization (which satisfies $\mathcal{F}_{\bgamma(t)}(\dot{\bgamma}(t))=1$).}
\begin{equation*}
	\Length_\cF(\gamma) := \int_0^1 \cF(\gamma(t), \dot \gamma(t)) \, {\rm d}t,
\end{equation*}
with the convention $\dot \gamma(t) := \frac d {dt} \gamma(t)$. The path is said to be \emph{normalized} w.r.t.\ $\cF$ iff $\cF(\gamma(t), \dot \gamma(t)) = \Length_\cF(\gamma)$ for all $t \in [0,1]$. Any Lipschitz continuous path of finite length can be normalized by a suitable reparametrization.
Finally, the quasi-distance \todo{R1.3: terminology} $d_\cF : \bM \times \bM \to [0,+\infty]$ is defined for all $\bp,\bq\in \bM$ by
\begin{equation}
\label{eqdef:dF}
\begin{split}
	d_\cF(\bp, \bq) := \inf \{ &\Length_\cF(\gamma) \; |\; \gamma \in \Gamma,\, \gamma(0)=\bp, \\ &\hspace*{9.2em}\gamma(1)=\bq\}, 
\end{split}
\end{equation}
with $\Gamma:=\Lip([0,1], \bM)$.
Normalized minimizers of \eqref{eqdef:dF} are called minimizing \todo{R5.1: terminology} geodesics from $\bp$ to $\bq$ w.r.t.\ $\cF$. For certain pairs $(\bp,\bq)$ these minimizers may not be unique, and these points are often of interest, see for example \cite{moiseev_maxwell_2010,bekkers_vessel_2017}

\begin{definition}[Maxwell point] \label{def:Maxwellpoint}
Let $\bp_S \in \bbM$ be a fixed point source and $\gamma \in \Gamma$ a geodesic connecting $\bp_S$ with $\bq \in \bbM$, $\bq \neq \bp_S$. Then $\bq$ is a Maxwell point if there exists another extremal path $\tilde{\gamma} \in \Gamma$ connecting $\bp_S$ and $\bq$, with $\Length_\cF(\gamma) = \Length_\cF(\tilde{\gamma})$. If $\bq$ is the first point (distinct from $\bp_S$) on $\gamma$ where such $\tilde{\gamma}$ exists, then $\bq$ is called the first Maxwell point. The curves $\gamma, \tilde{\gamma}$ lose global optimality after the first Maxwell point.
\end{definition}

\begin{remark}[Terminology]
We use the common terminology of `Finsler metric' for $\cF$, although it is also called  `Finsler function', `Finsler norm' or `Finsler structure', and despite the fact that $\cF$ is not a metric (distance) in the classical sense. The Finsler 
metric $\cF$ induces the quasi-distance $d_{\cF}$ as defined in \eqref{eqdef:dF}. If  $\cF(\bp,\dot{\bp}) = \cF(\bp,-\dot{\bp})$ for all $\bp \in \bbM$ and tangent vectors $\dot{\bp} \in T_{\bp}(\bbM)$, then $d_{\cF}$ is a true metric, satisfying $d_{\cF}(\bp,\bq) = d_{\cF}(\bq,\bp)$ for all $\bp,\bq \in \bbM$. However, to avoid confusion of the word metric, we will only refer to $d_{\cF}$ as a distance or quasi-distance. If the `Finsler metric' $\mathcal{F}$ is induced by a metric tensor field $\mathcal{G}$ on  Riemannian manifold $(\mathbb{M},\mathcal{G})$
then one has $\mathcal{F}(\ul{p},\dot{\ul{p}})=\sqrt{\left.\mathcal{G}\right|_{\ul{p}}(\dot{\ul{p}},\dot{\ul{p}})}$.

Throughout the document, we use the words path and curve synonymously. When we consider the formal curve optimization problem \eqref{eqdef:dF}, we speak of geodesics for the stationary curves. Such stationary curves are locally minimizing. A global minimizer of \eqref{eqdef:dF} is referred to as minimizing geodesic or minimizer.
\end{remark}

\subsection{Geometry of the Reeds-Shepp model}\label{subsec:introgeometry}
We introduce the Finsler metric $\cF_0$ underlying the Reeds-Shepp car model, and the Finsler metric $\cF_0^+$ 
corresponding to the variant without reverse gear. Let $(\bp, \dot \bp) \in T(\bM)$ be a pair consisting of a point $\bp \in \bbM$ and a tangent vector $\dot \bp \in T_\bp(\bbM)$ at this point. The physical and angular components of a point $\bp\in \bM$ are denoted by $\bx\in \bR^d$ and $\bn\in \bS^{d-1}$, and this convention carries over to the tangent:
\begin{align*}
\hspace{3.5em}
	\bp &= (\bx,\bn), &
	\dot \bp &= (\dot \bx, \dot \bn) \in T_\bp( \bM).
\end{align*}
We say that $\dot \bx$ is proportional to $\bn$, that we write as $\dot \bx \propto \bn$, iff there exists a $\lambda \in \bR$ such that $\dot \bx = \lambda \bn$. Define
\begin{align}
\label{eqdef:ReedSheppMetric}
	\cF_0(\bp, \dot \bp)^2 &:=
	\begin{cases}
	 \cC^{2}_1(\bp) |\dot \bx \cdot \bn|^2 +\cC_{2}^2(\bp) \|\dot \bn\|^2  & \text{if }\dot \bx \propto \bn, \\
	+\infty & \text{otherwise.}	
	\end{cases}
\\
\label{eqdef:ReedSheppForwardMetric}
	\cF^+_0(\bp, \dot \bp)^2 &:=
	\begin{cases}
	 \cC^{2}_1(\bp) |\dot \bx \cdot \bn|^2 + \cC_{2}^2(\bp) \|\dot \bn\|^2 & \parbox[t]{.6\textwidth}{ if $\dot \bx \propto \bn$ and \newline $\dot \bx \cdot \bn \geq 0$,} \\
	+\infty & \text{otherwise.}	
	\end{cases}
\end{align}
Here $\| \cdot \|$ denotes the norm and ``$\cdot$'' the usual inner product on the Euclidean space $\R^d$. The functions $\cC_{1}$ and $\cC_{2}$ are assumed to be continuous on $\bM$, and uniformly bounded from below by a positive constant $\delta>0$. In applications, $\cC_{1}$ and $\cC_{2}$ are chosen so as to favor paths which remain close to regions of interest, e.g.\ along blood vessels in retinal images, see Fig. \ref{fig:intro_lifting}. Note that their physical units are distinct: if one wishes $d_\cF$ to have the dimension $[T]$ of a travel time, then $\cC_1^{-1}$ is a physical, (strictly) spatial velocity $[\mathrm{Length}][T]^{-1}$, and $ \cC_2^{-1}$ is an angular velocity $[\mathrm{Rad}][T]^{-1}$. For simplicity one often sets $\cC_{1} = \xi \cC_{2}$, where $\xi^{-1}>0$ is a unit of spatial length. The special case $\cC_1(\bp) = \xi \cC_2(\bp) = \xi$ for all $\bp \in \bM$ is referred to as the uniform cost case.

\subsection{The eikonal equation and the fast marching algorithm}\label{sec:introeikonal}
We compute the distance map to a point source on a volume using the relation to eikonal equations. Let $\pSource\in \bM$ be an arbitrary source point, and let $U$ be the associated distance function
\begin{equation}\label{eqdef:U}
	U(\bp) := d_\cF(\pSource,\bp).
\end{equation}
Then $U$ is the unique viscosity solution \cite{crandall_viscosity_1983,crandall_users_1992} to the eikonal PDE:
\begin{equation}
\label{eqdef:EikonalPDE}
\left\{\begin{aligned}
	&\cF^*(\bp,\diff U(\bp)) = 1 \qquad
	\text{ for all } \bp \in \bM\sm \{\pSource\}, \\
	&U(\pSource) = 0.
	\end{aligned} \right.
\end{equation}
Here $\cF^*$ is the dual metric of $\cF$ and $\rmd U$ is the differential of the distance map $U$. However, for these relations to hold, and for numerical discretization to be practical, $\cF$ should be at least continuous\footnote{From a theoretical standpoint, one may rely on the notion of discontinuous viscosity solution \cite{Bardi:2008wn}. But this concept is outside of the scope of this paper, and in addition it forbids the use of a singleton $\{\bp_S\}$ as the target set.}. We therefore propose in Section \ref{subsec:ApproximateReedShepp} for both $\cF_0$ and $\cF_0^+$ an approximating metric, that we denote by $\cF_\ve$ and $\cF_\ve^+$, respectively, that are continuous and converge to $\cF_0$ and $\cF_0^+$ as $\ve \rightarrow 0$. The approximating metrics correspond to a highly anisotropic Riemannian and Finslerian metric, rather than a sub-Riemannian or sub-Finslerian metric. The metric $\cF_{\epsilon}$ is in line with previous approximations \cite{citti_cortical_2006,bekkers_pde_2015,sanguinetti_sub-riemannian_2015} for the case $d=2$.


We design a monotone and causal discretization scheme for the static Hamilton-Jacobi PDE \eqref{eqdef:EikonalPDE}, which allows to apply an \todo{E.4: typo} efficient, single pass Fast-Marching Algorithm \cite{tsitsiklis_efficient_1995}.
Let us emphasize that designing a causal discretization scheme for \eqref{eqdef:EikonalPDE} is non-trivial, because its local connectivity needs to \todo{R5.3: typo} obey an \emph{acuteness property}
\cite{sethian_ordered_2001,vladimirsky_static_2006}
depending on the geometry defined by $\cF$. We provide constructions for the metrics $\cF_\ve$ or $\cF_\ve^+$ of interest, based on the earlier works \cite{mirebeau_anisotropic_2014,mirebeau_efficient_2013}.

\subsection{Shortest Paths and Minimal Distances in Medical Images}

\begin{figure*}[t]
\centering
\includegraphics[width=0.9 \hsize]{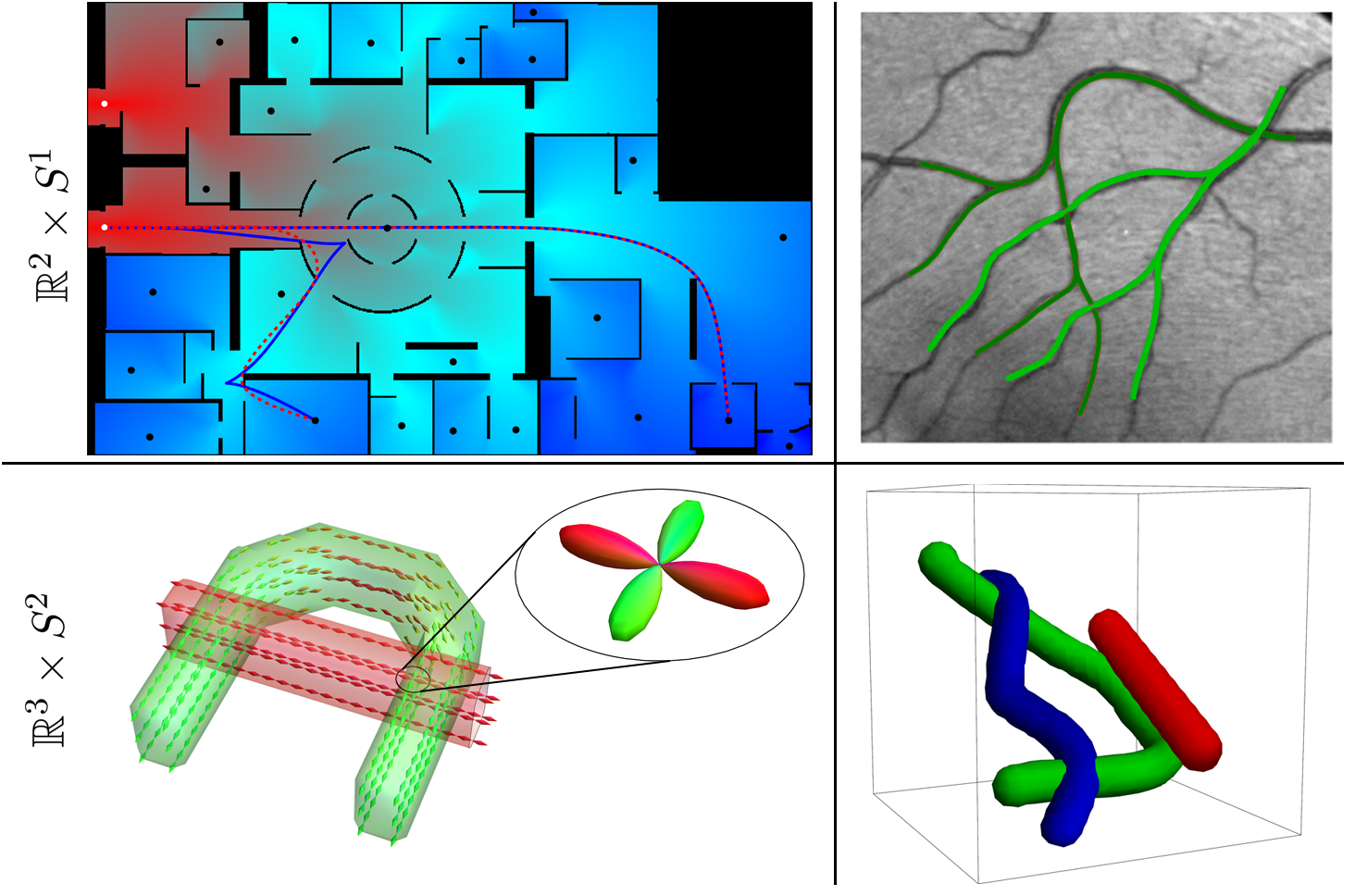}
\caption{Challenges and applications. Top row: the case $d = 2$, with a toy problem for finding the shortest way with or without reverse gear (blue and red, respectively) to the exit in Centre Pompidou (top left) and a vessel tracking problem in a retinal image. Bottom row: the case $d = 3$, connectivity in (simulated) dMRI data. Left: visualization of a dataset with two crossing bundles without torsion, with a glyph visualization of the data in $\R^3 \times \bS^2$ and a magnification of one such glyph, indicating two main fiber directions. Right: the spatial configuration in $\R^{3}$ of bundles with torsion in an artificial dataset on $\R^3 \times \bS^2$.}\label{fig:introfig_challenges}
\end{figure*}

The application of the Hamilton-Jacobi framework for finding shortest paths has been shown to be useful for vessel-tracking in retinal images \cite{bekkers_pde_2015}, see Fig. \ref{fig:introfig_challenges} (top, right) \todo{E.1: adapted fig. 2}. The computational advantage of the fast-marching solver over the numerical method in \cite{bekkers_pde_2015} in this setting was demonstrated by Sanguinetti et al. \cite{sanguinetti_sub-riemannian_2015}\todo{R1.2:literature}. A related approach using fast marching with elastica functionals can be found in \cite{chen_vessel_2016,chen_global_2016}. The sub-Riemannian approach by Bekkers et al. \cite{bekkers_pde_2015} concerns the two-dimensional Reeds-Shepp car model with reverse gear, where 2D gray-scale images are first lifted to an orientation score defined on the higher dimensional manifold $\bR^2 \times \bS^1$. There, the combination of the sub-Riemannian metric, the cost function derived from the orientation score, and the numerical fast-marching solver, provided a solid approach to accurately track vessels in challenging sets of images.

In the previous works \cite{bekkers_pde_2015} and \cite{bekkers_vessel_2017} the clear advantage
of sub-Riemannian geometrical models over isotropic Riemannian models on $\R^{2}\times \bS^{1}$ has been shown
with many experiments\footnote{For vessel tracking experiments that show the benefit of the \emph{sub-}Riemannian approach $(\R^{2}\times \bS^{1},d_{\mathcal{F}_0})$ in \cite{bekkers_pde_2015} see:\ {\scriptsize
http://epubs.siam.org/doi/suppl/10.1137/15M1018460}.
}.

In this work we will show similar benefits for our sub-Riemannian tracking in $\R^{3}\times \bS^{2}$.
In general, regardless the choice of image dimension $d \in \{2,3\}$, one has that our extension of the Hamilton-Jacobi framework from the conventional
base manifold of position space only (i.e. $\R^d$)
to the base manifold of positions and orientations (i.e. $\R^d \times \bS^{d\!-\!1}$),
generically deals with the `leakage problem' where wavefronts leak at crossings in the conventional eikonal frameworks acting directly in the image domain.
See Fig.~\!\ref{fig:leakage} where our solution to the `leakage problem' is illustrated for $d=2$.
\begin{figure*}
\centering
\includegraphics[width=0.75\hsize]{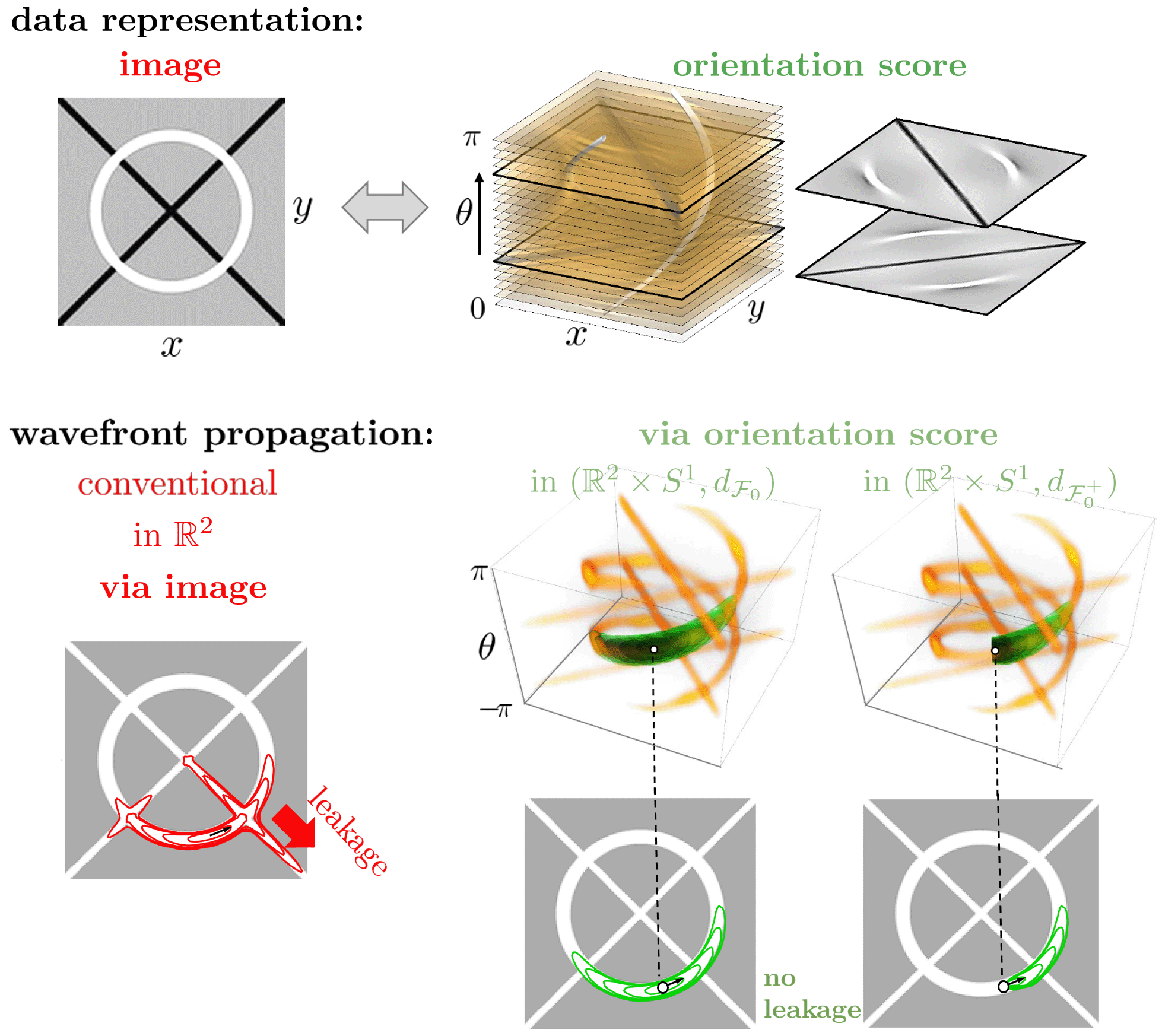}
\caption{Top: An orientation score \cite{duits_image_2006,MichielJMIV2} 
provides a complete overview of how the image is decomposed out of local orientations. It is a method that enlarges the image domain from $\R^{d}$ to $\R^{d} \times \bS^{d\!-\!1}$ (here $d=2$).
Bottom: Conventional geodesic wavefront propagation in images (in red) typically leaks at crossings, whereas wavefront propagation
in orientation scores (in green) does not suffer from this complication. A minimum intensity projection over orientation gives optimal fronts in the image.
The cost for moving through the orange
parts is lower than elsewhere, and is computed from the orientation score, see e.g.~\cite{bekkers_pde_2015}.
The `leakage problem' is gone both for propagating symmetric sub-Riemannian spheres (left), and it is also gone for propagation of asymmetric Finsler spheres (right). \label{fig:leakage}
}
\end{figure*}

Regarding image analysis applications, we propose to use the same strategy of sub-Riemannian and Finslerian
tracking above the extended base manifold \mbox{$\R^{3} \times \mathbb{S}^{2}$} of positions and orientations for fiber tracking and structural connectivity in brain white matter in
diffusion-weighted MRI data. 

For diffusion-weighted MRI images, a signal related to the amount of diffusion of water molecules is measured, which in the case of neuroimages is considered to reflect the structural connectivity in brain white matter.
The images can in a natural way be considered to have domain $\Omega \subset \bR^3 \times \bS^2$. Fig.~\ref{fig:introfig_challenges} (bottom) illustrates such images. On the left we use a glyph visualization, that shows a surface for each grid point, where the distance from the surface to the corresponding grid point $\ul{x}$ is proportional to the data-value $U(\ul{x},\ul{n})$ and the coloring is related to the orientation $\ul{n} \in \bS^{2}$. As such the dMRI data already provide a distribution on $\R^{3}\times \bS^2$ and does not require an `orientation score' as depicted in Fig.~\!\ref{fig:intro_lifting} and Fig.~\!\ref{fig:leakage}.

A large number of tractography methods exist, that are designed to estimate/approximate the fiber paths in the brain based on dMRI data. Most of these methods construct tracks that locally follow the structure of the data, see e.g. \cite{tournier_mrtrix:_2012,descoteaux_deterministic_2009} or references in \cite{jbabdi_tractography:_2011}. More related to our approach are geodesic methods, that have the advantage that they minimize a functional, and thereby are less sensitive to noise and provide a certain measure of connectivity between regions. These methods can be based on diffusion tensors in combination with Riemannian geometry on position space, e.g. \cite{Fletcher,lenglet_brain_2009,jbabdi_accurate_2008}. One can also make use of the more general Finsler geodesic tracking to include directionality \cite{melonakos_finsler_2007,melonakos_finsler_2008}, and use high angular resolution data (HARDI), examples of which can be found in \cite{SepasianPhD,Laura}. Recently, a promising method has been proposed, based on geodesics on the full position and orientation space using a data-adaptive Riemannian metric \cite{pechaud_brain_2009}. We also work on this joint space of positions and orientations, but use either Riemannian or asymmetric Finsler metrics that are highly anisotropic, that we solve by a numerical fast marching method that is able to deal with this high anisotropy. We show on artificial datasets how our method can be employed to give shortest paths \todo{R5.1: terminology} between two regions w.r.t the imposed Finsler metric, and that these paths correctly follow the bundle structure.

\subsection{Contributions and Outline}
The extension to 3D of the Reeds-Shepp car model and the adaptation to model shortest paths for cars that cannot move backwards are new and provide an interesting collection of new theoretical and practical results:

\begin{itemize}
\item In Theorem \ref{th:Controllability} we show that the Reeds-Shepp model is globally and locally controllable, and that the Reeds-Shepp model without reverse gear is globally but not locally controllable. Hence the distance map loses continuity.

\item We introduce regularizations $\cF_\ve$ and $\cF_\ve^+$ of the Finsler 
metrics $\cF_0$ and $\cF_0^+$, which make our numerical discretization possible. We show that both the corresponding distances converge to $d_{\cF_0}$ and $d_{\cF_0^+}$ as $\ve \rightarrow 0$ and the minimizing curves converge to the ones for $\ve = 0$, see Theorem \ref{th:ReedSheppCV}.

\item We present and prove for $d =2$ and uniform cost a theorem that describes the occurrence of cusps for the sub-Riemannian model using $\cF_0$, and that using $\cF_0^+$ leads to geodesics that are a concatenation of purely angular motion, a sub-Riemannian geodesic without cusps and again a purely angular motion. We call the positions where in-place rotation (or purely angular motion) takes place \emph{keypoints}. For uniform cost, we show that the only possible keypoints are the begin and end point, and for many end conditions we can describe how this happens. The precise theoretical statement and proof are found in Theorem \ref{th:CuspsAndRotations}.

\item Furthermore, we show in Theorem \ref{th:Backtracing} how the geodesics can be obtained from the distance map, for a general Finsler metric, and in the more specific cases that we use in this paper. For our cases of interest, we show that backtracking of geodesics is either done via a single intrinsic gradient descent (for the models with reverse gear), or via two intrinsic gradient descents (for the model without reverse gear).

\item For our numerical experiments we make use of a Fast-Marching implementation, for $d = 2$ introduced in \cite{mirebeau_anisotropic_2014}. In Section 6 we give a summary of the numerical approach for $d = 3$, but a detailed discussion of the implementation and an evaluation of the accuracy of the method is beyond the scope of this paper, and will follow in future work. For $d = 2$, we show an extensive comparison between the models with and without reverse gear for uniform cost, to illustrate the useful principle of the keypoints, and to show the qualitative difference between the two models. In examples with non-uniform cost, see for example the top row of Fig. \ref{fig:introfig_challenges}, we show that the model places the keypoints optimally at corners/bifurcations in the data, where the in-place rotation forms a natural, \emph{automatic} `re-initialization' of the tracking.

For $d = 3$, we give several examples to show the influence of the model parameters, in particular the cost parameter. The examples indicate that the method adequately deals with crossing or kissing structures.

\end{itemize}

\paragraph{Outline} In Section \ref{sec:results}, we give a detailed overview of the theoretical results of the paper. The theorems \ref{th:Controllability}, \ref{th:CuspsAndRotations} and \ref{th:Backtracing} are discussed and proven in Sections \ref{ch:3}, \ref{ch:proofcuspskeypoints} and \ref{ch:proofbacktracking}, respectively. The reader who is primarily interested in the application of the methods may choose to skip these three sections. The proof of Theorem \ref{th:ReedSheppCV} is given in Appendix \ref{app:WellPosedness}. We discuss the numerics briefly in Section \ref{sec:Implementation}. Section \ref{sec:Applications} contains all experimental results. Conclusion and discussion follow in Section \ref{sec:Conclusion}. For an overview of notations, Appendix \ref{app:TableOfNotations} may be helpful.

\section{Main results}
\label{sec:results}

In this section, we state formally the mathematical results announced in Section~\ref{sec:intro}. Some preliminaries regarding the distance function are introduced in the Section below. Results regarding the exact Reeds-Shepp car models are gathered in Section \ref{subsec:Geometry}. The description of the approximate models and the related convergence results appear in Section \ref{subsec:ApproximateReedShepp}.
Analysis of special interest points (cusps and keypoints) are done in Section \ref{subsec:interestpoints}.
Results on the eikonal equation, and subsequent backtracking of minimizing geodesics via intrinsic gradients is presented in Section \ref{subsec:eikonal}.

\subsection{Preliminaries on the (Quasi-)Distance Function and Underlying Geometry}\label{sec:preliminaries}

Geometries on the manifold of states $\bM = \bR^d \times \bS^{d-1}$ are defined by means of Finsler metrics %
which are functions $\cF : T(\bM) \to [0,+\infty]$. On each tangent space, the metric should be $1$-homogeneous, convex and quantitatively non-degenerate
 with a uniform constant $\delta >0$: for all $\bp=(\ul{x},\ul{n})\in \bM$, $\dot \bp, \dot \bp_0,\dot \bp_1\in T_\bp(\bM)$, and $\lambda\geq 0$:
\begin{align}
\centering
\label{eqdef:Metric}
	\cF(\bp, \lambda \dot \bp) &= \lambda \cF(\bp, \dot \bp), \nonumber \\
	\cF(\bp, \dot \bp_0+\dot \bp_1) &\leq \cF(\bp, \dot \bp_0) + \cF(\bp, \dot \bp_1),  \nonumber \\
	\cF(\bp,\dot \bp) &\geq \delta\sqrt{\|\dot \bx\|^2 + \|\dot \bn\|^2}. 
\end{align}
A weak regularity property is required as well, see the next remark.
The induced distance $d_\cF$, defined in \eqref{eqdef:dF}, obeys $d_\cF(\bp, \bq) = 0$ iff $\bp = \bq$, and obeys the triangle inequality. However, unlike a regular distance, $d_\cF$ needs not be finite, or continuous, or symmetric in its arguments. Note that $\cF_0$ and $\cF_0^+$ as defined in \eqref{eqdef:ReedSheppMetric} and \eqref{eqdef:ReedSheppForwardMetric}, respectively, indeed satisfy the properties in \eqref{eqdef:Metric}.

\begin{remark}
In contrast to the more common definition of Finsler metrics, we will \emph{not} assume the Finsler metric to be smooth on
$T(\bM)$, but use a weaker condition instead. Following \cite{DaChen2016Thesis}, we require that the sets
\begin{equation} \label{controlset}
\cB_\cF(\bp) := \{ \dot \bp \in T_\bp \bM \, | \, \cF(\bp, \dot \bp) \leq 1\}
 \end{equation}
 are closed and vary continuously with respect to the point $\bp\in \bM$ in the sense of the Hausdorff distance.
The sets $\cB_\cF(\bp)$ are illustrated in Fig. \ref{fig:freevspositive_introfig_full} for the models of interest.
The condition implies that a shortest \todo{R5.1: terminology} path exists from $\bp$ to $\bq\in \bM$ whenever $d_\cF(\bp, \bq)$ is finite, and is used to prove convergence results in Appendix \ref{app:WellPosedness}.
\end{remark}

A common technique in optimal control theory is to reformulate the shortest path problem defining the distance $d_\cF(\bp, \bq)$ into a time optimal control problem.
That is, for $p \in [1,\infty]$ one has by H\"{o}lder's (in)equality, time re-parametrization, and by 1-homogeneity of $\mathcal{F}$ in its 2nd entry, that:
\begin{align}
&d_{\mathcal{F}}(\bp,\bq) = \\ \ & \textstyle\inf \small\{\small\int \limits_{0}^{1} \mathcal{F}(\gamma(t),\dot{\gamma}(t))\, {\rm d}t \;|\; \gamma \in \Gamma,\;  \nonumber \gamma(0)=\bp,  \gamma(1)=\bq \small\} \nonumber \\
  =&  \inf \small\{(\textstyle\int \limits_{0}^{1} |\mathcal{F}(\gamma(t),\dot{\gamma}(t))|^{p}\, {\rm d}t) ^{\frac{1}{p}} \!  | \gamma \in \Gamma, \gamma(0)=\bp,  \gamma(1)=\bq \small\} \nonumber \\[8pt]
 =& \inf \small\{ T \geq 0 \;\;|\;\; \exists \gamma \in \Gamma_T, \; \gamma(0)=\bp, \nonumber \\ &\hspace*{3em}\gamma(T)=\bq, \forall_{t \in [0,T]}\,\dot{\gamma}(t) \in \mathcal{B}_{\mathcal{F}}(\gamma(t)) \small\} \label{viewpoint},
\end{align}
where $\Gamma_T :=\Lip([0,T],\bM)$, and with $\cB_\cF(\bp)$ as defined in \eqref{controlset}.
The latter reformulation is used in Appendix~A to prove convergence results via
closedness of controllable paths and Arzela-Ascoli's theorem, based on a general result originally applied to Euler elastica curves in \cite{DaChen2016Thesis}.

In the special case $\mathcal{F}=\mathcal{F}_{0}$ the geodesics are SR geodesics, where $\cF_0$ is obtained by the square root of quadratic form associated to a SR metric $\left.\mathcal{G}_0\right|_{\bp}(\cdot,\cdot)=\mathcal{F}_{0}(\bp,\cdot)^2$ on
a SR manifold $(\mathbb{M}, \Delta, \mathcal{G}_0)$, where $\Delta \subset T(\mathbb{M})$ is a strict subset of allowable tangent vectors that comes along with the horizontality constraint
\begin{equation} \label{hconstraint}
\dot{\ul{x}}(t) = (\dot \bx(t)\cdot \bn(t)) \ul{n}(t), \qquad \forall t \in [0,1],
\end{equation}
that arises from (\ref{eqdef:ReedSheppMetric}). For details on the case $d = 2$ see \cite{boscain_curve_2014,sachkov_cut_2011}, for $d =3$ see \cite{duits_sub-riemannian_2016}.

Finally, we note that for the uniform cost case ($\xi^{-1}\mathcal{C}_{1}=\mathcal{C}_{2}=1$), the problem is covariant with respect to rotations and translations. For the data-driven case, such covariance is only obtained when simultaneously rotating the data-driven cost factors  $\mathcal{C}_{1}, \mathcal{C}_{2}$. Therefore, only in the uniform cost case, for $d=2,3$, we shall use a reference point (`the origin') $\ul{e} \in \R^{d} \times \bS^{d-1}$.
To adhere to common conventions we use
\begin{equation}\label{origin}
	\begin{aligned}
		\ul{e}=(\ul{0},\ul{a}) &\in \R^{d}\times \bS^{d-1}, \textrm{ with } \\ & \ul{a}:=(1,0)^T \qquad \textrm{ if } d=2
		\textrm{ and } \\ &\ul{a}:=(0,0,1)^T \quad \textrm{ if }d=3.
	\end{aligned}
\end{equation}

\subsection{Controllability of the Reeds-Shepp model}\label{subsec:Geometry}
A model $(\bbM, d_\cF)$ is \emph{globally controllable} if the distance $d_\cF$ takes finite values on $\bbM \times \bbM$, in other words, a car can go from any place on the manifold to any other place in finite time \todo{R1.4: context}. In Theorem \ref{th:Controllability} we show that this is indeed the case for $\cF = \cF_0$ and $\cF = \cF_0^+$, given in \eqref{eqdef:ReedSheppMetric} and \eqref{eqdef:ReedSheppForwardMetric}. \emph{Local controllability} is satisfied when $d_\cF$ satisfies a certain continuity requirement: if $\bp \to \bq \in (\mathbb{M},\|\cdot\|)$, with $\|\cdot\|$ denoting the standard (flat) Euclidean norm on $\mathbb{M}=\R^{d} \times \mathbb{S}^{d-1}$, we must have $d_{\mathcal{F}}(\bp,\bq) \to 0$. We prove in Theorem \ref{th:Controllability} that the metric space $(\bbM, d_{\cF_0})$ is locally controllable\todo{E.4: typo}, but the quasi-metric space $(\bbM,d_{\cF_0^+})$ is not. 
Indeed the SR Reeds-Shepp car can achieve sideways motions by alternating the forward and reverse gear with slight direction changes, whereas the model without reverse gear lacks this possibility. For completeness, the theorem contains a standard (rough) estimate of the distance near the source (due to well-known estimates \cite{gromov_carnot-caratheodory_1996,terelst,citti_cortical_2006,jorg}).

Furthermore, we prove existence of minimizers for the Reeds-Shepp model without reverse gear. Existence results of minimizers of the model with reverse gear (the SR model) already exist, by the Chow-Rashevski theorem and Fillipov theorems \cite{agrachev_control_2004}.
\begin{theorem}[(Local) controllability properties]
\label{th:Controllability}
Minimizers exist for both the classical Reeds-Shepp model, and for the Reeds-Shepp model without reverse gear.
Both models are globally controllable.
\begin{itemize}
	\item  The Reeds-Shepp model without reverse gear is not locally controllable, since
\begin{equation} \label{estimate}
\limsup_{\bp' \to \bp} d_{\cF^+_0}(\bp,\bp')\geq 2 \pi \delta, \textrm{ for all }\bp\in \mathbb{M}.
\end{equation}
If the cost $\cC_{2}=\delta$ is constant on $\bM$, then this inequality is sharp:
	\begin{equation} \label{equality}
		\limsup_{\bp' \to \bp} d_{\cF^+_0}(\bp,\bp') = \lim_{\mu\downarrow 0} d_{\cF_0^+} ((\bx,\bn),\, (\bx-\mu\bn,\bn)) = 2 \pi \delta.
	\end{equation}
\item The sub-Riemannian Reeds-Shepp model is locally controllable, since 
	\begin{align}\label{approx}
		\hspace*{-1.5em}d_{\cF_0}(\bp,\bp') &= \cO \left( \cC_{2}(\bp) \|\bn-\bn'\|+ \sqrt{\cC_{2}(\bp)\cC_{1}(\bp)\|\bx-\bx'\|}\right) \ \nonumber \\
		&\hspace{1em}\text{ as } \bp'=(\bx',\bn') \to \bp=(\bx,\bn).
	\end{align}
	\end{itemize}
\end{theorem}
For a proof see Section~\ref{ch:3}.

\subsection{A Continuous Approximation for the Reeds-Shepp geometry}
\label{subsec:ApproximateReedShepp}

We introduce approximations $\cF_\ve$ and $\cF_\ve^+$ of the Finsler metrics $\cF_0$ and $\cF_0^+$, depending on a small parameter $0< \ve \leq 1$, which are continuous and in particular take only finite values. This is a prerequisite for our numerical methods. Both approximations penalize the deviation from the constraints of collinearity $\dot \bx \propto \bn$, and in addition, $\cF_\ve^+$ penalizes negativity of the scalar product $\dot \bx \cdot \bn$, appearing in \eqref{eqdef:ReedSheppMetric} and \eqref{eqdef:ReedSheppForwardMetric}. For that purpose, we introduce some additional notation: for $\dot \bx \in \bR^d$ and $\bn\in \bS^{d-1}$ we define
\begin{align}\label{eq:wedgeminmax}
	\| \dot \bx \wedge \bn \|^2 &:= \| \dot \bx \|^2 - |\dot \bx \cdot \bn|^2,  \\
	(\dot \bx \cdot \bn)_- &:= \min \{0,  \dot \bx \cdot \bn\}, \quad (\dot \bx \cdot \bn)_+ := \max\{\dot \bx \cdot \bn,0\}.\nonumber
\end{align}
These are respectively the norm of the orthogonal projection\footnote{The quantity $\| \dot \bx \wedge \bn \|$ is also the norm of the wedge product of $\dot \bx$ and $\bn$, but defining it this way would require introducing  some algebra which is not needed in the rest of this paper.} of $\dot \bx$ onto the plane orthogonal to $\bn$, and the negative and positive parts of their scalar product.
The two metrics $\cF_\ve,\cF_\ve^+ : T(\bM) \to \bR_+$ are defined for each $0<\ve \leq 1$, as follows: for $(\bp, \dot \bp) \in T(\bM)$ with components $\bp = (\bx,\bn)$ and $\dot \bp = (\dot \bx,\dot\bn)$ we define
\begin{alignat}{2} \label{importantFs1}
	\cF_\ve(\bp, \dot \bp)^2 &:=  &&\cC_{1}(\bp)^2 ( |\dot \bx \cdot \bn|^2 + \ve^{-2}\|\dot \bx \wedge \bn\|^2) + \nonumber \\ &&& \cC_{2}(\bp)^2 \| \dot \bn\|^2,\\[3mm]
	\cF^+_\ve(\bp, \dot \bp)^2 &:= &&\, \cC_{1}(\bp)^2 ( |\dot \bx \cdot \bn|^2 + \ve^{-2}\|\dot \bx \wedge \bn\|^2 + \nonumber \\ &&& (\ve^{-2}-1) (\dot \bx \cdot \bn)_-^2) +\, \cC_{2}(\bp)^2 \| \dot \bn\|^2 \label{importantFs2} \\[6pt]
 &= &&\cC_{1}(\bp)^2 ( (\dot \bx \cdot \bn)_+^2 + \ve^{-2}\|\dot \bx \wedge \bn\|^2 + \nonumber \\ &&& \ve^{-2} (\dot \bx \cdot \bn)_-^2) \, +\, \cC_{2}(\bp)^2 \| \dot \bn\|^2.\label{importantFs3}
\end{alignat}

See Fig. \ref{fig:approximatemetric} for a visualization of a level set of both metrics in $\R^2 \times \bS^1$. Note that $\cF_\ve$ is a Riemannian metric on $\bM$ (with the same smoothness as the cost functions $\cC_{2}, \cC_{1}$), and that $\cF_\ve^+$ is neither Riemannian nor smooth due to the term $(\dot \bx \cdot \bn)_-$. One clearly has the pointwise convergence $\cF_\ve(\bp, \dot \bp) \to \cF_0(\bp, \dot \bp)$ as $\ve\to 0$, and likewise $\cF^+_\ve(\bp, \dot \bp) \to \cF^+_0(\bp, \dot \bp)$. The use of $\cF_\ve$ and $\cF_\ve^+$ is further justified by the following convergence result.

\begin{figure}[t]
\centering
\includegraphics[width=0.9 \hsize]{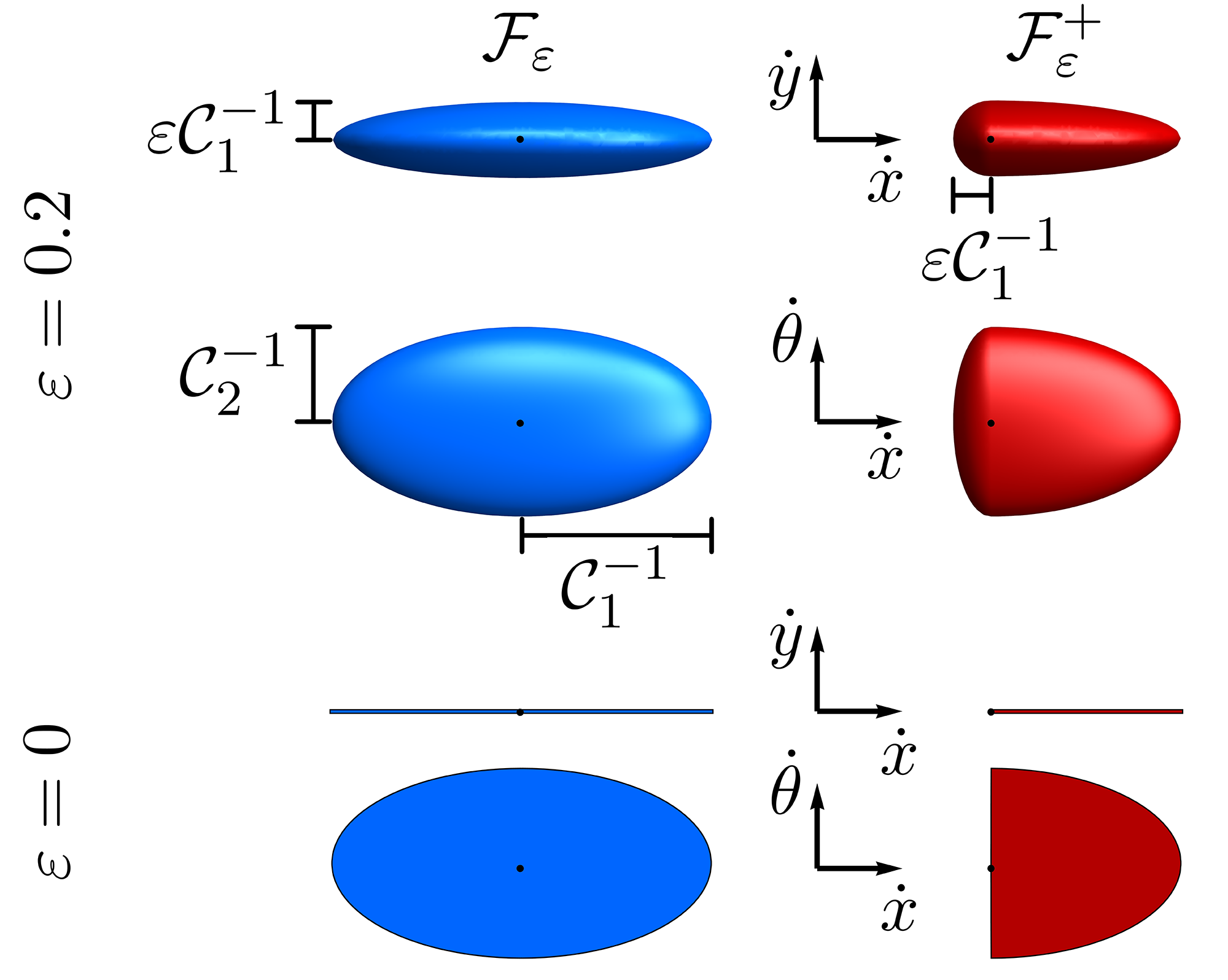}
\caption{Levelsets for $d = 2$ of the (approximating) metrics $\cF_\ve(\mathbf{0},(\dot{x},\dot{y},\dot{\theta})) = 1$ (left) and $\cF^+_\ve(\mathbf{0},(\dot{x},\dot{y},\dot{\theta})) = 1$ (right), with $\ve = 0.2$ (top) and $\ve = 0$ (bottom). In this example, $\cC_2(\mathbf{0}) = 2 \cC_1(\mathbf{0})$.}\label{fig:approximatemetric}
\end{figure}

\begin{theorem}[Convergence of the Approximative Models to the Exact Models]
\label{th:ReedSheppCV}
One has the pointwise convergence: for any $\bp,\bq\in \bM$
\begin{equation*}
	\begin{split}
		d_{\cF_\ve}(\bp,\bq) &\to d_{\cF_0}(\bp,\bq), \\
		d_{\cF_\ve^+}(\bp,\bq) &\to d_{\cF_0^+}(\bp,\bq),
	\end{split}
	\quad \text{as } \ve \to 0.
\end{equation*}

Consider for each $\ve>0$ a minimizing path $\gamma_\ve^*$ from $\bp$ to $\bq$, with respect to the metric $\cF_\ve$, parametrized at constant speed
\begin{equation*}
\mathcal{F}_{\ve}(\gamma^*_{\ve}(t),\dot{\gamma}_{\ve}^*(t))= d_{\mathcal{F}_{\ve}}(\bp,\bq), \qquad \forall t \in [0,1].
\end{equation*}
Assume that there is a unique shortest \todo{R5.1: terminology} path $\gamma^*$ from $\bp$ to $\bq$ with respect to the sub-Riemannian distance $d_{\mathcal{F}_{0}}$ (in other words $\bq$ is not within the cut locus of $\bp$), parametrized at constant speed:
\[
\mathcal{F}_{0}(\gamma^*(t),\dot{\gamma}^*(t))= d_{\mathcal{F}_{0}}(\bp,\bq), \qquad \forall t \in [0,1].
\]
Then $\gamma_{\ve}^* \to \gamma^*$ as $\ve\to 0$, uniformly on $[0,1]$. \
Likewise replacing $\mathcal{F}_{\ve}$ with $\mathcal{F}_{\ve}^+$ for all $\ve \geq 0$.
\end{theorem}
The proof, presented in Appendix \ref{app:WellPosedness} is based on a general result originally applied to the Euler elastica curves in \cite{DaChen2016Thesis}.
Combining Theorem \ref{th:ReedSheppCV} with the local controllability properties established in Theorem \ref{th:Controllability}, one obtains that $d_{\cF_\ve} \to d_{\cF_0}$ locally uniformly on $\bM \times \bM$, and that the convergence $d_{\cF_\ve^+} \to d_{\cF_0^+}$ is only pointwise.

\begin{remark}
If there exists a family of minimizing \todo{R5.1: terminology} geodesics $(\gamma_i^*)_{i \in I}$ from $\bp$ to $\bq$ with respect to $\cF_0$ (resp.\ $\cF_0^+$), then one can show that for any sequence $\ve_n \to 0$ one can find a subsequence and an index $i \in I$ such that $\gamma_{\ve_{\vp(n)}}^* \to \gamma_i^*$ uniformly as $n\to \infty$.
\end{remark}
\begin{figure*}
\centerline{
\includegraphics[width=0.45 \hsize]{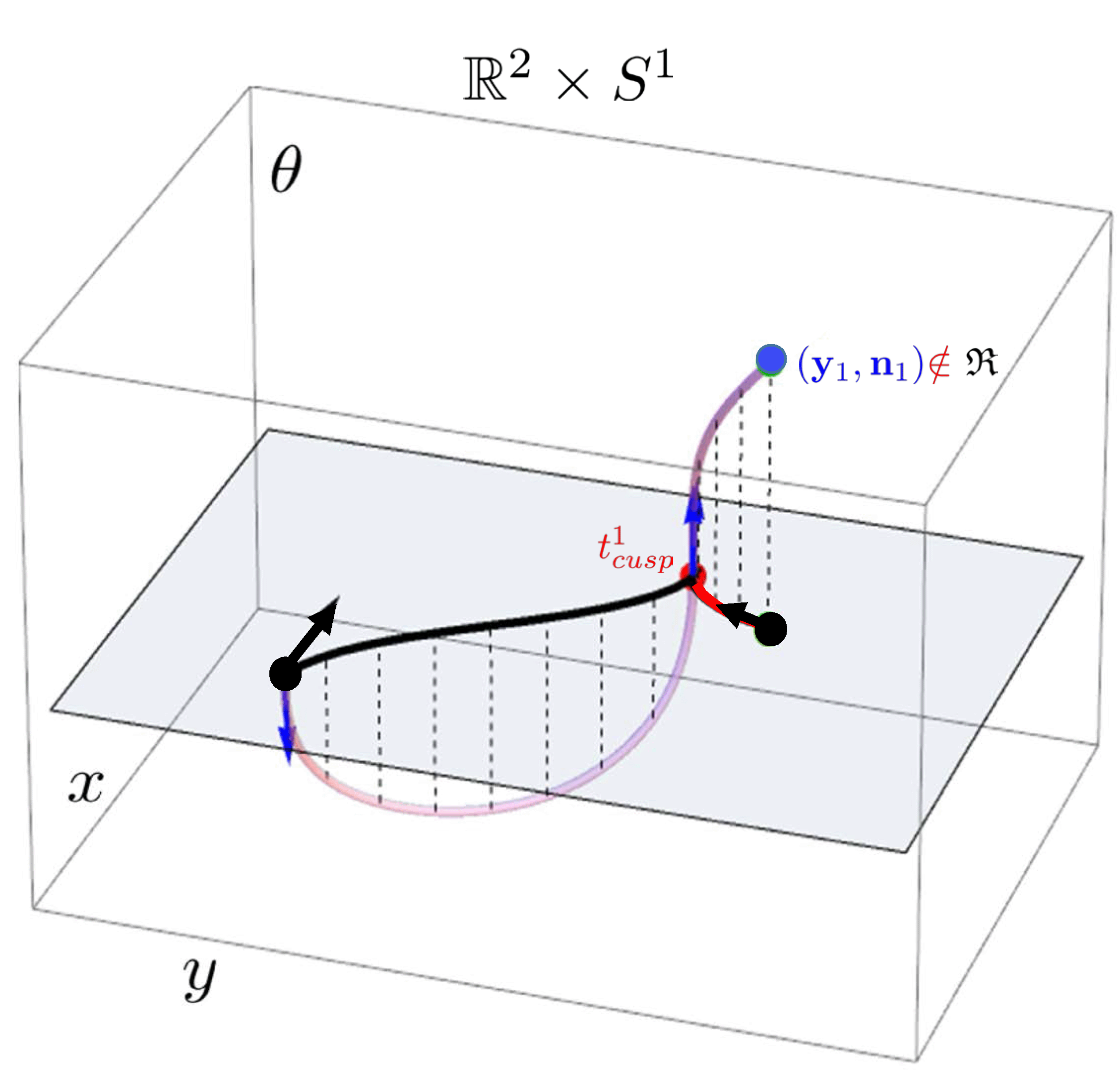}
\includegraphics[width=0.55 \hsize]{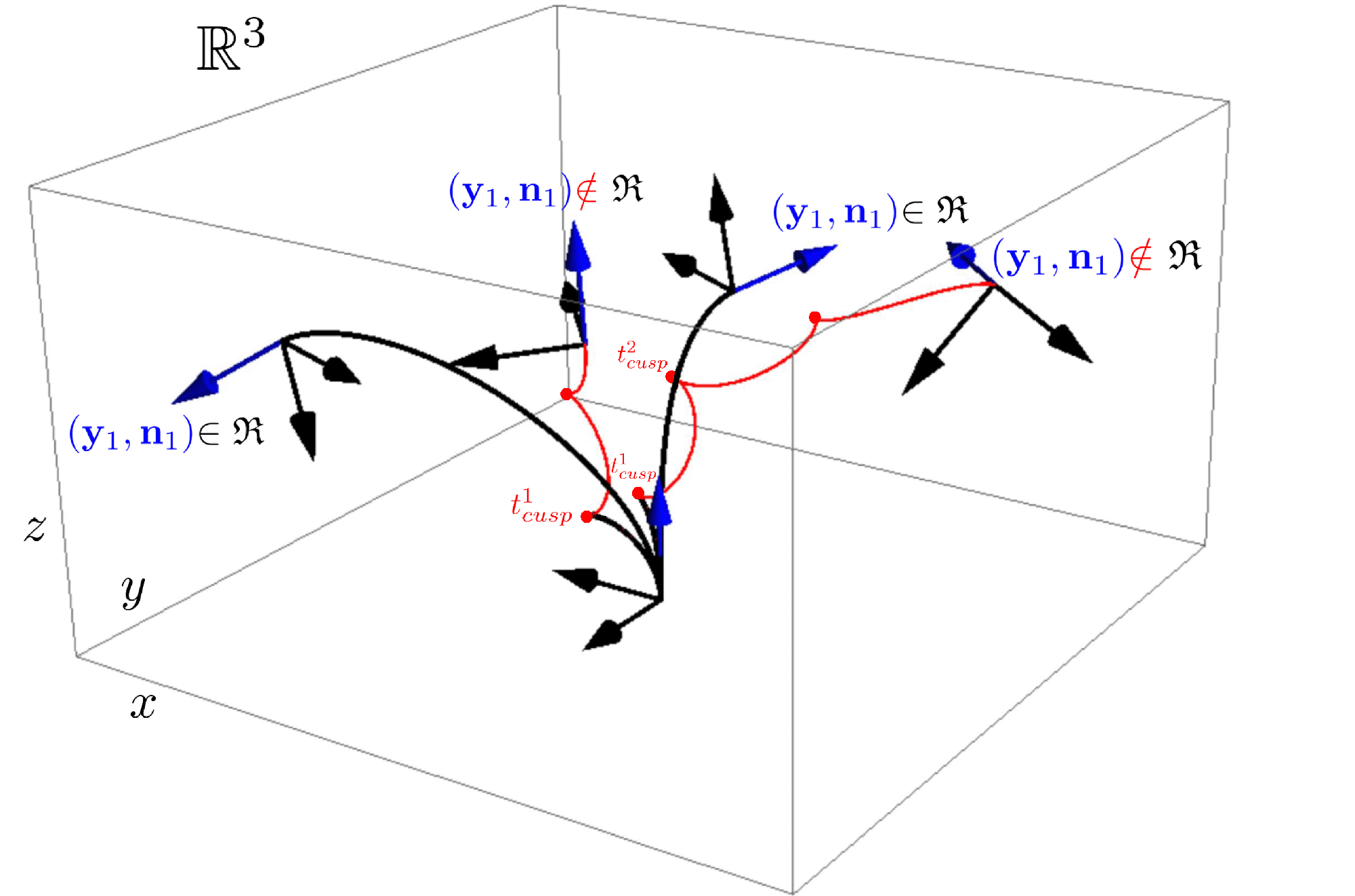}
}
\caption{Illustration of cusps in SR ($\ve=0$) geodesics (possibly non-optimal) in $\mathbb{M}=\mathbb{R}^{d} \times \bS^{d-1}$. Left: cusps in spatial projections $\ul{x}(\cdot)$ of SR geodesics $\bgamma(\cdot)=(\ul{x}(\cdot),\ul{n}(\cdot))$ for $d=2$, right: cusps (red dots) appearing in spatial projections of SR geodesics for $d=3$. In the 3D case we indicate the corresponding rotations $\mathbf{R}_{\ul{n}_1}$ via a local 3D frame.\label{fig:cusp}}
\end{figure*}

\subsection{Points of Interest in Spatial Projections of Geodesics for the Uniform Cost Case: \mbox{Cusps vs. Keypoints} \label{subsec:interestpoints}}

Next we provide a theorem that tells us in each of the models/metric spaces $(\mathbb{M},d_{\mathcal{F}_{0}})$,  $(\mathbb{M},d_{\mathcal{F}_{\ve}})$ and $(\mathbb{M},d_{\mathcal{F}_{0}^+})$,  $(\mathbb{M},d_{\mathcal{F}_{\ve}^+})$,
with $\cC_1=\cC_2=1$ and $d=2$ where cusps occur in spatial projections of geodesics or where keypoints with in-place rotations take place.

Note that for vessel-tracking applications, cusps are not wanted (there is no reason why the entering angle should be the same as the departing angle), whereas keypoints are only desirable at bifurcations\todo{R1.4:cusps/keypoints}.
\begin{definition}[Cusp]\label{def:cusp}
A cusp point $\ul{x}(t_0)$ on a spatial projection of a (SR) geodesic $ t \mapsto (\ul{x}(t),\ul{n}(t))$ in $\mathbb{M}$ is
a point where
\begin{equation}\label{def:utilde}
\begin{array}{l}
\tilde{u} (t_0)=0, \textrm{ and }\ \dot{\tilde{u}}(t_0) \neq 0, \\[5pt]
\text{where } \tilde{u}(t):= \ul{n}(t) \cdot \dot{\ul{x}}(t) \textrm{ for all }t.
\end{array}
\end{equation}
I.e. a cusp point is a point where the spatial control aligned with $\ul{n}(t_0)$ vanishes and switches sign locally.
\end{definition}
Although this definition explains the notion of a cusp geometrically (as can be observed in Fig.~\ref{fig:freevspositive_introfig_full} and Fig.~\ref{fig:cusp}), it contains a redundant part for the relevant case of interest:
the second condition automatically follows when considering the SR geodesics in $(\mathbb{M}, d_{\mathcal{F}_{0}})$. The following lemma gives a characterization of a cusp point in terms of the distance function along a curve.
\begin{lemma}\label{lem:cusp}
Consider a SR geodesic $\gamma=(\ul{x},\ul{n}) : [0,1] \to (\mathbb{M}, d_{\mathcal{F}_{0}})$, parametrized at constant speed, and which physical position $\bx(\cdot)$ is not identically constant. Denote $\bp_S := \gamma(0)$ and $U(\cdot):= d_{\mathcal{F}_{0}}(\bp_{S},\cdot)$. Let $t_0\in (0,1)$ be such that $U$ is differentiable at $\gamma(t_0) =(\ul{x}(t_0),\ul{n}(t_0))$. Then
\begin{equation}
\begin{split}
\bx(t_0) \textrm{ is a cusp point }\desda \ul{n}(t_0) \cdot \dot{\ul{x}}(t_0) =0 \\
\desda \ul{n}(t_0) \cdot \nabla_{\R^{d}} U(\ul{x}(t_0),\ul{n}(t_0))=0.
\end{split}
\end{equation}
\end{lemma}
The proof can be found in Appendix~\ref{app:cusp}.

\begin{definition}[Keypoint]\label{def:keypoint}
A point $\tilde{\bx}$ on the spatial projection of a geodesic $\gamma(\cdot) = (\bx(\cdot),\bn(\cdot))$ in $\bbM$ is a keypoint of $\gamma$ if there exist $t_0 < t_1$, such that $\bx(t) = \tilde{\bx}$ and $\dot{\bn}(t) \neq 0$ for all $t \in [t_0,t_1]$, i.e., a point where an in-place rotation takes place.
\end{definition}

\begin{definition}\label{def:cusplessset}
 We define the set $\gothic{R} \subset \bbM$ to be all endpoints that can be reached with a geodesic $\gamma^*:[0,1] \to \mathbb{M}$ in
 $(\mathbb{M},d_{\mathcal{F}_0})$ whose spatial control $\tilde{u}(t)$ stays positive for all $t \in [0,1]$.
\end{definition}

\begin{remark}\label{rem:globalmingeodesic}
The word `geodesic' in this definition can (in the case $d = 2$) be replaced by `globally minimizing geodesic' \cite{boscain_curve_2014}. For a definition in terms of the exponential map of a geometrical control problem $\ul{P}_{curve}$, see e.g. \cite{duits_association_2013,duits_cuspless_2014}, in which the same positivity condition for $\tilde{u}$ is imposed. Fig.~\ref{fig:R} shows more precisely what this set looks like for $d = 2$ \cite{duits_association_2013}, in particular that it is contained in the half-space $\ul{a} \cdot \ul{x} \geq 0$, and for $d = 3$ \cite{duits_cuspless_2014}. We extend these results with the following theorem.
\end{remark}

\begin{figure*}
\includegraphics[width=\hsize]{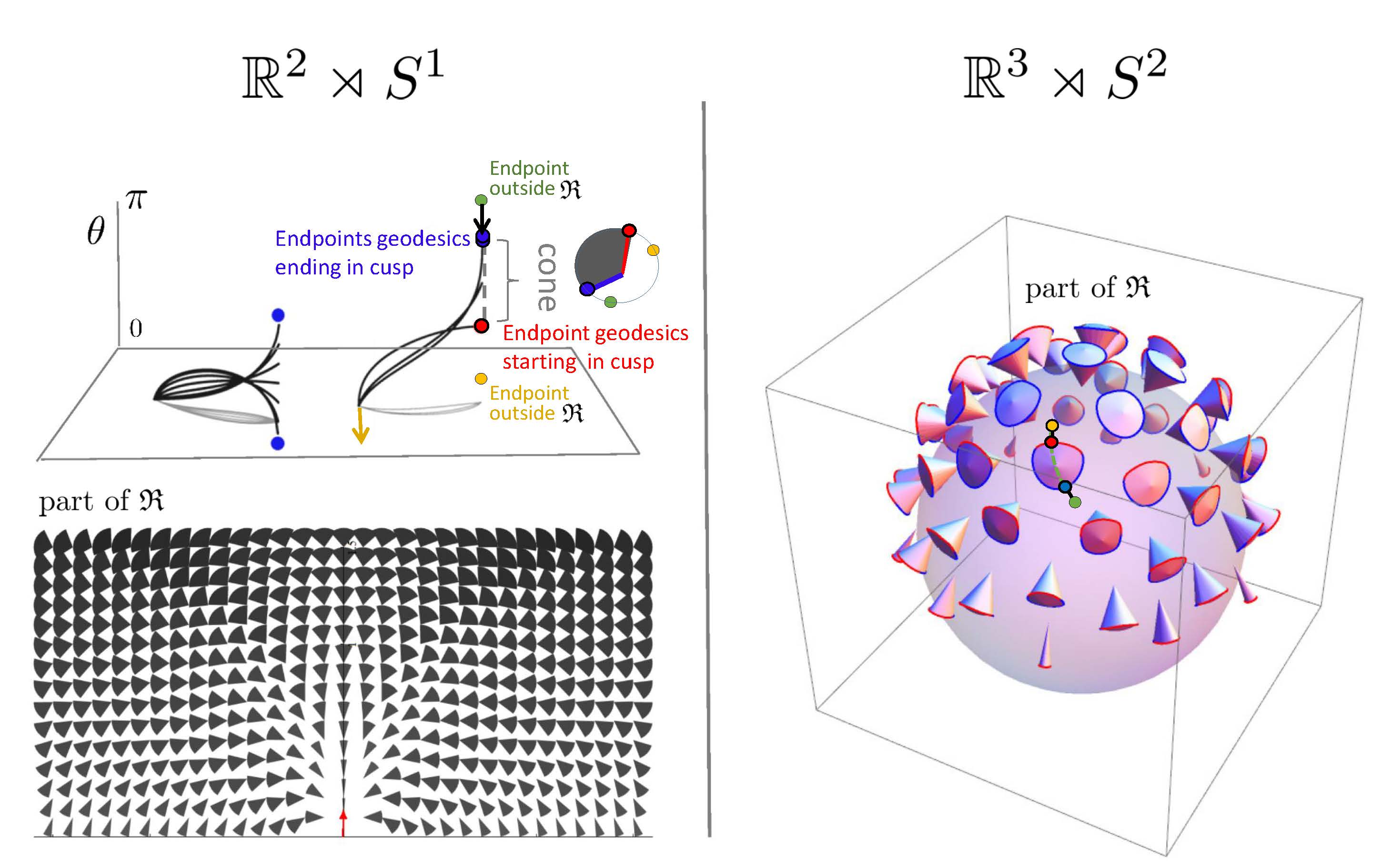}
\caption{The set $\gothic{R}$ of endpoints reachable from the origin $\ul{e}$ (recall (\ref{origin})) via SR geodesics whose spatial projections do not exhibit cusps has been studied for the case $d=2$ (left), and for the case $d=3$ (right). For $d=2$ it is contained in $x \geq 0$ and for $d=3$ it is contained in $z\geq 0$. The boundary of this set contains of endpoints of geodesics departing at a cusp (in red) or of endpoints of geodesics ending in a cusp (in blue).
If an endpoint $(\ul{x},\ul{n})$ is placed outside $\gothic{R}$ (e.g. the green points above) then following the approach in Theorem~\ref{th:Backtracing},
depending on its initial spatial location it first connects to a blue point $(\ul{x},\ul{n}_{new})$ via a spherical geodesic end then connects to the origin $\ul{e}$ via a SR geodesic. Then it has a keypoint at the endpoint. For other locations spatial locations (orange points), the geodesic has the keypoint in the origin, or even at both boundaries, cf.~Fig.~\!\ref{fig:JorgSE2}. \label{fig:R}}
\end{figure*}
\begin{theorem}[Cusps and Keypoints]
\label{th:CuspsAndRotations}
Let $\ve>0$, $d=2$, $\cC_{1}=\cC_{2}=1$.
 Then,
\begin{itemize}
\item in $(\mathbb{M},d_{\mathcal{F}_{0}})$ cusps are present in spatial projections of almost every optimal SR geodesics when their times $t$ are extended on the real line (until they lose optimality) \todo{R.5.1 Terminology}. The straight-lines connecting specific boundary points $\ul{p}=(\ul{x},\ul{n})$ and $\ul{q}=(\ul{x}+ \lambda \ul{n},\ul{n})$ with $\lambda \in \R$ are the only exceptions.
\item  in $(\mathbb{M},d_{\mathcal{F}_{\ve}^+})$  and $(\mathbb{M},d_{\mathcal{F}_{\ve}})$ and $(\mathbb{M},d_{\mathcal{F}_{0}^+})$ no cusps appear in spatial projections of geodesics.
\end{itemize}
Furthermore,
\begin{itemize}
\item in $(\mathbb{M},d_{\mathcal{F}_{0}})$, $(\mathbb{M},d_{\mathcal{F}_{\ve}})$ and $(\mathbb{M},d_{\mathcal{F}_{\ve}^+})$ keypoints only occur with vertical geodesics (moving only angularly).
\item  in $(\mathbb{M},d_{\mathcal{F}_{0}^+})$ keypoints only occur at the endpoints of shortest \todo{R5.1: terminology} paths.
\end{itemize}
A minimizing geodesic $\gamma_{+}$ in $(\mathbb{M},d_{\mathcal{F}_{0}^+})$ departing from $\ul{e}=(0,0,0)$ and ending in $\bp=(x,y,\theta)$ has
\begin{itemize}
\item[A)] no keypoint if $\bp \in \overline{\gothic{R}}$,
\item[B)] a keypoint in $(0,0)$ if $x<0$,
\item[C)] a keypoint \emph{only} in $(x,y)$ if\footnote{Here $\overline{\gothic{R}}^{c}=\mathbb{M} \setminus \overline{\gothic{R}}$ denotes the complement of the closure $\overline{\gothic{R}}$ of $\gothic{R}$, and $E(z,m)=\int_{0}^z \sqrt{1- m \sin^2v}\, {\rm d}v$.}
\begin{enumerate}
\item[C1)] $\bp \in \overline{\gothic{R}}^{c}$ and $x \geq 2$,
\item[C2)] $\bp \in \overline{\gothic{R}}^{c}$ and $0 \leq x < 2$ and \\ $|y| \leq -i x \, E \left(i \textrm{arcsinh}\left(\frac{x}{\sqrt{4-x^2}}\right), \frac{x^2-4}{x^2}  \right)$, where $E(z,m)$ denotes the Elliptic integral of the second kind.
    \end{enumerate}
\end{itemize}

\end{theorem}

\begin{figure*}
\centerline{
\includegraphics[width=\hsize]{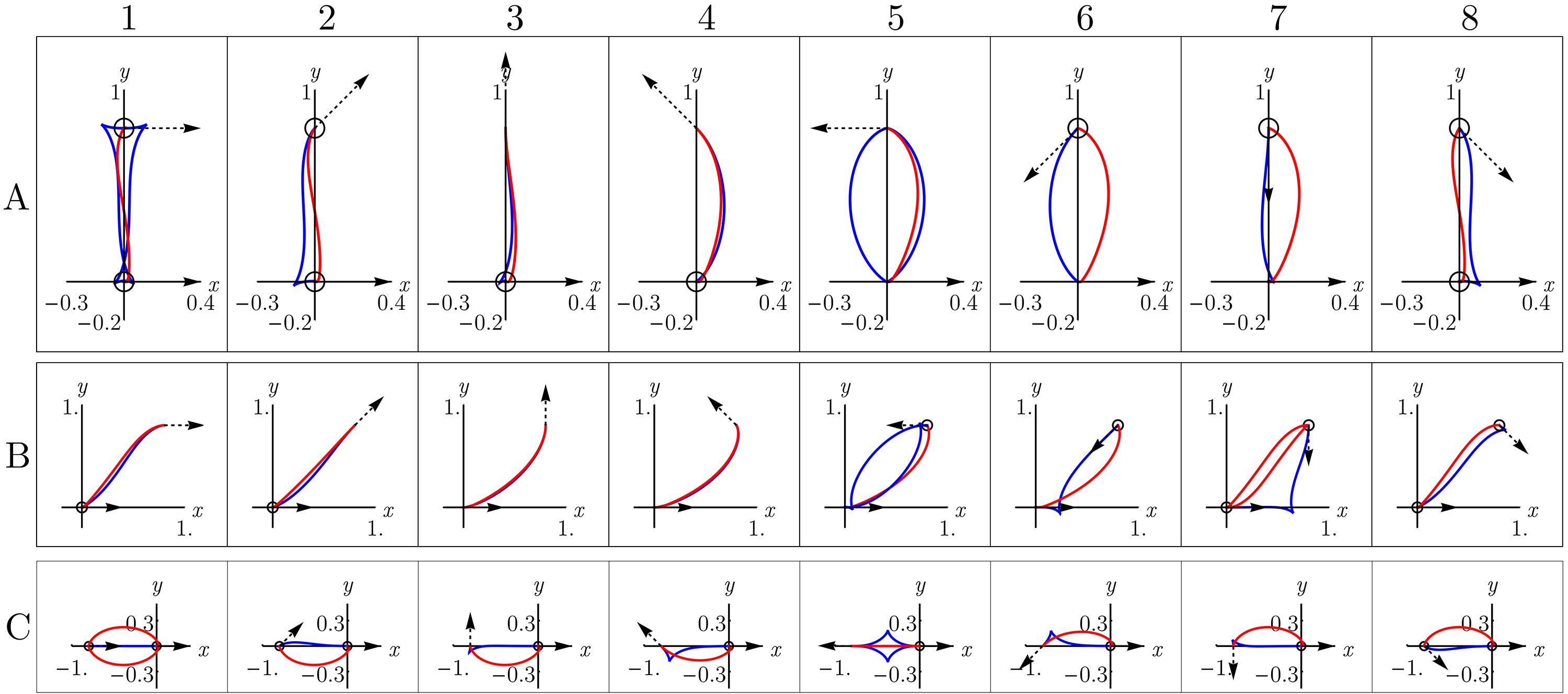}
}
\caption{Shortest paths for $d=2$ using the Finsler metrics $\cF_{0}$ (blue) and $\cF_{0}^+$ (red),
with point source $\bp_S = (0,0,0)$ and varying end conditions.
Row A: $\bp = (0,0.8,\pi n /4)$. Row B: $\bp = (0.8,0.8,\pi n/4)$. Row C: $\bp = (-0.8,0,\pi n/4)$. Here $n = 1, \dots, 8$, corresponding to the columns. When there are two minimizing geodesics, both are drawn. Circles around the begin or end point indicate in-place rotation of the red curve at that point. We see that whenever the blue geodesic has a cusp, the red geodesic has at least one in-place rotation (keypoint). This numerically supports our statements in Theorem~\ref{th:CuspsAndRotations} considering cusps and keypoints. For high accuracy we applied the relatively slow iterative PDE approach \cite{bekkers_pde_2015}
on a $101 \times 101 \times 64$-grid in $\mathbb{M}$ to compute $d_{\cF_{0}}(\bp,\bp_S)$ and $d_{\cF_{0}^+}(\bp,\bp_S)$, see App. \ref{app:iterative}.
\label{fig:JorgSE2}}
\end{figure*}

\begin{remark}
In case A, $\gamma_{+}$ is a minimizing geodesic in $(\mathbb{M},d_{\mathcal{F}_0})$ as well. In case B, $\gamma_{+}$ departs from a cusp. In case C, $\gamma^+$ is a concatenation of a minimizing geodesic in $(\mathbb{M},d_{\mathcal{F}_0})$ and an in-place rotation. For other endpoints $(x,y,\theta)$ for geodesics departing from $\ul{e}$ with $0\leq x < 2$, other than the ones reported in C2 it is not immediately clear what happens, due to \cite[Thm.9]{duits_association_2013}. Also points with $x<0$ may have keypoints at the end as well. See Fig.~\ref{fig:JorgSE2} where various
cases of minimizing geodesics in $(\mathbb{M},d_{\mathcal{F}_0^+})$ are depicted.
\end{remark}


\begin{remark}
It is also interesting to study the effect of $\ve \geq 0$ on the removal (or rather smoothing out in practice) of cusps on \emph{non-optimal} geodesics in $(\mathbb{M},d_{\mathcal{F}_{\ve}})$ and keypoints in $(\mathbb{M},d_{\mathcal{F}_{\ve}^{+}})$ when $\ve$ moves away from $0$.
See Fig.~\ref{fig:smoothout}, where such non-optimal geodesics are obtained via Euler-Lagrange formalism (or equivalently by integration of the canonical equations in the Pontryagin Maximum principle).
\end{remark}

\begin{figure*}
\includegraphics[width=\hsize]{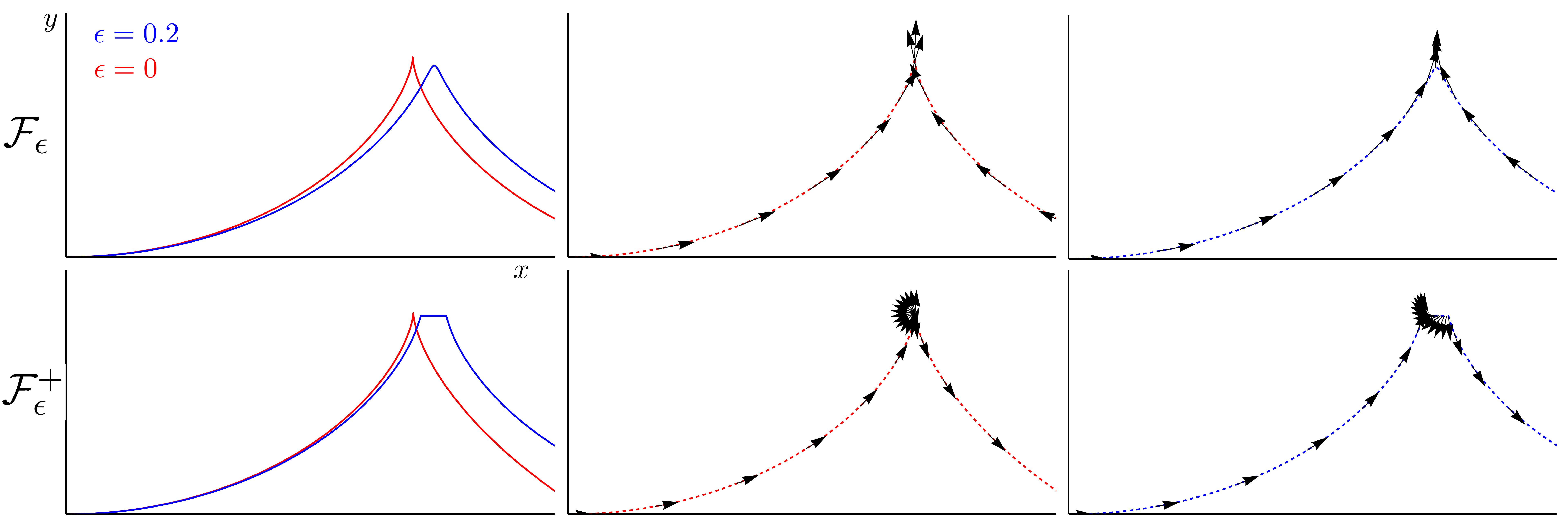}

\caption{Non-optimal geodesics in the special case $\mathcal{C}=1$ and $d=2$. In particular in the case $\cF_0$, keypoints do not appear in the interior of globally optimal curves, only at the end, cf.~Theorem~\ref{th:Controllability}.
We also observe that cusps disappear when $\ve>0$, cf.~Theorem~\ref{th:CuspsAndRotations}. The red curves are spatially very similar, but the part after the cusp/keypoint is traversed with a different orientation. Similar behavior can be observed for the blue curves.
}
\label{fig:smoothout}
\end{figure*}

\subsection{The Eikonal PDE Formalism \label{subsec:eikonal}}

As briefly discussed in Section \ref{sec:introeikonal}, continuous metrics like $\cF_\ve$ and $\cF_\ve^+$ for any $\ve>0$, allow to use the standard theory of viscosity solutions of eikonal PDEs, and thus to design provable and efficient numerical schemes for the computation of distance maps and minimizing \todo{R5.1: terminology} geodesics.
More precisely, consider a continuous Finsler metric $\cF \in C^0(T(\bM),\R^+)$, and define the dual $\cF^*$ on the co-tangent bundle as follows: for all $(\bp, \cot \bp) \in T^*(\bM)$
\begin{equation} \label{dualFinslerfunction}
	\cF^*(\bp, \cot \bp) := \sup_{\dot \bp \in T_\bp \bM\sm \{0\}} \frac{\<\cot \bp, \dot \bp\>}{\cF(\bp, \dot \bp)}.
\end{equation}
The distance map $U = d_\cF(\pSource, \cdot)$ from a given source point $\pSource \in \bM$ is the unique solution, in the sense of viscosity solutions, of the static Hamilton Jacobi equation: $U(\pSource)=0$, and for all $p \in \bM$
\begin{equation}
\label{eq:HJBGeneral}
	\cF^*(\bp, \diff U(\bp))=1.
\end{equation}
Furthermore, if $\gamma$ is a minimizing \todo{R5.1: terminology} geodesic from $\pSource$ to some $\bp \in \bM$, then it obeys the ordinary differential equation (ODE):
\begin{equation}
\label{eq:GeodesicODE}
\left\{\begin{aligned}
	&\dot \gamma(t) = L \ \diff_{\cot \bp} \cF^*(\gamma(t), \diff U(\gamma(t))), \; L := d_\cF(\pSource, \bp) \\
	&\gamma(0) = \bp_S, \quad \gamma(1) = \bp.
	\end{aligned}\right.
\end{equation}
for any $t\in [0,1]$ such that the differentiability of $U$ and $\cF^*$ holds at the required points. The proof of the ODE \eqref{eq:GeodesicODE} is for completeness derived in Proposition \ref{prop:GeodesicODE} of Appendix~\ref{app:Backtracing}, where we also discuss in Remark \ref{rem:LagrangianHamiltonian} the common alternative formalism based on the Hamiltonian.
We denoted by $\diff_{\cot \bp} \cF^*$ the differential of the dual Finsler 
metric $\cF^*$ with respect to the second variable $\cot \bp$, hence $\diff_{\cot \bp} \cF^*(\bp, \cot \bp) \in T^{**}_\bp(\bM) \cong T_\bp(\bM)$ is indeed a tangent vector to $\bM$, for all $(\bp,\cot \bp)\in T^*\bM$.

In the rest of this section, we specialize \eqref{eq:HJBGeneral} and \eqref{eq:GeodesicODE} to the Finsler 
metrics $\cF_\ve$ and $\cF_\ve^+$.
Our first result provides explicit expressions for the dual Finsler 
metrics (required for the eikonal equation).
\begin{proposition}
\label{prop:DualMetric}
	For any $0<\ve\leq 1$, the duals to the approximating Finsler metrics $\cF_\ve$ and $\cF_\ve^+$ are: for all $(\bp,\cot\bp) \in T^*(\bM)$, with $\bp = (\bx,\bn)$ and $\cot \bp = (\cot\bx, \cot\bn)$
	\begin{equation}\label{eq:FEpsPlusStar}
	\begin{split}
		\cF_\ve^*(\bp, \cot\bp)^2 &= (\cC_{2}(\bp))^{-2} \|\cot\bn\|^2 + (\cC_{1}(\bp))^{-2} (|\cot \bx\cdot\bn|^2 \\
			&\hspace*{4em} + \ve^2 \|\cot \bx\wedge \bn\|^2) \\
		\cF_\ve^{+*}(\bp, \cot\bp)^2  &= (\cC_{2}(\bp))^{-2} \|\cot\bn\|^2 + (\cC_{1}(\bp))^{-2} (|\cot \bx\cdot\bn|^2  \\
		 	&\hspace*{4em} + \ve^2 \|\cot \bx\wedge \bn\|^2  - (1-\ve^2) (\cot \bx \cdot \bn)_-^2) \\
		&= (\cC_{2}(\bp))^{-2} \|\cot\bn\|^2 + (\cC_{1}(\bp))^{-2} ((\cot \bx\cdot\bn)_+^2  \\
			&\hspace*{4em} + \ve^2(\cot \bx \cdot \bn)_-^2 + \ve^2 \|\cot \bx\wedge \bn\|^2)
	\end{split}
	\end{equation}
\end{proposition}

In order to relate the Finslerian HJB equation \eqref{eq:HJBGeneral} and backtracking equation \eqref{eq:GeodesicODE} to some more classical Riemannian counterparts, we introduce two Riemannian metric tensor fields on $\bM$.
The first is defined as the polarization of the norm $\cF_\ve(\bp,\cdot)$ 
\begin{equation} \label{importantmetric}
\begin{split}
\mathcal{G}_{\bp ; \ve}(\dot{\bp},\dot{\bp}) &= |\mathcal{F}_{\ve}(\bp,\dot{\bp})|^2  \\&= \mathcal{C}_{1}^{2}(\bp) ((\dot \bx \cdot \bn)^2+\ve^{-2}\|\dot\bx \wedge \bn\|^2) \\
&\hspace*{7.35em}+ \mathcal{C}_{2}^2(\bp) \|\dot{\ul{n}}\|^2\ ,
\end{split}
\end{equation}
where $\dot{\bp}=(\dot{\ul{x}},\dot{\ul{n}})$,
and then one can also rely on gradient fields $\bp \mapsto \mathcal{G}_{\bp; \ve}^{-1}{\rm d}U(\bp)$ relative to this metric tensor. This has benefits if it comes to geometric understanding of the eikonal equation and its tracking. Even in the analysis of the non-symmetric case --where one does not have a single metric tensor--
this notion plays a role, as we will see in the next main theorem.
To this end, in the non-symmetric case, we shall rely on a second spatially isotropic metric tensor given by:
\begin{equation} \label{Gtilde}
\widetilde{\mathcal{G}}_{\bp; \ve}(\dot{\bp},\dot{\bp}):= \mathcal{C}_{1}^{2}(\bp) \, \ve^{-2}\, \|\dot{\ul{x}}\|^2 + \mathcal{C}_{2}^2(\bp) \|\dot{\ul{n}}\|^2.
\end{equation}

We denote by $\nabla_{\bS^{d-1}}$ the gradient operator on $\bS^{d-1}$ with respect to the inner product induced by the embedding $\bS^{d-1} \subset \bR^d$, and by $\nabla_{\bR^d}$ the canonical gradient operator on $\bR^d$.

\begin{corollary} \label{cor:eik}
Let $\ve \geq 0$. 
Then the eikonal PDE (\ref{eqdef:EikonalPDE}) for the case $(\mathbb{M},\cF_\ve)$ takes the form
\begin{equation*}
\begin{array}{c}
\sqrt{
	\textstyle{\frac{\|\nabla_{\bS^{d\!-\!1}}U(\bp)\|^2}{\mathcal{C}^{2}_{2}(\bp)}} +
	\textstyle{\frac{\ve^{2}\|\nabla_{\R^{d}}U(\bp)\|^2 + (1-\ve^2) |\, \ul{n} \cdot \nabla_{\R^{d}}U(\bp) \,|^2}{\mathcal{C}^{2}_{1}(\bp)}}
} = 1, \\
\Leftrightarrow \\
\left.\mathcal{G}_{\bp ; \ve}\right|_{\bp}\left(\,\mathcal{G}_{\bp ; \ve}^{-1} {\rm d}U(\bp)\,, \mathcal{G}_{\bp ; \ve}^{-1} {\rm d}U(\bp)\, \right)=1.
\end{array}
\end{equation*}
The eikonal PDE (\ref{eqdef:EikonalPDE}) for the case $(\mathbb{M},\cF_\ve^+)$ now takes the explicit form:
\begin{equation*}
\begin{array}{c}
\sqrt{
\begin{split}
&\textstyle{\frac{\|\nabla_{\bS^{d\!-\!1}}U^+(\bp)\|^2}{\mathcal{C}^{2}_{2}(\bp)}} + \\
&\textstyle{\hspace*{2em}\frac{\ve^{2}\|\nabla_{\R^{d}}U^+(\bp)\|^2 + (1-\ve^2) |\, (\,\ul{n} \cdot \nabla_{\R^{d}}U^+(\bp)\,)_+\, \,|^2}{\mathcal{C}^{2}_{1}(\bp)}}
\end{split}
}=1 \\[5pt]
\Leftrightarrow \\[5pt]
\left\{
\begin{array}{l}
\left.\mathcal{G}_{\bp ; \ve}\right|_{\bp}\left(\,\mathcal{G}_{\bp ; \ve}^{-1} {\rm d}U^+(\bp)\,, \mathcal{G}_{\bp ; \ve}^{-1} {\rm d}U^+(\bp)\, \right)=1, \\
	\hspace*{3em} \textrm{ if }\bp \in \mathbb{M}_{+}:=\{\bp \in \mathbb{M} \;|\;  \langle {\rm d}U^+(\bp), \ul{n} \rangle > 0 \}, \\[6pt]
\left.\widetilde{\mathcal{G}}_{\bp ; \ve}\right|_{\bp}\left(\,\widetilde{\mathcal{G}}_{\bp ; \ve}^{-1} {\rm d}U^+(\bp)\,, \widetilde{\mathcal{G}}_{\bp ; \ve}^{-1} {\rm d}U^+(\bp) \, \right)=1, \\
	\hspace*{3em} \textrm{ if }\bp \in \mathbb{M}_- := \{\bp \in \mathbb{M} \;|\;  \langle {\rm d}U^+(\bp), \ul{n} \rangle < 0 \}.
\end{array}
\right.
\end{array}
\end{equation*}
for those $\ul{p} \in \mathbb{M}_+ \cup \mathbb{M}_{-}$ where $U^{+}$ is differentiable\footnote{On $\partial M_{\pm}$ distance function $U^+$ is not differentiable}.
\end{corollary}
The proof of Proposition~\ref{prop:DualMetric} and Corollary~\ref{cor:eik} can be found in Section~\ref{ch:proofbacktracking}. \\
\\
We finally specialize the geodesic ODE \eqref{eq:GeodesicODE} to the models of interest. Note that for the model $(\bbM,d_{\cF_\ve^+})$, the backtracking switches between qualitatively distinct modes, respectively almost sub-Riemannian and almost purely angular, in the spirit of Theorem \ref{th:CuspsAndRotations}.
Given $\ve > 0$ and $\bn \in \bS^{d-1}$ let $D_\bn^\ve$ denote the $d \times d$ symmetric positive definite matrix with eigenvalue $1$ in the direction $\bn$, and eigenvalue $\ve^2$ in the orthogonal directions :
\begin{equation} \label{Dnve}
D_\bn^\ve := \bn \otimes \bn + \ve^2 (\Id-\bn\otimes\bn).
\end{equation}

\begin{theorem}[Backtracking]
\label{th:Backtracing}
Let $0< \ve <1$. Let $\pSource\in \bM$ be a source point.
Let $U(\bp):= d_{\mathcal{F}_{\ve}}(\bp,\bp_s)$, $U^+(\bp):=d_{\mathcal{F}_{\ve}^+}(\bp,\bp_s)$ be distance maps from $\bp_{s}$, w.r.t. the
Finsler metric $\mathcal{F}_{\ve}$, and $\mathcal{F}_{\ve}^+$.
Let $\gamma, \gamma^+:[0,1] \to \mathbb{M}$ be normalized geodesics of length $L$ starting at $\bp_s$ in $(\mathbb{M},d_{\mathcal{F}_{\ve}})$ resp.
$(\mathbb{M},d_{\mathcal{F}_{\ve}^+})$. Let time $t \in [0,1]$.

For the Riemannian approximation paths of the Reeds-Shepp car we have, provided that $U$ is differentiable at $\gamma(t)=(\ul{x}(t),\ul{n}(t))$, that
\begin{equation} \label{btsimple}
\begin{array}{c}
\dot{\gamma}(t) = L \, \mathcal{G}^{-1}_{\gamma(t); \ve} {\rm d}U(\gamma(t)) \\ \Leftrightarrow \\
\left\{
\begin{array}{ll}
\dot \bn(t) &= L\, \cC_{2}(\gamma(t))^{-1}\; \nabla_{\bS^{d-1}} U(\gamma(t)), \\
		\dot \bx(t) &= L\, \cC_{1}(\gamma(t))^{-1}\;
		D_{\bn(t)}^\ve \nabla_{\bR^d} U(\gamma(t)).
\end{array}
\right.
\end{array}
\end{equation}

For the approximation paths of the car without reverse gear we have, provided that $U^+$ is differentiable at $\gamma^+(t) = (\bx^+(t),\bn^+(t))$, that
\begin{equation} \label{btabstract}
\dot{\gamma}^+(t) = L \left\{
\begin{array}{ll}
\mathcal{G}^{-1}_{\gamma^+(t); \ve} {\rm d}U^+(\gamma^+(t))  &\textrm{if }\gamma^+(t) \in \mathbb{M}_+,\\
\widetilde{\mathcal{G}}^{-1}_{\gamma^+(t); \ve} {\rm d}U^+(\gamma^+(t))  &\textrm{if }\gamma^+(t) \in \mathbb{M}_-,
\end{array}
\right.
\end{equation}
with $\widetilde{\mathcal{G}}_{\bp; \ve}(\dot{\bp},\dot{\bp})$ given by (\ref{Gtilde}), with disjoint Riemannian manifold splitting $\mathbb{M}=\mathbb{M}_{+} \cup \mathbb{M}_{-} \cup \partial \mathbb{M}_{\pm}$. Manifold $\mathbb{M}_{+}$ is equipped with metric tensor $\mathcal{G}_{\ve}$, $\mathbb{M}_{-}$ is equipped with metric tensor $\widetilde{\mathcal{G}}_{\ve}$ and
\begin{equation}\label{transitionsurface}
\partial\mathbb{M}_{\pm} := \overline{\bbM_+} \setminus \bbM_+ = \overline{\bbM_-} \setminus \bbM_-
\end{equation}
denotes the transition surface (surface of keypoints).
\end{theorem}
\begin{remark}
The general abstract formula (\ref{btabstract}) reflects that the backtracking in $(\mathbb{M},\mathcal{F}^{+})$ is a combined gradient descent flow on
the distance map $U^+$ on a splitting of $\mathbb{M}$ into two (symmetric) Riemannian manifolds.
Its explicit form (likewise (\ref{btsimple})) is
\begin{equation} \label{simple2}
\left\{
\begin{array}{ll}
\dot \bn^+(t) &= L\, \cC_{2}(\gamma^+(t))^{-1}\; \nabla_{\bS^{d-1}} U^+(\gamma^+(t)), \vspace{1em}\\
		\dot \bx^+(t) &=
L \left\{
\begin{array}{lll}
 \cC_{1}(\gamma^+(t))^{-1}\; D_{\bn(t)}^\ve \nabla_{\bR^d} U^+(\gamma^+(t))  \vspace{.2em}
 	\\ \qquad \qquad \textrm{ if }\gamma^+(t) \in \mathbb{M}_+, \vspace{1em}\\
 \ve^2 \,  \cC_{1}(\gamma^+(t))^{-1}\; \nabla_{\bR^d} U^+(\gamma^+(t))  \vspace{.2em}
	\\ \qquad \qquad \textrm{ if }\gamma^+(t) \in \mathbb{M}_-,\\
\end{array}
\right.
\end{array}
\right.
\end{equation}
Note that for the (less useful) isotropic case $\ve=1$, 
$\mathcal{F}_{1}$ and $\mathcal{F}_{1}^+$ coincide and geodesics consist of straight lines $\ul{x}(\cdot)$ in $\R^{d}$ and great circles $\ul{n}(\cdot)$ in $\bS^{d}$ that do not influence each other.
\end{remark}

\begin{remark}\label{rem:nonsmoothpoint}
In Theorem \ref{th:Backtracing}, we assumed distance maps $U$ and $U^+$ to be differentiable along the path, which is not always the case. In points where the distance map is not differentiable, one can take any sub-gradient in the sub-differential $\partial U(\ul{p})$ in order to identify Maxwell points (and Maxwell strata). In particular, in SR geometry, the set of points where the squared distance function $(d_{\cF_0}(\cdot,e))^2$ is smooth is open and dense in any compact subset of $\bbM$, see \cite[Thm. 11.15]{agrachev_introduction_2016}. The points where it is non-smooth are rare and meaningful: they are either first Maxwell points, conjugate points or abnormal points. The last type does not appear here, because we have a 2-bracket generating distribution, see e.g. \cite[Remark 4]{duits_sub-riemannian_2016} and \cite[Ch. 20.5.1.]{agrachev_introduction_2016}. At points in the closure of the first Maxwell set, two geodesically equidistant wavefronts collide for the first time, see for example \cite[Fig.3, Thm 3.2]{bekkers_pde_2015} for the case $d=2$ and $\mathcal{C}=\mathcal{C}_{1}=\mathcal{C}_{2}=1$. See also Fig.~\ref{fig:JorgSE2}, where for some end conditions 2 optimally back-tracked geodesics end with the same length in such a first Maxwell point. The conjugate points are points where local optimality is lost, for a precise definition see e.g. \cite[Def. 8.43]{agrachev_introduction_2016}.
\end{remark}

\begin{remark}\label{rem:continuousfitonlyifepsilonis0}
Recall the convergence result from Theorem~\ref{th:ReedSheppCV}, and the non-local-controllability for the model $(\mathbb{M},d_{\mathcal{F}_{0}^+})$.
From this we see that the convergence holds pointwise but not uniformly (otherwise the limit distance $d_{\mathcal{F}_0^+}$ was continuous).
Nevertheless the shortest \todo{R5.1: terminology} paths converge strongly as $\ve \downarrow 0$, and we see that the spatial velocity tends to $0$ in (\ref{simple2}) if $\ve \downarrow 0$ if $\gamma_{\ve}^*(t) \in \mathbb{M}_{-}$. In the SR case $\ve=0$, the gradient flows themselves fit continuously and the interface $\partial \bM_{\pm}$ is reached with $\dot{\bx} \cdot \bn = 0$ (and $\dot{\bx} = 0$).

\end{remark}

Theorem \ref{th:Backtracing} can be extended to the SR case:
\begin{corollary}[SR Backtracking]
Let the cost $\cC_1, \cC_2$ be smooth, let the source $\bp_S \in \bbM$ and $\bp \neq \bp_S \in \bbM$ be such that they can be connected by a unique smooth minimizer $\gamma_\ve^*$ in $(\bbM,\cF_\ve)$ and $\gamma_0^*$ in $(\bbM, \cF_0)$, such that $\gamma_\ve^*(t)$ is not a conjugate point for all $t \in [0,1]$ and all sufficiently small $\ve >0$, say $\ve < \ve_0$, for some $\ve_0 > 0$.
Then defining $U_0: \bq \in \bM \mapsto d_{\cF_\ve}(\bp_s,\bq)$ one has

\[
\dot{\gamma}^*_0(t) = U_0(\bp)\mathcal{G}_{\gamma^*_0(t) ; 0}^{-1} \, {\rm d}U_{0}(\gamma_0^*(t)), \qquad t \in [0,1],
\]
assuming $U_0$ is differentiable at $\gamma_0^*(t)$. In addition $U_0$ satisfies the SR eikonal equation:

\begin{equation*}
\sqrt{\mathcal{G}_{\bp;0}\left(\mathcal{G}_{\bp;0}^{-1} {\rm d}U_0(\bp),
\mathcal{G}_{\bp;0}^{-1} {\rm d}U_0(\bp)\right)} =1.
\end{equation*}

\end{corollary}

\begin{proof}
From our assumptions on $\bp$ and $\gamma_\ve^*(t)$ for $\ve< \ve_0$, we have, recall Remark \ref{rem:nonsmoothpoint}, that $(U_\ve(\cdot))^2$ is differentiable at $\gamma_\ve^*(t)$ for all $0 \leq t \leq 1$ and $0 \leq \ve < \ve_0$. This implies that $U_\ve$ is differentiable at $\{\gamma_\ve^*(t) \; \vline \; 0 < t \leq 1 \}$, for all $0<\ve < \ve_0$.

From Theorem \ref{th:ReedSheppCV} we have pointwise convergence $U_\ve(\bp)\rightarrow U_0(\bp)$ and uniform convergence $\gamma_\ve^* \rightarrow \gamma_0^*$ as $\ve \downarrow 0$. Moreover, as $\gamma_\ve^*$ and $\gamma_0^*$ are solutions of the canonical ODEs of Pontryagin's Maximum Principle, the trajectories are continuously depending on $\ve >0$, and so are the derivatives $\dot{\gamma}_\ve^*$. As a result, we can apply the backtracking Theorem \ref{th:Backtracing} for $\ve > 0$ and take the 
limits:
\begin{equation}
\begin{aligned}
\dot{\gamma}_0^*(t) &= \lim_{\ve \downarrow 0} \dot{\gamma}_\ve^*(t) \\
&\hspace*{-.9em}\overset{\text{Thm. 4}}{=} \lim_{\ve \downarrow 0} U_\ve(\bp)\,(\mG_{\gamma_\ve^*(t);\ve}^{-1} {\rm d} U_\ve)(\gamma_\ve^*(t)) \\
&= U_0(\bp) \, \left( \lim_{\ve \downarrow 0} \mG_{\gamma_\ve^*(t);\ve}^{-1}\right)\left( \lim_{\ve \downarrow 0} ({\rm d} U_\ve(\gamma_\ve^*(t)))\right) \\
&\hspace*{-.9em}\overset{\text{Thm. 2}}{=} U_0(\bp) \, \mG_{\gamma_0^*(t);0}^{-1}({\rm d} U_0)(\gamma_0^*(t)).
\end{aligned}
\end{equation}
Furthermore,
\begin{equation*}
\begin{split}
1&=
\lim \limits_{\ve \downarrow 0} \sqrt{\mathcal{G}_{\bp;\ve}\left(\mathcal{G}_{\bp;\ve}^{-1} {\rm d}U_{\ve}(\bp),
\mathcal{G}_{\bp;\ve}^{-1} {\rm d}U_{\ve}(\bp)\right)} 
\\
&=\sqrt{\mathcal{G}_{\bp;0}\left(\mathcal{G}_{\bp;0}^{-1} {\rm d}U_0(\bp),
\mathcal{G}_{\bp;0}^{-1} {\rm d}U_0(\bp)\right)}
\end{split}
\end{equation*}
where we recall Corollary~\ref{cor:eik}.
Here due to our assumptions, $U_\ve$ and $U_0$ are both differentiable at $\bp$. Note that the limit for the inverse metric $\mG_{\bp,\ve}^{-1}$ as $\ve \downarrow 0$ exists, recall Cor. \ref{cor:eik}.
\end{proof}

Now that we stated our 4 main theoretical results we will prove them in the subsequent sections (and Appendix~\ref{app:WellPosedness}).

\section{Controllability Properties: Proof of Theorem~\ref{th:Controllability}, and Maxwell-points in $(\mathbb{M},d_{\cF_0^+})$ \label{ch:3}}

{\bf (Global controllability)} The two considered Reeds-Shepp models 
 $(\mathbb{M},d_{\cF_0})$ and $(\mathbb{M},d_{\cF_0^+})$ are globally controllable, in the sense that the distances $d_{\cF_0}$ and $d_{\cF_0^+}$ take finite values on $\bM\times \bM$. This easily follows from the observation that any path $\bx : [0,1] \to \bR^d$, which time derivative $\dot \bx := \frac {\diff \bx}{\diff t}$ is Lipschitz and non-vanishing, can be lifted into a path $\gamma : [0,1]\to \bM$ of finite length w.r.t.\ $\cF_0$ and $\cF_0^+$, defined by $\gamma(t):=(\bx(t),\dot \bx(t)/\|\dot \bx(t)\|)$ for all $t\in [0,1]$.
The fact that the infimum in (\ref{eqdef:dF}) is actually a minimum for $\mathcal{F}=\mathcal{F}_{0}^+$ follows by
Corollary~\ref{corol:ExistsMinOC} in App. \ref{app:WellPosedness} and (\ref{viewpoint}), and the fact that the quasi-distances \todo{R1.3: terminology} take finite values.

{\bf (Local controllability)} In order to show that the model $(\bbM,d_{\cF_0^+})$ is \emph{not} locally controllable, we need the following lemma.
\begin{lemma}\label{lemma:JM}
Let $\ul{n}:[0,\pi] \to \mathbb{S}^{d-1}$ be strictly 1-Lipschitz. Then $\int_{0}^{\pi} \ul{n}(0) \cdot \ul{n}(t)\, {\rm d}t>0$.
Let $\ul{n}:\R \to \mathbb{S}^{d-1}$ be strictly 1-Lipschitz and $2\pi$-periodic. Then all points $\ul{n}(t)$ lay in a common strict hemisphere. In particular $\ul{0} \notin \textrm{Hull}\{\ul{n}(t)\;|\; t \in [0,2\pi] \}$.
\end{lemma}
\begin{proof}  The Lipschitzness assumption implies $\ul{n}(0) \cdot \ul{n}(t) > \cos (t)$ for all $t \in (0,\pi]$ so $\int_{0}^{\pi} \ul{n}(0) \cdot \ul{n}(t)\, {\rm d}t>0$.

Let $\ul{n}:\R \to \mathbb{S}^{d-1}$ be strictly 1-Lipschitz and $2\pi$-periodic.
Set $\ul{M}:= \int_{0}^{2\pi}\ul{n}(t)\, {\rm d}t$. Then for any $t_0 \in [0,2\pi]$ one has by the two assumptions
\[
\ul{n}(t_0) \cdot \ul{M} = \int \limits_{0}^{\pi} \ul{n}(t_0) \cdot \ul{n}(t_0+t) \,{\rm d}t + \int \limits_{0}^{\pi} \ul{n}(t_0) \cdot \ul{n}(t_0-t) \,{\rm d}t >0,
\]
so for all $t_0$, $\bn(t_0) \in \{\ul{n} \in \mathbb{S}^{d-1} \;|\; \ul{n} \cdot \ul{M} >0\}$.
\end{proof}
Now the statements \eqref{estimate} and \eqref{equality} on the non-local-controllability of $(\bbM, d_{\cF^+_0})$ are shown in two steps.
\\
Step 1: we show in the case of a constant cost function $\cC_{2}=\delta$ one has
 {\small $\limsup \limits_{\bp' \to \bp} d_{\cF^+_0}(\bp,\bp') \leq 2\pi \delta$}, for any $\bp \in \bM$. Indeed, one can design an admissible curve in $(\mathbb{M}, \mathcal{F}_{0}^+)$ as the concatenation of an in-place rotation, a straight line, and an in-place rotation. The length of the straight line is $\cO(\|\bp'-\bp\|)$ and vanishes when $\bp' \to \bp$, and the in-place rotations each have maximum cost $\pi \delta$.
\\
Step 2: we prove the lower bound {\small $\lim \limits_{\mu \downarrow 0} d_{\cF_0^+}((\bx,\bn),(\bx-\mu\bn,\bn)) \geq 2 \pi \delta$}, for any $(\bx,\bn)\in \bM$. This and the above established upper bound implies the required result.  As $\mathcal{C}_{1}, \mathcal{C}_{2} \geq \delta$,
we can restrict ourselves to the case of uniform cost $\mathcal{C}_{1}=\mathcal{C}_{2}=\delta=1$ and just show equality (\ref{equality}),
as the estimate (\ref{estimate}) follows by scaling with $\delta$.

Consider a Lipschitz regular path $\gamma(t) = (\bx(t), \bn(t))$, with $\dot \bx \propto \bn \text{ and } \dot \bx \cdot \bn \geq 0$, from $(\bx,\bn)$ to $(\bx-\mu\bn,\bn)$.
Then
\begin{equation*}
	\ul{0}= \mu \bn + \int_0^1 \dot \bx(t) {\rm d}t = \mu \bn(0) +\int_0^1 \|\dot \bx(t)\| \bn(t) {\rm d}t,
\end{equation*}
so $\ul{0} \in \Hull\{\bn(t);\, 0 \leq t \leq 1\}$. Let $\ul{m}:[0,1] \to \mathbb{S}^{d-1}$ be a constant speed parametrization of $\ul{n}$. Let $\tilde{\ul{m}}: \R \to \mathbb{S}^{d-1}$ be defined by $\tilde{\ul{m}}(2\pi t)=\ul{m}(t)$
for all $t \in [0,2\pi]$, and extended by $2\pi$-periodicity. If $\tilde{\ul{m}}(\cdot)$ were strictly 1-Lipschitz then by Lemma~\ref{lemma:JM} we would get
$\ul{0} \notin \textrm{Hull}\{\tilde{\ul{m}}(t)\;|\; t \in [0,2\pi] \}=\textrm{Hull}\{\ul{n}(t)\;|\; t \in [0,1]\}$ and a contradiction. Hence there exists a $t_0 \in \R$ such that $\|\dot{\tilde{\ul{m}}}(t_0)\|\geq 1$ and via the constant speed parametrization
assumption we get the required coercivity:
\begin{equation*}
\begin{split}
&1 \leq \|\dot{\tilde{\ul{m}}}(t_0)\| = \frac{1}{2\pi} \int_0^1 \|\dot{\ul{n}}(t)\|\, {\rm d}t \Rightarrow \\ &\int_{0}^{1} \mathcal{F}_{0}^{+}(\gamma(t),\dot{\gamma}(t))\, {\rm d}t \geq \int_{0}^{1} \mathcal{C}_{2}(\gamma(t))\,
\|\dot{\ul{n}}(t)\|\, {\rm d}t \geq 2\pi \delta.
\end{split}
\end{equation*}

To prove local controllability of the model $(\bbM, d_{\cF_0})$, we apply the logarithmic approximation for weighted sub-coercive operators on Lie groups, cf.~\cite{terelst} applied to the Lie group $SE(d)=\R^{d} \rtimes SO(d)$, in which the space of positions and orientations is placed via a Lie group quotient
$SE(d)/(\{0\} \times SO(d-1))$. One obtains a sharp estimate\footnote{For specific sharp estimates for $d=3$, in the context of heat-kernels estimation, see \cite[ch.5.1]{jorg}.}, where the weights of allowable (horizontal) vector fields is $1$, whereas the remaining spatial vector fields orthogonal to $\mathbf{n} \cdot \nabla_{\R^d}$ get weight $2$, as they follow by a single commutator of allowable vector fields, see e.g. \cite{duits_sub-riemannian_2016,duits_cuspless_2014}.
Relaxing all spatial weights to $2$ and continuity of costs $\mathcal{C}_{1}, \mathcal{C}_{2}$, yields (\ref{approx}). $\hfill \Box$

\begin{remark}
In view of the above one might expect that the point $(\ul{x} - \mu \ul{n}, \ul{n})$ is reached by a geodesic that consists of a concatenation of 1. an in-place rotation by $\pi$, 2. a straight line, 3. an in-place rotation by $\pi$.
However, this is not the case as can be observed
in the very lower left corner in Fig.~\ref{fig:JorgSE2}, where the \emph{two} minimizing red curves show a very different behavior.
This is explained by the next lemma.
\end{remark}
\begin{lemma} \label{lemma:frombehind}
Let $\mu>0$, and $\mathcal{C}_{1}=\mathcal{C}_{2}=\delta$. Let $\bfR_\theta$ denote the (counter-clockwise) rotation matrix about the origin by angle $\theta$. The endpoint $(\bx-\mu\bn,\bn)$ for each $\mu \geq 0$ is a Maxwell point w.r.t. $(\bx,\bn)$, since there are two minimizing geodesics in $(\mathbb{M},d_{\mathcal{F}_{0}^+})$ that are a concatenation
\begin{enumerate}
\item an in-place rotation from $(\bx,\bn)$ to $(\bx,\bfR_{\pm \frac{\pi}{2}}\bn)$,
\item a full U-curve, see \cite{moiseev_maxwell_2010}, departing from and ending in a cusp from $(\bx,\bfR_{\pm \frac{\pi}{2}}\bn)$ to $(\bx - \mu \bn,\bfR_{\mp \frac{\pi}{2}}\bn)$,
\item an in-place rotation from $(\bx - \mu \bn,\bfR_{\mp \frac{\pi}{2}}\bn)$ to \\ $(\bx-\mu\bn,\bn)$.
\end{enumerate}
We have the limit $\lim \limits_{\mu \downarrow 0} d_{\mathcal{F}_{0}^+}((\bx,\bn),(\bx-\mu\bn,\bn))=2\pi \delta$.
\end{lemma}
\begin{proof}
By rotation and translation covariance, we can restrict ourselves to $\ul{x}=\ul{0}$, $\ul{n}=\ul{a}$, and by the coplanarity result in \cite[Cor.6, Thm.8]{duits_sub-riemannian_2016}, we only need to consider $d=2$ and $\ul{a}=(1,0)^T$.
From Theorem \ref{th:CuspsAndRotations} (that we prove later) we know that keypoints only occur at the endpoints of minimal paths. Since cuspless geodesics stay in the positive half-space set by their initial orientation, recall Remark \ref{rem:globalmingeodesic}, the optimal path from $(\mathbf{0},\ba)$ to $(-\mu \ba,\ba)$ starts with a rotation by at least a $\frac{\pi}{2}$ angle. For the same reason, it ends with a rotation by at least a $\frac{\pi}{2}$ angle. There exists a U-curve, starting and ending in a cusp, optimally connecting $(\mathbf{0},\bfR_{\frac{\pi}{2}}\ba)$ with $(-\mu \ba, \bfR_{-\frac{\pi}{2}}\ba)$ \cite{duits_association_2013}. Hence the concatenation of curves as described points 1.-3. is optimal, and has the same length as the alternative (with rotations in the opposite direction).
For the total distance by such a curve we use \footnote{In \cite[Cor.2]{duits_association_2013} one must set $z(0)=z(L)=1$ and
$\dot{z}(0)=\chi \downarrow 0$} \cite[Cor.2]{duits_association_2013}:
\begin{equation*}
\begin{split}
&d_{\mathcal{F}_{0}^+}((\bx,\bn),(\bx-\mu\bn,\bn))= \\
&\frac{\pi}{2}+ \int_{0}^{s_{max}(1,\mu)}
 \frac{1}{\sqrt{1- |\mu \sinh s + \cosh s|^2}}\, {\rm d}s  +\frac{\pi}{2},
\end{split}
\end{equation*}
with $s_{max}(1,\mu)$ the first positive root of the denominator of the integrand. Letting $\mu \downarrow 0$ we get $\lim \limits_{\mu \downarrow 0} d_{\mathcal{F}_{0}^+}((\bx,\bn),(\bx-\mu\bn,\bn))=2\pi$.
\end{proof}

\begin{remark}
Consider the case $d=2$, $\mathcal{C}_{1}=\mathcal{C}_{2}=\delta$, and source point $\ul{p}_S=(\ul{x},\ul{n})=\ul{e}=(0,0,\theta=0)$.
The end-points $(\bx-\mu\bn,\bn)=(-\mu,0,0)$, with $\mu >0$ sufficiently small, are 1st Maxwell-points in $(\mathbb{M},d_{\mathcal{F}_{0}^+})$ where geodesically equidistant wavefronts departing from the source point collide for the first time, see Fig.~\ref{fig:Spheres}C.
The distance mapping $d_{\mathcal{F}_{0}}^{+}(\bp_{S}, \cdot)$ is not continuous, but the asymmetric distance spheres

\begin{equation*}
\mathcal{S}_R:=\{\ul{p} \in \mathbb{M} \;|\; d_{\mathcal{F}_{0}}^{+}(\bp_{S}, \bp)=R\}
\end{equation*} 
are connected and compact, and they collide at $R=2\pi$ in such a way that the origin $\bp_{s}$ becomes an interior point in the asymmetric balls of radius $R> 2\pi$.
\end{remark}
\begin{figure*}
\centerline{
\includegraphics[width=\hsize]{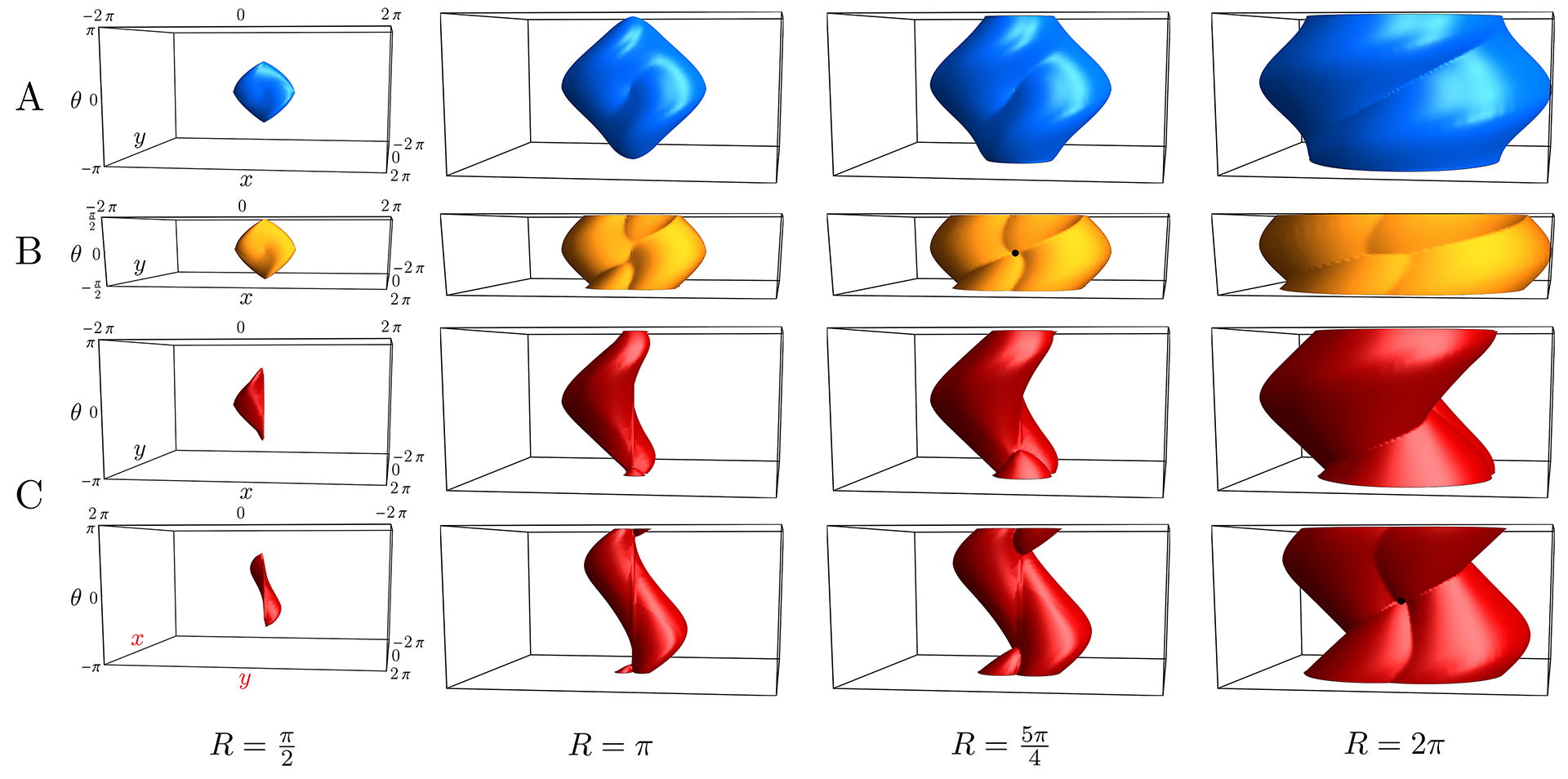}
}
\caption{The development of spheres centered around $\ul{e}=(0,0,0)$ with increasing radius $R$. \textbf{A:} the normal SR spheres on $\mathbb{M}$ given by $\{\ul{p} \in \mathbb{M} \;|\; d_{\mathcal{F}_0}(\ul{p},\ul{e})=R\}$ where the folds reflect the 1st Maxwell sets \cite{bekkers_pde_2015,sachkov_cut_2011}.
\textbf{B: } the SR spheres with identification of antipodal points given by $\left\{\ul{p} \in \mathbb{M} \;|\; \min\{\; d_{\mathcal{F}_0}(\ul{p},\ul{e}), d_{\mathcal{F}_0}(\ul{p}+(0,0,\pi),\ul{e})\; \}=R\right\}$ with additional folds (1st Maxwell sets) due to $\pi$-symmetry.
\textbf{C: } the asymmetric Finsler norm spheres given by $\{\ul{p} \in \mathbb{M} \;|\; d_{\mathcal{F}_0^+}(\ul{p},\ul{e})=R\}$ visualized from two perspectives with extra folds (1st Maxwell sets) at the back $(-\mu,0,0)$.
The black dots indicate points with two folds. In the case of B, this is a Maxwell-point with 4 geodesics merging. In the case of C, this is just the origin itself reached from behind at $R=2\pi$, recall Lemma~\ref{lemma:frombehind}. Although not depicted here, if the radius $R>2\pi$ the origin becomes an interior point of the corresponding ball.
 \label{fig:Spheres}}
\end{figure*}

\section{Cusps and Keypoints: Proof of Theorem~\ref{th:CuspsAndRotations} }\label{ch:proofcuspskeypoints}

In this section we provide a proof of Theorem~\ref{th:CuspsAndRotations} on the occurrence of cusps and keypoints. For the uniform cost case
$\mathcal{C}_{1}=\mathcal{C}_{2}=1$ for $d=2$, our
curve-optimization problem (\ref{eqdef:dF})
$(\mathbb{M},d_{\mathcal{F}_{0}})$ in consideration, boils down to a standard left-invariant curve optimization in the roto-translation group $SE(2)=\R^{2} \rtimes SO(2)$.
As we will apply tools from previous works \cite{duits_association_2013,boscain_curve_2014,boscain_existence_2010,sachkov_cut_2011}, we will make use of the following notations for expansion\footnote{Note that we use upper-indices for the control's (velocity components) as they are contra-variant.} of velocity and momentum in the left-invariant (co)-frame:
\begin{equation}\label{eq:speedmomentuminframe}
\begin{array}{c}
\left\{
\begin{array}{l}
\mathcal{A}_{1}:= \cos \theta \, \partial_{x} + \sin \theta \, \partial_y, \\
\mathcal{A}_{2}:= -\sin \theta \, \partial_x + \cos \theta \, \partial_y, \\
\mathcal{A}_{3}:=\partial_{\theta},
\end{array}
\right. \vspace{1em} \\
\left\{
\begin{array}{l}
\omega^{1} := \cos \theta \, {\rm d}x + \sin \theta \, {\rm d}y, \\
\omega^{2} := -\sin \theta \, {\rm d}x + \cos \theta \, {\rm d}y, \\
\omega^{3} := {\rm d}\theta,
\end{array}
\right. \\
\hspace*{1em}
\dot{\gamma}(t)= \sum \limits_{i=1}^{3} u^{i}(t) \, \left.\mathcal{A}_{i}\right|_{\gamma(t)} \in T_{\gamma(t)}(\mathbb{M}), \\ \hspace{1em}
\hat{\bp}(t)= \sum \limits_{i=1}^{3} \hat{p}_{i}(t) \, \left.\omega^{i}\right|_{\gamma(t)} \in T_{\gamma(t)}^*(\mathbb{M}),
\end{array}
\end{equation}
where the indexing of the left-invariant frame is different here, in order to stick to the ordering $(x,y,\theta)$ applied in this article.
Note that for the case $\ve=0$ admissible smooth curves $\gamma$ in $(\mathbb{M}, d_{\mathcal{F}_{0}})$ satisfy the horizontality constraint
$\dot{\gamma}(t) \in \textrm{Span}\{\left.\mathcal{A}_{1}\right|_{\gamma(t)},\left.\mathcal{A}_{3}\right|_{\gamma(t)}\}$.

\textbf{Proof of the statements regarding cusps:}
\begin{itemize}
\item
We can describe our curve optimization problem (\ref{eqdef:dF}) using a Hamiltonian formalism, with Hamiltonian $H(\hat{\bp})= \frac{1}{2}\left( \hat{p}_{1}^2 + \hat{p}_{3}^2 \right)= \frac{1}{2}$ \cite{moiseev_maxwell_2010}. By Pontryagin's Maximum Principle, geodesics adhere to the following Hamilton equations:

\begin{equation} \label{PMPinaccessibleform}
\left\{
\begin{array}{l}
\dot{p}_{1}=u^1= \hat{p}_{1}, \\
\dot{p}_{2}=u^2= 0, \\
\dot{p}_{3}=u^3=\hat{p}_{3},
\end{array}
\right.
, \hspace{2em}\
\left\{
\begin{array}{l}
\frac{d \hat{p}_{1}}{dt}= \hat{p}_{2} \hat{p}_{3}, \\
\frac{d \hat{p}_{2}}{dt}= -\hat{p}_{1} \hat{p}_{3}, \\
\frac{d \hat{p}_{3}}{dt}= -\hat{p}_{1} \hat{p}_{2}.
\end{array}
\right.
\end{equation}
For fixed initial momentum $\hat{\bp}(0)$, this uniquely determines a SR geodesic. Moreover, SR geodesics are contained within the (co-adjoint) orbits
\begin{equation} \label{coadjoint}
(\hat{p}_{1}(t))^2 + (\hat{p}_{2}(t))^2=(\hat{p}_{1}(0))^2 + (\hat{p}_{2}(0))^2.
\end{equation}
The parameter $t$ in the system \eqref{PMPinaccessibleform} is SR arc length, but by reparametrizing (possible as long as $u^1$ does not change sign) to spatial arc length parameter $s$, with $\frac{ds}{dt} = \hat{p}_1$, we get a partially linear system. Combining \eqref{PMPinaccessibleform} and \eqref{coadjoint}, we find orbits in the (hyperbolic) phase portrait induced by
\[
\begin{array}{l}
\left\{
\!
\begin{aligned}
\hat{p}_{2}'(s)= -\hat{p}_{3} \\
\hat{p}_{3}'(s)= -\hat{p}_{2}
\end{aligned} \right.\hspace{-.5em}  \Rightarrow \!
\left\{
\!\begin{aligned}
\hat{p}_2(s) &= \hat{p}_2(0) \cosh{s} - \hat{p}_3(0) \sinh{s}\\
\hat{p}_3(s) &= -\hat{p}_2(0) \sinh{s} + \hat{p}_3(0) \cosh{s}.
\end{aligned}\right.

\end{array}
\]
Hence $|\hat{p}_3(s)| = 1$ always has a solution for some finite (possibly negative) $s$, except when $\hat{p}^{2}(0)=\hat{p}_{3}(0)=0$, in which case the solutions are straight lines. Preservation of the Hamiltonian then implies $\hat{p}_1(s) = u^1(s) = \tilde{u}(s)=0$. We conclude that every SR geodesic (with unconstrained time $t \in \R$) in $(\mathbb{M},d_{\mathcal{F}_0})$ which is not a straight line admits a cusp.
\item We now consider $(\mathbb{M},d_{\mathcal{F}_{\ve}})$, $\ve > 0$. To have a cusp, we need $\hat{p}_1(t) = \hat{p}_2(t) = 0$ for some $t \in \R$. The co-adjoint orbit condition \eqref{coadjoint} then implies that $\hat{p}_1(t) = \hat{p}_2(t) = 0$ for all $t$, corresponding to a vertical geodesic that has purely angular momentum and no cusp. The same argument holds for $(\bbM,d_{\cF_\ve^+})$. In $(\bbM,d_{\cF_0^+})$ we have the condition that $u^1 \geq 0$, hence by definition it can never switch sign and all geodesics are cuspless.
\end{itemize}
\textbf{Proof of the statements regarding keypoints: }
\begin{itemize}
\item For the cases $(\bbM,d_{\cF_\ve})$ and $(\bbM,d_{\cF_\ve^+})$ with $\ve > 0$ we can use the same line of arguments as above. Also here both spatial controls have to vanish, resulting in vertical geodesics. The spatial projection of such curves is a single keypoint. For $(\bbM,d_{\cF_0})$ we rely on the result that SR geodesics are analytical, and therefore if the control $u^1(t) = 0$ for some open time interval $(t_0,t_1)$, then $u^1(t) = 0$ for all $t \in \R$, again corresponding to purely angular motion.
\item Geodesics in $(\mathbb{M},d_{\mathcal{F}_0^+})$ can have keypoints only at the boundaries. Suppose a geodesic
$\gamma:[0,1] \to \mathbb{M}$ in $(\mathbb{M},d_{\mathcal{F}_0^+})$ has an internal keypoint, with a corner of angle $\delta>0$, at internal time $T_{1} \in (0,1)$. Then one can create a local shortcut with a straight line segment connecting two sufficiently close points before and after the corner with two in-place rotations whose angles add up to $\delta$. With a suitable mollifier this shortcut can be approximated by a curve in $\Gamma$. For details see similar arguments in \cite{boscain_curve_2014}.
\end{itemize}
Next we explain the cases A), B) and C), where we fix initial point $\gamma(0)=\ul{e}=(0,0,0)$.
\begin{itemize}
\item[A)] Suppose that the endpoint $\ul{p}=(x,y,\theta) \in \overline{\gothic{R}}$ and $x \geq 0$. Then $\ul{p}$ can already be reached by a geodesic in $(\mathbb{M},d_{\mathcal{F}_0})$ and the positivity constraint (i.e. no reverse gear), which can only increase length, becomes obsolete.
\item[B)] Now suppose the endpoint $\ul{p}=(x,y,\theta)$ lays in the half-space $x<0$. Then by the half-space property of geodesics in $(\mathbb{M},d_{\mathcal{F}_0})$, cf.\cite[Thm.7]{duits_association_2013}, the geodesic in $(\mathbb{M},d_{\mathcal{F}_0^+})$
     must have a keypoint. By the preceding keypoints 
     can only be located at the boundaries. If it takes place at the endpoint only, then still the constraint $x<0$ is not satisfied, thereby it must take place at the origin.
\item[C)] In those cases the endpoint $\ul{p}$ lays outside the connected cone of reachable angles, which are by \cite[Thm.9]{duits_association_2013} bounded (for those endpoints) by geodesics ending in a cusp (so not endpoints of geodesics starting at a cusp). So for those points, minimizing geodesics will first move by an in-place rotation (along a spherical geodesic) until it hits the cusp surface $\partial \gothic{R}$, after which it is traced back to the origin by a regular geodesic with strictly positive spatial control inside the volume $\gothic{R}$.
\end{itemize}

\section{Eikonal equations and backtracking: Proof of Prop. \ref{prop:DualMetric}, Corr. \ref{cor:eik} and Thm. ~\ref{th:Backtracing}}\label{ch:proofbacktracking}

First we shall prove Proposition~\ref{prop:DualMetric}, regarding the duals of $\cF_\ve$ and $\cF_\ve^+$, and Corollary~\ref{cor:eik}, providing explicit expressions for the corresponding eikonal equations. To this end we need a basic lemma on computing dual norms on $\R^{n}$, where later we will set $n = 2d - 1 = \mathrm{dim}(\bbM)$.
\begin{lemma}\label{le:normanddualnorm}
Let $\ul{w} \in \R^{n}$ and let $M \in \R^{n\times n}$ be symmetric, positive definite. Define the norm
$F_{M, \ul{w}}:\R^{n} \to \R^+$ by
\[
F_{M, \ul{w}}(\ul{v})= \sqrt{(M\ul{v},\ul{v}) + (\ul{w},\ul{v})_{-}^2}.
\]
Then its dual norm $F_{M, \ul{w}}^*: (\R^{n})^* \to \R^+$ equals
\begin{equation}\label{eq:dualFMw}
F_{M, \ul{w}}^*(\hat{\ul{v}})= \sqrt{(\hat{\ul{v}},\hat{M}\hat{\ul{v}}) +  (\hat{\ul{v}},\hat{\ul{w}})_{+}^2},
\end{equation}
\mbox{with $\hat{M}=(M + \ul{w} \otimes \ul{w})^{-1}$ and $\hat{\ul{w}}=\frac{M^{-1}\ul{w}}{\sqrt{1+ (\ul{w}, M^{-1}\ul{w})}}$}.
\end{lemma}
\begin{proof}
For $n=1$ the result is readily verified, and for $\ul{w}=\ul{0}$ the result is classical.
We next turn to the special case $M= \Id$, and $\bw = (w_1, \mathbf{0}_{\bR^{n-1}})$ is zero except maybe for its first coordinate $w_1$.
Thus for any $\bv = (v_1,\bv_2) \in \bR^n = \bR \times \bR^{n-1}$ one has the splitting
\begin{equation}
\begin{split}
F_{M,\bw}(v_1, \bv_2)^2 =& \left(|v_1|^2 + ( w_1 v_1)_-^2\right) + \|\bv_2\|^2 \\:= & F_1(v_1)^2+F_2(\bv_2)^2.
\end{split}
\end{equation}
Using the compatibility of norm duality with such splittings, and the special cases $n=1$ and $\bw=0$ mentioned above, we obtain
\begin{equation*}
\begin{split}
	(F_{M, \bw}^*(\cot v_1, \cot \bv_2))^2 &= (F_1^*(\cot v_1))^2+ (F_2^*(\cot \bv_2))^2  \\
&=\frac{|\cot v_1|^2+ (w_1\cot v_1)_+^2}{1+|w_1|^2} + \|\cot \bv_2\|^2,
\end{split}
\end{equation*}
which is exactly of the form \eqref{eq:dualFMw}. The general case for arbitrary $\bw$ and symmetric positive definite $M$ follows from affine invariance. Indeed let $A$ be an invertible $n \times n$ matrix, and let $M' = A^\trans M A$ and $\bw' = A^\trans \bw$. Let $F = F_{M,\bw}$ and $F'=F_{M',\bw'}$, so that $F'(\bv)=F(A\bv)$ for all $\bv \in \bR^n$. Let $F^*$, $\hat M$, $\hat \bw$, and $F'^*$, $\hat M'$, $\hat \bw'$, be respectively the dual norms and the matrices defined by the explicit formulas above. Then denoting $B := (A^\trans )^{-1}$ one has by the definition of dual norms that $F'^*(\cot \bv) = F^*(B \cot \bv)$ for all $\cot \bv \in \bR^n$, and by the explicit formulas $\hat M' = B^\trans \hat M B$, $\bw' = B^\trans \bw$. Thus, $F^*=F^*_{M, \bw}$ holds if and only if $F'^* = F^*_{M', \bw'}$. Since for any $M, \bw$, there exists a linear change of variables $A$ such that $M'=\Id$ and $\bw'$ is zero except maybe for its first coordinate, the proof is complete.
\end{proof}

Now Proposition~\ref{prop:DualMetric} follows from Lemma \ref{le:normanddualnorm} by writing out the dual norm, using for each $\bp \in \bbM$:
\begin{equation}\label{eq:Mpandwp}
\begin{aligned}
M_{\bp} &= (\mathcal{C}_{1}(\bp))^2 (D_{\ul{n}}^{\ve})^{-1} \oplus (\mathcal{C}_{2}(\bp))^2 I_{d} \quad \textrm{ and } \\
\ul{w}_{\bp}&= \left\{
\begin{array}{ll} \cC_1(\bp) \sqrt{\ve^{-2}-1} \; (\ul{n},\ul{0}), \quad & \text{for } \cF_\ve^+, \\
\mathbf{0}, \quad &\text{for } \cF_\ve, \end{array}\right.
\end{aligned}
\end{equation}
with $D_\bn^\ve$ as in \eqref{Dnve}. Corollary~\ref{cor:eik} then follows by setting the momentum covector $\hat{\bp}={\rm d}U(\bp)$ equal to the derivative of the value function evaluated at $\bp$.

Now that we have derived the eikonal equations, we obtain the
backtracking Theorem~\ref{th:Backtracing} by Proposition~\ref{prop:GeodesicODE} in App.~\ref{app:Backtracing}, which shows us that level sets of solutions of the eikonal equations are geodesically equidistant surfaces and that geodesics are found by an intrinsic gradient descent.

However, to obtain the explicit backtracking formulas we
 differentiate the Hamiltonian, rather than the dual metric, which is equivalent thanks to \eqref{eq:HamiltonBacktracing} (in Remark~\ref{rem:LagrangianHamiltonian} in App.~\ref{app:Backtracing}). 
 We focus below on the model $(\bbM,d_{\cF_\ve^+})$ without reverse gear, since the other case is similar. Let $\bp \in \bM$, let $F := \cF_\ve^+(\bp, \cdot)$, and let $\cot \bp = (\cot \bx, \cot \bn) \in T^*_\bp (\bM)$. Then differentiating w.r.t.\ $\cot \bn$ we obtain
	\begin{equation*}
		 \diff_{\cot \bn} F^*(\cot \bx, \cot \bn)^2
		 =  \cC_2(\bp)^{-2} \, \diff_{\cot \bn} \|\cot \bn\|^2
		 = 2 \, \cC_2(\bp)^{-2} \cot \bn,
	\end{equation*}
	where $\|\cdot\|$ is the Riemannian metric induced by the embedding $\bS^{d-1} \subset \bR^d$.
Differentiating w.r.t.\ $\cot \bx$ we obtain 
	\begin{align}\label{eq:dhatxFstar}
			\diff_{\cot \bx} F^*(\cot \bx, \cot \bn)^2 &
		= \cC_1(\bp)^{-2} \, \diff_{\cot \bx} (
		\cot \bx \cdot D_\bn^\ve \cot \bx -(1-\ve^2) (\cot \bx \cdot \bn)^2_-) \nonumber
		\\ &
		= 2 \, \cC_1(\bp)^{-2}
		\begin{cases}
			D_\bn^\ve \cot \bx & \text{if } \cot \bx \cdot \bn \geq 0,\\
			\ve^2 \Id \cot \bx & \text{if } \cot \bx \cdot \bn \leq 0.
		\end{cases}		
	\end{align}
	The announced result (\ref{simple2}), which is equivalent to its more concise abstract form (\ref{btsimple}), follows by choosing $\cot \bx := \nabla_{\bR^d} U(\gamma(t))$ and $\cot \bn := \nabla_{\bS^{d-1}} U(\gamma(t))$ and a basic re-scaling $[0,L]\in t \mapsto t/L \in [0,1]$. $\hfill \Box$

\begin{remark}\label{rem:dualnormmovingframe}
The computation of the dual norms can be simplified by expressing velocity (entering the Finsler metric) and momentum (entering the dual metric) in a (left-invariant) local, orthogonal, moving frame of reference, attached to the point $\bp = (\bx,\bn) \in \bbM$:
\begin{equation} \label{localframe}
\dot{\ul{p}}= \sum \limits_{i=1}^{2d-1} u^{i} \left.\mathcal{A}_{i} \right|_{\bp}, \ \ \hat{\ul{p}}= \sum \limits_{i=1}^{2d-1} \hat{p}_{i} \, \left.\omega^{i}\right|_{\ul{p}} 
\end{equation}
where a moving frame of reference is chosen such that
\[
\left\{
\begin{array}{l}
u^{d}= \tilde{u} = \ul{n}\cdot \dot{\ul{x}}, \\
\sum \limits_{i=1}^{d-1} (u^{i})^2 = \|\dot{\ul{x}}\|^{2}- (\ul{n} \cdot \dot{\ul{x}})^2, \\
\sum \limits_{i=1}^{d-1} (u^{d+i})^2 = \|\dot{\ul{n}}\|^{2},
\end{array}
\right.
\]
inducing a corresponding dual frame $\{\left.\omega^{i}\right|_{\ul{p}}\}$ via
\begin{equation}\label{eq:dualframe}
\langle \left.\omega^{i}\right|_{\bp}, \left.\mathcal{A}_{j}\right|_{\bp} \rangle =\delta^{i}_{j}, \textrm{ for all } i,j=1,\ldots,2d-1.
\end{equation}
W.r.t. the left-invariant frame the matrices $D_\bn^\ve$, $M_\bp$ as in \eqref{eq:Mpandwp} and $\hat{M_\bp}$ all become diagonal matrices, and the dual can be computed straightforwardly. Furthermore, in this formulation we can see from the expression for the dual $(\cF_0^+)^*$, i.e. in the limit $\ve \downarrow 0$, that the positive spatial control $u^d$ constraint results in a positive momentum $\hat{p}_d$ constraint:
\begin{equation}
(\cF_0^+)^*(\bp,\hat{\bp}) = \sqrt{\frac{(\hat{p}_d)_+^2}{\cC_1^{2}(\bp)} + \frac{1}{\cC_2^{2}(\bp)} \sum_{i = d+1}^{2d-1} (\hat{p}_i)^2}.
\end{equation}
Therefore the eikonal equation in the positive control model $(\bM, d_{\cF_0^+})$ is simply given by
\begin{equation}
\sqrt{
	\frac{\|\nabla_{\bS^{d\!-\!1}}U(\bp)\|^2}{\mathcal{C}^{2}_{2}(\bp)} +
	\frac{ ((\ul{n} \cdot \nabla_{\R^{d}}U(\bp))_+)^2}{\mathcal{C}^{2}_{1}(\bp)}
} = 1
\end{equation}
\end{remark}

\section{Discretization of the Eikonal PDEs}
\label{sec:Implementation}
\label{ch:implement}

\subsection{Causal operators and the fast marching algorithm}

The fast marching algorithm is an efficient numerical method \cite{tsitsiklis_efficient_1995} for numerically solving the static first order Hamilton-Jacobi-Bellman (or simply eikonal) PDE \eqref{eqdef:EikonalPDE} which characterizes the distance map $U$ to a fixed source point $\pSource$. Fast marching is tightly connected with Dijkstra's algorithm on graphs, and \todo{R5.3: typo} in particular it shares the $\cO(K N \ln N)$ complexity, where $N = \#(X)$ is the cardinality of the discrete domain $X \subset \bM$, $X\ni \pSource$, and $K$ is the average number of neighbors for each point.
Both fast marching and Dijkstra's algorithms can be regarded as specialized solvers of non-linear fixed point systems of equations $\Lambda u = u$, where the unknown $u\in \bR^X$ is a discrete map representing the front arrival times, which rely on the a-priori assumption that the operator $\Lambda:\bR^X \to \bR^X$ is \emph{causal} (and monotone, but this second assumption is not discussed here). Causality informally means that the estimated front arrival time $\Lambda u(\bp)$ at a point $\bp \in X$ depends on the given arrival times $u(\bq)$, $\bq\in X$, prior to $\Lambda u(\bp)$, but not on the simultaneous or the future ones.
Formally, one requires that for any $u,v\in \bR^X$, $t\in \bR$:
\begin{equation}
\label{eqdef:CausalityProperty}
\begin{split}
	&\text{If } u^{<t}=v^{<t} \text{ then } (\Lambda u)^{\leq t} = (\Lambda v)^{\leq t}, \\
	&\quad\text{where }
		u^{<t}(\bp) :=
	\begin{cases}
		u(\bp) & \text{if } u(\bp)<t,\\
		+\infty & \text{otherwise},
	\end{cases}
\end{split}
\end{equation}
and  $v^{<t}$, $(\Lambda u)^{\leq t}$ and $(\Lambda v)^{\leq t}$ are defined similarly. 

\paragraph{A semi-Lagrangian scheme.}
We implemented two discretizations of the eikonal equation \eqref{eqdef:EikonalPDE} which benefit from the causality property. The first one is a semi-Lagrangian scheme, inspired by \emph{Bellman's optimality principle} which informally states that any sub-policy of an optimal policy is an optimal policy.
Formally, let $\cF$ be a Finsler metric, and let $U := d_\cF(\cdot, \pSource)$ be defined as the distance to a given source point $\pSource$. Then for any $\bp\in \bM$ and any neighborhood $V$ of $\bp$ not containing $\pSource$ one has the property
\begin{equation}
\label{eq:Bellman}
	U(\bp) := \min_{\bq \in \partial V} d_\cF(\bp,\bq) + U(\bq).
\end{equation}
In the spirit of \cite{tsitsiklis_efficient_1995,sethian_ordered_2001} we discretize \eqref{eq:Bellman} by introducing for each interior $\bp \in X \sm \{\pSource\}$ a small polygonal neighborhood $V(\bp)$, which vertices belong to the discrete point set $X$. The nonlinear operator $\Lambda$ is defined as 
\begin{equation}
\label{eqdef:HopfLax}
\begin{split}
\Lambda u(\bp) :=  \min_{\substack{\{\bq_1, \cdots, \bq_n\}\\ \text{facet of } \partial V(\bp)}} \min_{\xi \in \Xi}
\cF\left(\bp, \sum_{i=1}^n \xi_i \bq_i - \bp\right) \\ + \sum_{i=1}^n \xi_i u(\bq_i),
\end{split}
\end{equation}
where $\Xi = \{ \xi \in \bR^n_+; \, \sum_{i=1}^n \xi_i=1\}$. In other words, the boundary point $\bq \in \partial V(\bp)$ in \eqref{eq:Bellman} is represented in \eqref{eqdef:HopfLax} by the barycentric sum $\bq = \sum_{i=1}^n \xi_i \bq_i$, the distance $d_\cF(\bp, \bq)$ is approximated with the norm $\cF(\bp, \bq-\bp)$, and the value $U(\bq)$ is approximated with the interpolation $\sum_{i=1}^n \xi_i u(\bq_i)$.

We refer to \cite{sethian_ordered_2001,vladimirsky_static_2006} for proofs of convergence, and for the following essential property: the operator \eqref{eqdef:HopfLax} obeys the causality property \eqref{eqdef:CausalityProperty} iff the chosen stencil $V(\bp)$ obeys the following generalized acuteness property: for any $\bq, \bq'$ in a common facet of $V(\bp)$, one has
\begin{equation*}
\< d_{\cot \bp} \cF(\bp, \bq-\bp), \bq'-\bp\> \geq 0.
\end{equation*}
For the construction of such stencils $V(\bp)$, $\bp \in X$, we rely on the previous works \cite{mirebeau_anisotropic_2014,mirebeau_efficient_2013} and on the following observation: the metrics $\cF_\ve$ and $\cF_\ve^+$ associated to the Reeds-Shepp car models can be decomposed as 
\begin{equation}
\label{eq:SplittingVariables}
	\cF(\bp, (\dot \bx, \dot \bn))^2 = \cF_1(\bp, \dot \bx)^2 +\cF_2(\bp, \dot \bn)^2,
\end{equation}
which allows to build the stencils $V(\bp)$ for $\cF$ by combining, as discussed in \cite[p.\ 9]{mirebeau_anisotropic_2014}, some lower dimensional stencils $V_1(\bp)$ and $V_2(\bp)$ built independently for for the spatial $\bx \in \bR^d$ and spherical $\bn \in \bS^{d-1}$ variables.

We discretize $\bS^1$ uniformly, with the standard choice of stencil. We discretize $\bS^2$ by refining uniformly the faces of an icosahedron and projecting their vertices onto the sphere (as performed by the Mathematica$^\text{\textregistered}$ Geodesate function). The resulting triangulation only features acute interior angles, in the classical Euclidean sense, and thus provides adequate stencils since in our applications $\cF_2(\bp, \dot \bn) = \cC_2(\bp) \|\dot\bn\|$ is proportional to the Euclidean norm, see Fig. \ref{fig:Stencils2DAndSphere}. We typically use $60$ discretization points for $\bS^1$, and from $200$ to $2000$ points for $\bS^2$.

We discretize $\bR^d$ using the Cartesian grid $h \bZ^d$, where $h >0$ is the discretization scale.
The norm $\cF_{\ve,1}(\bp, \dot \bx) = \cC_1(\bp) \sqrt{\dot \bx^T (D_\bn^{\ve})^{-1} \dot \bx}$, recall the notation in \eqref{eq:SplittingVariables}, induced by the approximate Finsler metric $\cF_\ve$ on the physical variables in $\bR^d$, is of Riemannian type and strongly anisotropic. In dimension $d \leq 3$, this is the adequate setting for the adaptive stencils of \cite{mirebeau_anisotropic_2014}, built using discrete geometry tools known as lattice basis reduction.
The norm $\cF_{\ve,1}^+(\bp, \dot \bx) = \cC_1(\bp) \sqrt{\dot \bx^T (D_\bn^{\ve})^{-1} \dot \bx+ (\ve^{-2}-1)(\bn,\bx)_-^2}$ induced by $\cF_\ve^+$ on $\bR^d$ is Finslerian (i.e.\ non-Riemannian) and strongly anisotropic. In dimension $d=2$, this is the adequate setting for the adaptive stencils of \cite{mirebeau_efficient_2013}, built using an arithmetic object known as the Stern-Brocot tree.

\begin{figure*}
\centering
\includegraphics[height=4cm]{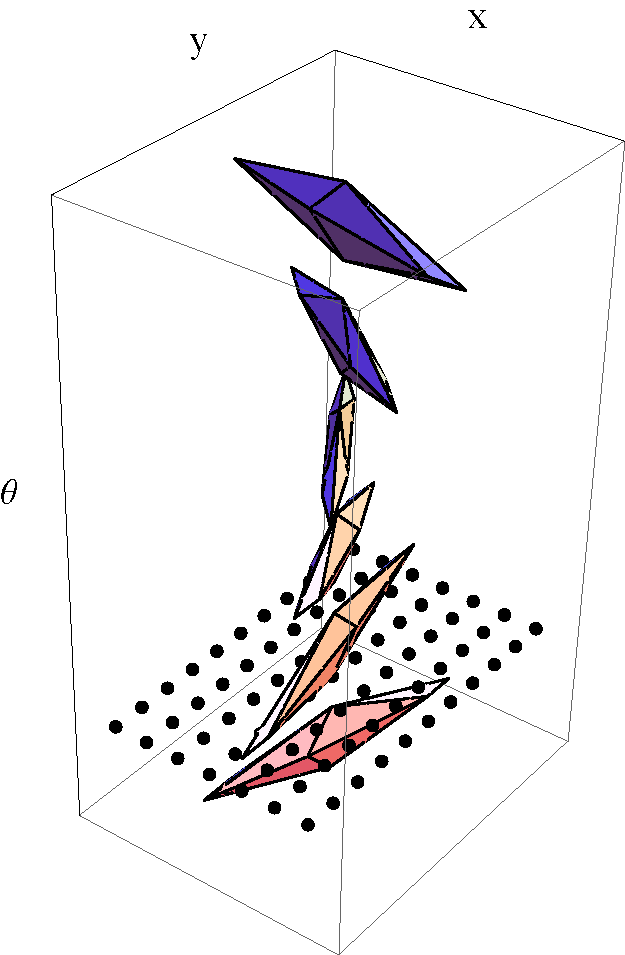}
\hspace{1cm}
\includegraphics[height=4cm]{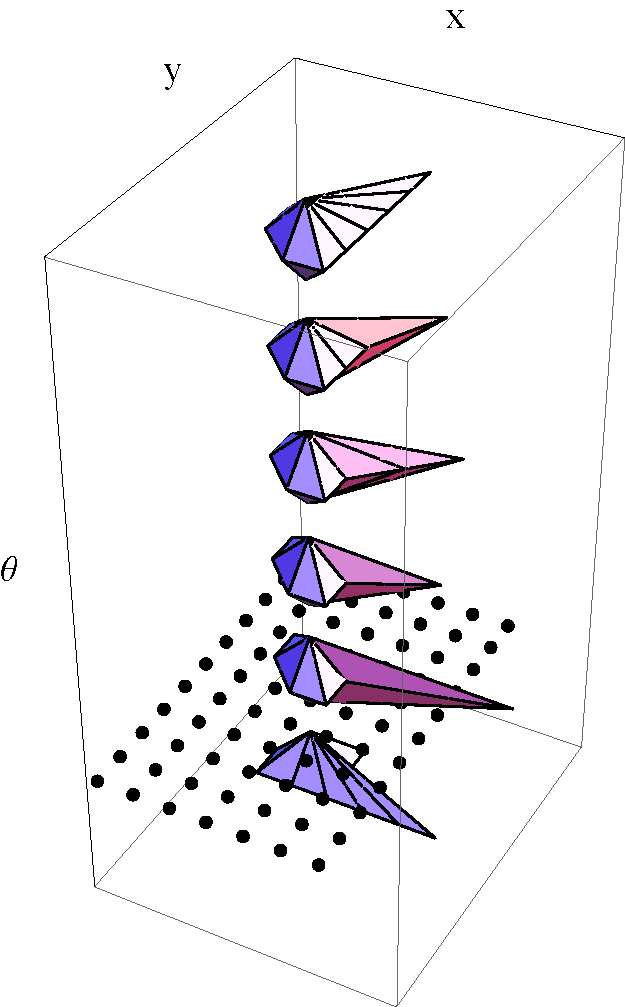}
\hspace{2cm}
\includegraphics[height=3.8cm]{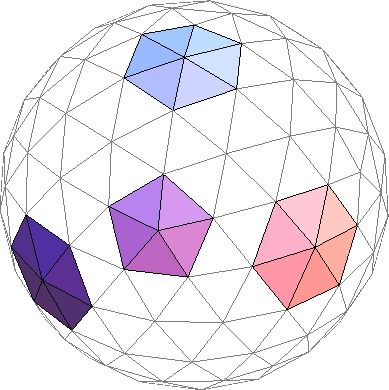}
\caption{Left: Stencil used for the metric $\cF_\ve$ on $\bR^2 \times \bS^1$, $\ve=0.1$, obeying the generalized acuteness property required for the Bellman type discretization \eqref{eqdef:HopfLax}. See also the control sets in Fig.~\ref{fig:freevspositive_introfig_full}. Center: likewise with $\cF_\ve^+$, $\ve=0.1$. Right: Coarse discretization of $\bS^2$ with 162 vertices, used in some experiments posed on $\bR^3 \times \bS^2$. Some acute stencils (in the classical Euclidean sense) shown in color.}
\label{fig:Stencils2DAndSphere}
\end{figure*}

\begin{figure*}
\centering
\includegraphics[height=3.8cm]{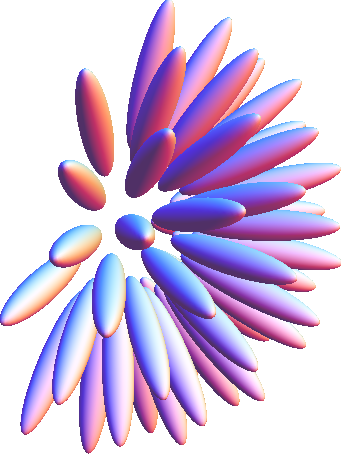}
\includegraphics[height=3.8cm]{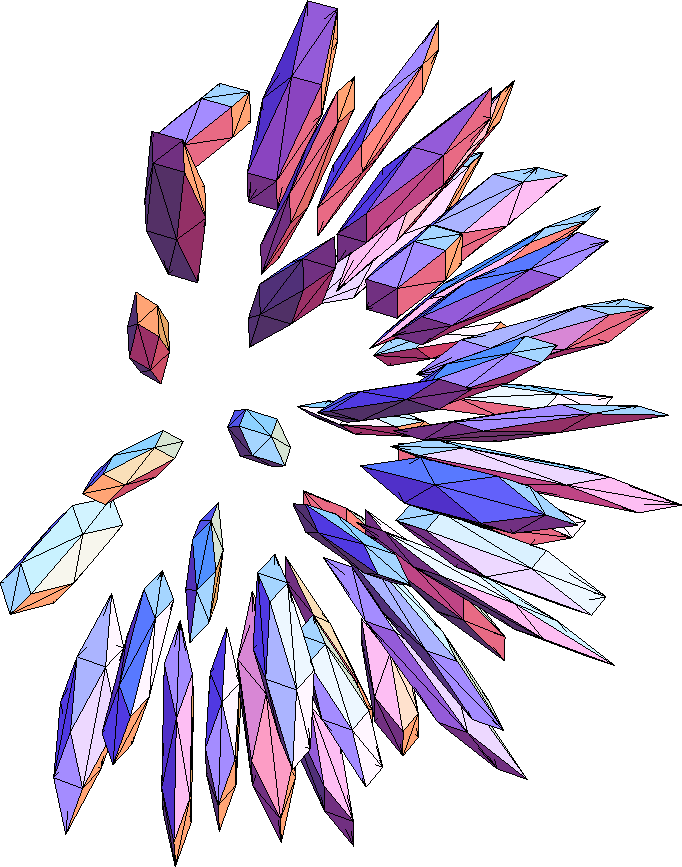}
\hspace{1cm}
\includegraphics[height=3.8cm]{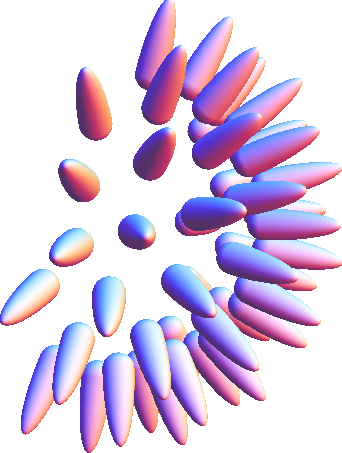}
\includegraphics[height=3.8cm]{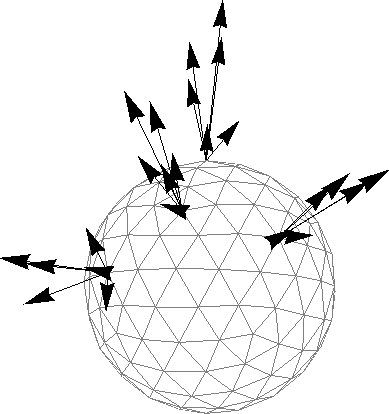}
\caption{Left: Slice in $\bR^3$ of the control sets \eqref{controlset} for $\cF_\ve$ on $\bR^3\times \bS^2$, $\ve=0.2$, for different orientations of $\bn$. Stencils obeying the generalized acuteness property required for Bellman type discretizations \eqref{eqdef:HopfLax}. Right: Slice in $\bR^3$ of the control sets for $\cF_\ve^+$, $\ve=0.2$. Offsets used for the finite differences discretization \eqref{eq:CausalFiniteDiff}, for four distinct orientations $\bn$.}\label{fig:3Dstencils}	
\end{figure*}

\begin{figure}[t!]
\centerline{
\includegraphics[width=\hsize]{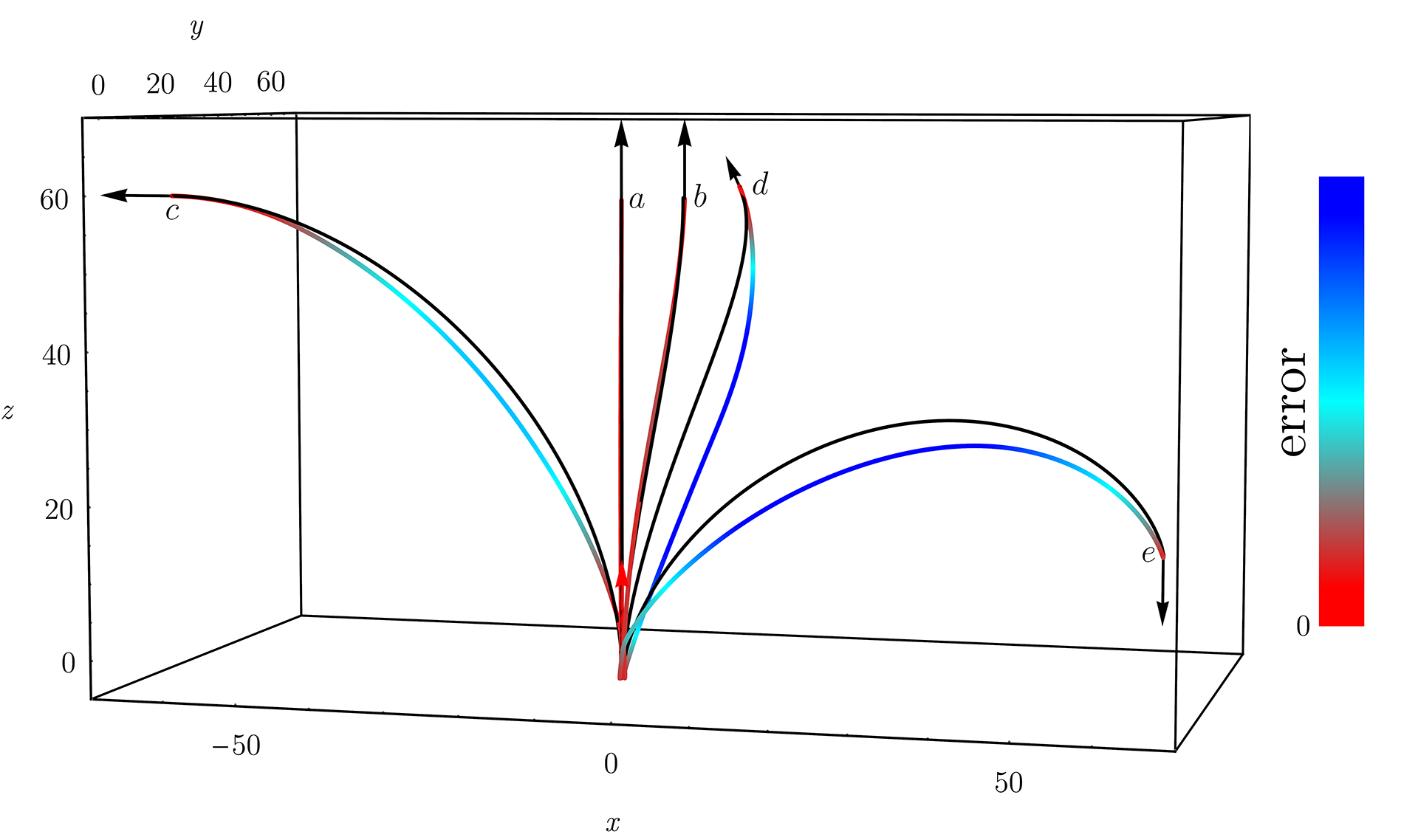}
}
\caption{Comparison of exact geodesics (black curves) and their numerical approximation (colored curves), with $\xi = 1/64$ and $\ve = .1$, for five different end conditions ($a = ((0,0,60),(0,0,1))$, $b = ((6.4,6.4,60),(0,0,1))$, $c = ((-60,0,60),(-1,0,0))$, $d = ((0,60,60),1/\sqrt{6}(-1,2,1))$, $e = ((60,60,10),(0,0,-1))$. The color indicates the error with the exact sub-Riemannian  geodesics \cite{duits_sub-riemannian_2016}.}\label{fig:curvecomparison}
\end{figure}\todo{E.4: typo}

\paragraph{Direct approximation of the Hamiltonian.} A new approach, not semi-Lagrangian, had to be developed for the Finsler metric 
$\cF_\ve^+$ in dimension $d=3$ due to our failure to construct viable (i.e.\ with a reasonably small number of reasonably small vertices) stencils obeying the generalized acuteness property in this case, see Fig. \ref{fig:3Dstencils}.
For manuscript size reasons, we only describe it informally, and postpone proofs of convergence for future work.

Let $\bn \in \bS^2$ and let $\ve>0$ be fixed.  Then one can find \emph{non-negative} weights and \emph{integral} vectors $(\rho_i,\bw_i) \in (\bR_+\times \bZ^3)^6$, such that for all $\bv \in \bR^3$
\begin{equation}
\label{eq:Vor1Eps}
	\sum_{1 \leq i \leq 6} \rho_i (\bw_i\cdot\bv)^2 = (\bn\cdot\bv)^2 + \ve^2 \|\bn \times \bv\|^2.
\end{equation}
A simple and efficient construction of $(\rho_i, \bw_i)_{i=1}^6$, relying on the concept of obtuse superbase of a lattice, is in \cite{Fehrenbach:2013ut} described and used to discretize anisotropic diffusion PDEs.
One may furthermore assume that $(\bn, \bw_i) \geq 0$ for all $1 \leq i \leq 6$, up to replacing $\bw_i$ with its opposite. Then 
\begin{equation}\label{eq:CausalFiniteDiff}
\begin{array}{l}
	\sum_{1 \leq i \leq 6} \rho_i (\bw_i\cdot\bv)_+^2 \approx (\bn\cdot\bv)_+^2,\\ 
	(\bn \cdot \nabla_{\bR^3} U(\bp))_+^2  \approx \\
\frac 1 {h^2}\sum_{i=1}^6 \rho_i (U(\bx,\bn)-U(\bx-h\bw_i,\bn))_+^2,
\end{array}
\end{equation}
up to respectively an $\cO(\ve^2)\|\bv\|^2$ and $\cO(\ve^2+h)$ error.
Following \cite{Rouy:2006ji}, we design a similar upwind discretization of the angular part of the metric
\begin{equation}
\label{eq:CausalFiniteDiffS2}
	\|\nabla_{\bS^2} U(\bp)\|^2 \approx (\delta_\theta U(\bp))^2 + \frac 1 {\sin^2 \theta} (\delta_\vp U(\bp))^2,\\
\end{equation}
where $\delta_\theta U(\bp)$, and likewise $\delta_\vp U(\bp)$, is defined as
\begin{align*}
 \delta_\theta U(\bp) := \frac 1 h \max \{0, &U(\bx,\bn)- U(\bx,\bn(\theta+h, \vp)),\\
	& U(\bx,\bn)-U(\bx,\bn(\theta-h, \vp))\}.
\end{align*}
We denoted by $\bn(\theta, \vp) := (\sin \theta \cos \vp, \sin \theta \sin \vp, \cos \theta)$ the parametrization of $\bS^2$ by Euler angles $(\theta, \vp) \in [0,\pi] \times [0, 2 \pi]$.
Combining \eqref{eq:CausalFiniteDiff} and \eqref{eq:CausalFiniteDiffS2}, one obtains an approximation  of $\cF_0^{+*}(\bp, \diff U(\bp))^2$, within $\cO(\ve^2+r(\ve)h)$ error for smooth $U$, denoted $\gF_\ve U(\bp)$. We denoted by $r(\ve) := \max_{i=1}^6 |\bw_i|$ the norm of the largest offset appearing in \eqref{eq:Vor1Eps}, since these clearly depend on \todo{R5.3: typo} $\ve$.
Importantly, $\gF_\ve U(\bp)$ only depends on positive parts of finite differences $(U(\bp)-U(\bq))_+$, hence the system $\gF_\ve U(\bp)=1$ can be solved using the fast-marching algorithm, as shown in \cite{Rouy:2006ji}. The convergence analysis of this discretization, as the grid scale $h$ and tolerance $\ve$ tend to zero suitably, is postponed for future work, see \cite{mirebeau_voronoi_2017,mirebeau_curvature_2017}. 

Note that this approach could also be applied in dimension $d=2$, and to the symmetric model $(\bM, d_{\cF_{\ve}})$ featuring a reverse gear.
We present only a single assessment of the numerical performance of our method, see Fig. \ref{fig:curvecomparison}. We compare numerically obtained shortest paths with exact SR geodesics for a small number of end points, that correspond to various types of curves. For fair end conditions (a, b, c) the numerical curves are close to \todo{R1: typo} the exact curves. For very challenging end conditions inducing torsion (d) or extreme curvature (e) the curves are further from the exact SR geodesics. An extensive evaluation of the performance of the numerics is left for future work.

\section{Applications}
\label{sec:Applications}

To show the potential of anisotropic fast marching for path-tracing in 2D and 3D (medical) images we performed experiments on each of the datasets in Fig. \ref{fig:introfig_challenges}:
\begin{itemize}
\item a 2D toy example using a map of Centre Pompidou,
\item a 2D retinal image,
\item two synthetic Diffusion-weighted Magnetic Resonance Imaging (dMRI) datasets, with different bundle configurations.
\end{itemize}
We use the 2D datasets to point out the difference in results when using the metric $\cF_\ve$ and $\cF_\ve^+$, and to explain the role of the keypoints when using $\cF_\ve^+$, that occur instead of (possibly unwanted) cusps.

On the synthetic dMRI datasets we present the first application of our methods to this type of data. We present how a cost function can be extracted from the data, and how this leads to correct tracking of bundles, similar to the 2D case. The benefits of anisotropic metrics compared to isotropic metrics are demonstrated by performing backtracking for various model parameter variations.

The experiments were performed using an anisotropic FM implementation written in C++, for $d = 2$ described in \cite{mirebeau_anisotropic_2014}. Implementation details for $d = 3$ will be described in future work. Mathematica 11.0 (Wolfram Research, Inc., Champaign, IL) was used for further data analysis, applying Wolfram LibraryLink (Wolfram Research, Inc., Champaign, IL) to interface with the FM library.

\subsection{Applications in 2D}

\begin{figure*}[t]
\centerline{
\includegraphics[width=\hsize]{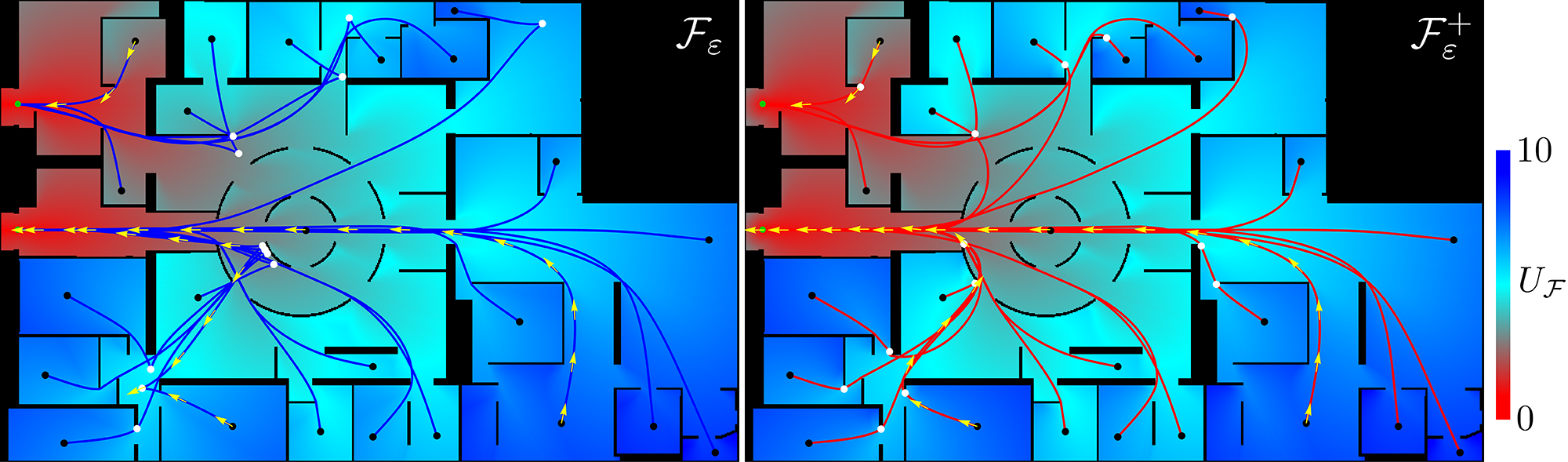}
}
\caption{ Comparison between the shortest paths from end points (black) to one of the exits (green) in a model map of Centre Pompidou, for cars with (left, blue lines) and without (right, red lines) reverse gear. The yellow arrows indicate the orientation of the curve. The background colors show the distances at each position, minimized over the orientation. White points left indicate the cusps, white points right indicate the (automatically placed) keypoints where in-place rotations take place.}
\label{fig:results_centrepompidou}
\end{figure*}

\begin{figure*}[t]
\centerline{
\includegraphics[width=\hsize]{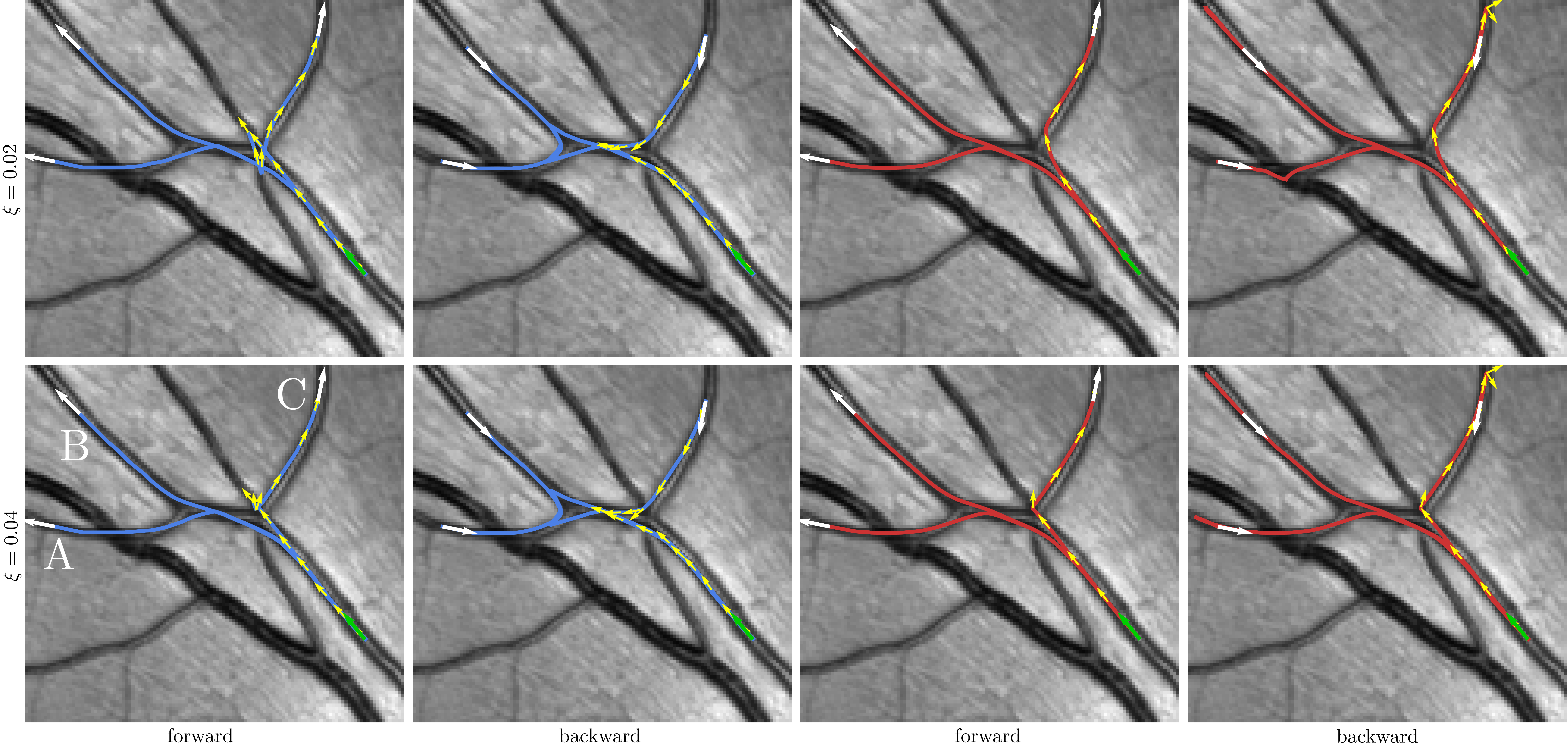}
}
\caption{Left: SR geodesics (in blue) in $(\mathbb{M},d_{\mathcal{F}_{\ve}})$ with given boundary conditions (both forward and backward).
Right: SR geodesics (in red) in $(\mathbb{M},d_{\mathcal{F}_{\ve}^+})$ with the same boundary conditions.
We recognize one end-condition case where on the left we get a cusp, whereas on the right we have a key-point (with in-place rotation) precisely at the bifurcation.
}
\label{fig:results_retina_complete}
\end{figure*}

\subsubsection{Shortest Path to the Exit in Centre Pompidou}
To illustrate the difference between the models with and without reverse gear and to show the role of the keypoints for non-uniform cost, we use a map of Centre Pompidou as a 2D image, see Fig. \ref{fig:results_centrepompidou}. The walls (in black) have infinite cost, everywhere else the cost is $1$. We place end points (black dots) in various places of the museum and look for the shortest path from those points to one of the two exits, regardless of the end orientation. Since there are now two exits, say at $\bp_0$ and $\bp_1$, the distance $U(\bp)$ of any point $\bp \in \bbM$ to one of the exits is given by

\begin{equation}
U_{\cF}(\bp) = \min \{ d_{\cF}(\bp_0,\bp), d_{\cF}(\bp_1,\bp)\}.
\end{equation}

We use a resolution of $N_x \times N_y \times N_o = 706 \times 441 \times 60$. The cost in this example is only dependent on position, but constant in the orientation. Moreover, we use $\cC_1 = \cC_2$ and $\ve = 1/10$.

On the left of Fig.~\ref{fig:results_centrepompidou} we see optimal paths (in blue) obtained using the Finsler 
metric $\cF = \cF_\ve$. The fast marching algorithm successfully connects all end points to one of the exits. Some of the geodesics have cusps, indicated with white points, resulting in backward motion on (a part of) the curve. The colors show the distance $U_{\cF_\ve}$ as above, at each position minimized over the orientations.

On the right, the optimal paths using the asymmetric Finsler 
metric $\cF = \cF_\ve^+$ are shown in red. The curves no longer exhibit cusps, but have in-place rotations (white dots) instead. These keypoints occur in this example on corners of walls. (Due to the fact that $\ve$ is small but nonzero, there can still be small sideways motion.)
The shortest paths for this model are successions of sub-Riemannian geodesics and of in place rotations, which can be regarded as reinitializations of the former: the orientation is adapted until an orientation is found from which the path can continue in an optimal sub-Riemannian way. 

We stress that the fast marching algorithm has no special treatment for keypoints, which are only detected in a post-processing step. We observe that keypoints are automatically positioned at positions where it makes sense to have an in-place rotation. Small differences in the distance maps between $U_{\cF_\ve}$ left and $U_{\cF_\ve^+}$ right can be observed: the constrained model usually has a slightly higher cost right around corners.

\subsubsection{Vessel Tracking in Retinal Images}
Another application is vessel tracking in retinal images, for which the model with reverse gear and the fast-marching algorithm have shown to be useful in \cite{bekkers_pde_2015,sanguinetti_sub-riemannian_2015}. Although the algorithm works fast and led to successful vessel segmentation in many cases, in some cases, in particular bifurcations of vessels, cusps occur. Fig.~\ref{fig:results_retina_complete} shows one such example on the left. The image has resolution $N_x \times N_y \times N_o = 121 \times 114 \times 64$. The cost is constructed as in \cite{bekkers_pde_2015}: the image is first lifted using cake wavelets \cite{duits_image_2006}, resulting in an image on $\mathbb{R}^2 \times \bS^1$. For the lifting and for the computation of the cost function from the lifted image, we rely on their parameter settings. We use $\cC_1 = \xi \cC_2$, with $\xi = 0.02$ (top) and $\xi = 0.04$, and $\ve = 0.1$. The orientations of the end conditions A, B and C (white arrows) are chosen tangent to the vessel, where we considered both the forward and the backward case. The vessel with end condition C is particularly challenging, since it comes across a bifurcation. For the tracking of this vessel, we indicated the orientation with yellow arrows.

The unconstrained model $(\bM,d_{\cF_\ve})$, corresponding to the blue tracks on the left half of Fig. \ref{fig:results_retina_complete}, gives a correct vessel tracking for the forward end conditions of A and B, for both values of $\xi$. This is obviously the better choice than the backward cases. However, for end condition C, neither the forward or backward with neither values of $\xi$ gives a vessel tracking without cusps. On the other hand, if we use the constrained model $(\bM,d_{\cF^+_\ve})$, we obtain an in-place rotation or keypoint in the neighborhood of the bifurcation. Typically a higher value of $\xi$ brings these points closer to the bifurcation. Taking the backward end conditions in combination with this model, we see in some cases that end locations are first passed by the vessel tracking algorithm, until it reaches a point where in-place rotation is cheaper, and then returns to the end position.

\begin{figure*}[t]
\centerline{
\includegraphics[width=\hsize]{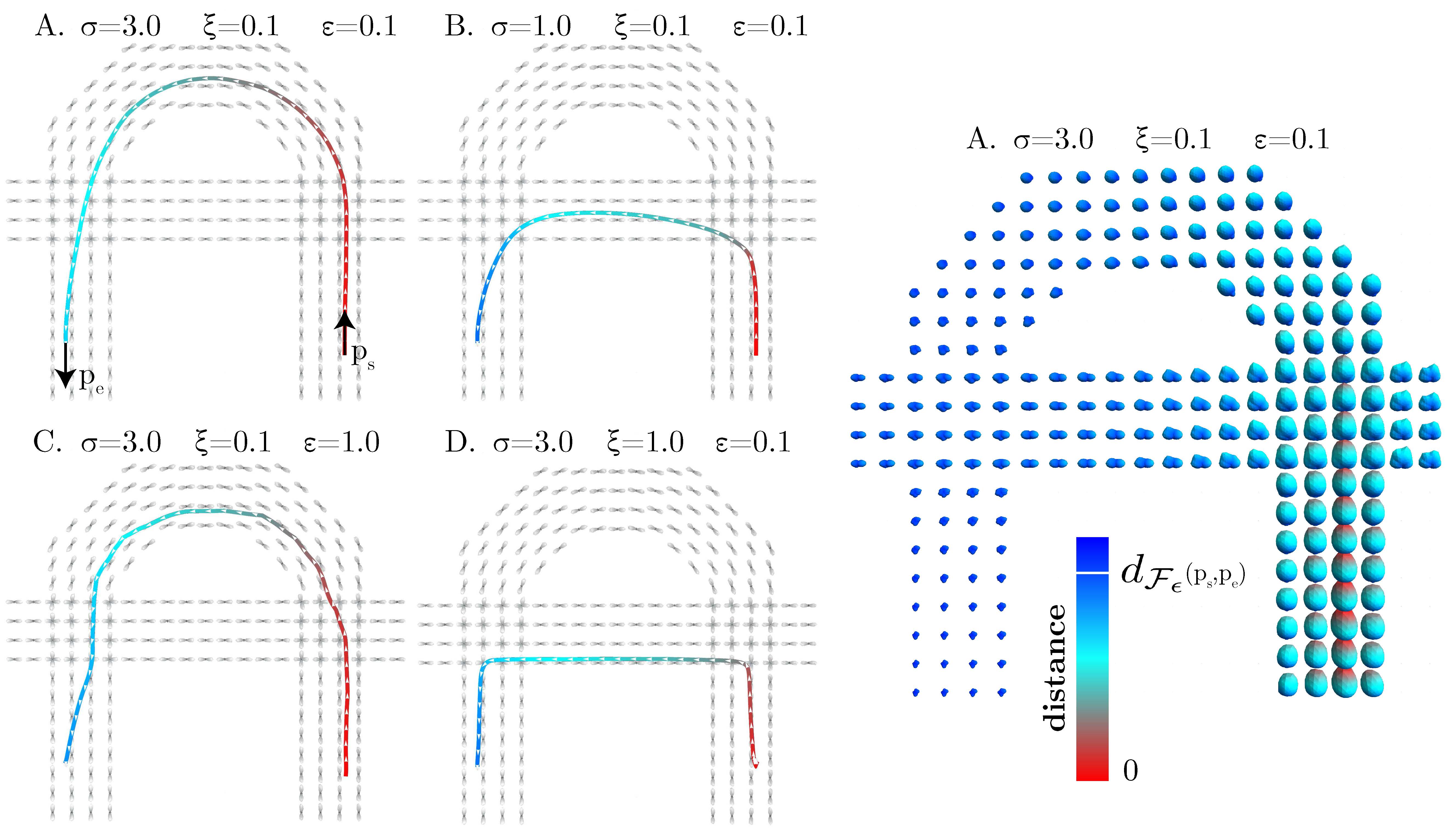}
}
\caption{Comparison of the results of backtracking on a 2D plane in a synthetic dMRI dataset on $\mathbb{M}=\R^{3}\times \mathbb{S}^2$. In case A the default parameters for $\sigma$, $\xi$ and $\ve$ are applied resulting in a global minimizing geodesic (left) and its corresponding distance map (right). Case B reflects the influence of the data-term $\sigma$. Case C reflects the isotropic Riemannian case. Case D reflects a high cost for moving spatially and results in curves that resemble a piecewise linear curve. The distance map is illustrated using a glyph visualization in which the size of the glyph corresponds to $exp(-d_{\mathcal{F}_\ve}(\bp_s,\bp_e)/s)^p$ where $\bp_s$ is the seed location, $\bp_e$ is a location on a glyph, and $s$ and $p$ are chosen based on visualization clarity.  }
\label{fig:results_2Din3D}
\end{figure*}

\subsection{Application to Diffusion-Weighted MRI Data}
DW-MRI is a magnetic resonance technique for non-invasive measurement of water diffusion in fibrous tissues \cite{Mori2007}.
In the brain, diffusion is less constrained parallel to white matter fibers (or axons) than perpendicular to them, allowing us to infer the paths of these fibers.
The diffusion measurements are distributions $(\by, \bn) \mapsto U(\by,\bn)$ within the manifold $\bbM$ for $d=3$.
From these measurements a fiber orientation distribution (FOD) can be created, yielding a probability of finding a fiber at a certain position and orientation \cite{Tuch:2004}.

Backtracking is performed through forward Euler integration of the backtracking PDE involving the intrinsic gradient, following Theorem~\ref{th:Backtracing} and Eq.~\!(\ref{btsimple}) and Eq.~\!(\ref{simple2}). The spatial derivative was implemented as a first-order Gaussian derivative. The angular derivatives are implemented by a first-order spherical harmonic derivative. The latter has the key advantage that in a spherical harmonic basis exact analytic computations can be done. Here, one must rely on two-fold recursions in \cite[Lemma 2 \& 4]{eshagh_alternative_2010}, so that the poles due to a standard Euler angle parametrization of $\bS^{2}$ do not appear in exact recursions of Legendre polynomials!

If data-driven factors $\mathcal{C}_{1}$ and $\mathcal{C}_{2}$ come in a spherical sampling or if one wants to work in a spherical sampling (e.g. higher order tessellation of the icosahedron) in a fast-marching method, then \todo{R1: typo} one can easily perform the pseudo-inverse of the discrete inverse spherical harmonic transform, where one typically keeps the number of spherical harmonics very close to the number of spherical sampling points, so that maximum accuracy order is maintained for computing angular derivatives in the intrinsic gradient descent of Theorem~\ref{th:Backtracing}.

\subsubsection{Construction of the Cost Function}

The synthetic dMRI data is created by generating/simulating a Fiber Orientation Density (FOD) of a desired structure. There are sophisticated methods for this, e.g. \cite{close_software_2009,caruyer_phantomas:_2014}, but evaluation on phantom data constructed with these tools is left for future work. Here we use a basic but practical method on two simple configurations of bundles in $\mathbb{R}^3$, the ones on the bottom row in Fig. \ref{fig:introfig_challenges}. In each voxel inside a bundle, we place a spherical $\delta$-distribution, with the peak in the orientation of the bundle. We convolve each $\delta$-distribution with an FOD kernel that was extracted from real dMRI data and is related to the dMRI signal measured in a voxel with just a single orientation of fibers. Spherical rotation of the FOD kernel is done in the spherical harmonics domain by use of the Wigner D-matrix to prevent interpolation issues. We compose from all distributions an FOD function $W: \bbM \to \R^+$. This function evaluates to high values in positions/orientations that are inside and aligned with the bundle structure.

We use the FOD $W$ to define the cost function $\frac{1}{1+\sigma} \leq \cC \leq 1$ via
\begin{align*}
	\cC(\bp) &= \frac{1}{1 + \sigma \left| \frac{W_+(\bp)}{\|W_+ \|_\infty} \right|^p}
\end{align*}
where $\sigma \geq 0$, $p \in \mathbb{N}$, with $\| \cdot \|_\infty$ the sup-norm and
$W_+(\bp) = \operatorname{max}\{0,W(\bp)\}$.
The cost function $\cC$ induces the following spatial and angular cost functions $(\cC_1, \cC_2)$: 
\begin{equation*}
	\cC_1(\bp) = \xi \cC(\bp), \qquad 	\cC_2(\bp) = \cC(\bp).
\end{equation*}
The implementation of nonuniform cost is comparable to the application of vessel tracking in retinal images in $d=2$ by Bekkers et al.~\cite{bekkers_pde_2015}.

\subsubsection{Influence of model parameters}

The first synthetic dataset consists of a curved and a straight bundle (tube)\todo{E.3: terminology}, which cross at two locations as shown in Fig. \ref{fig:results_2Din3D}.
The experiments using metric $\cF_\ve$ demonstrate the effect of the model parameters on the geodesic back-traced from the bottom-left to the seed location at the bottom-right of the curved bundle\todo{E.3: terminology}.
A distance map is computed for parameter configuration A (Fig. \ref{fig:results_2Din3D}, right) in which suitable values are used for the data-term $\sigma$, and the fast-marching parameters $\xi$ and $\ve$. Furthermore, fixed values are used for data sharpening $p=3$, spatial smoothing $\sigma s=0.5$, forward-Euler integration step size $\delta t=0.04$, and a gridscale of 1. 
By use of these parameters the global minimizing geodesic (Fig. \ref{fig:results_2Din3D}.A, left) is shown to take the longer, curved route.
In parameter configuration B the data-term $\sigma$ is lowered, which creates a geodesic that is primarily steered by internal curve-dependent costs and is shown to take the shortcut route (Fig. \ref{fig:results_2Din3D}.B).
Setting $\ve=1$ in configuration C leads to a Riemannian case where the geodesic resembles a piecewise linear curve.
In configuration D the relative cost of spatial movement relative to angular movement is high, leading to geodesics with shortcuts.

We conclude that configuration $\textbf{A}$ with a relatively strong data term, large bending stiffness ($\xi^{-1}=10$), and a nearly SR geometry ($\ve=0.1$) avoids unwanted shortcuts.

\subsubsection{Positive control constraint}

For the application of FM in dMRI data it is desirable that the resulting geodesic is not overly sensitive to the boundary conditions, i.e. the placement and orientation of the geodesic tip. Furthermore, since neural fibers do not form cusps, these are undesirable in the backtracking results. In Fig. \ref{fig:results_2Dpluscomparison} the backtracking results are shown for the cases without reverse gear $\cF_\ve^+$ (top) and the model with reverse gear $\cF_\ve$ (bottom). The distance map for $\cF_\ve^+$ was computed by the iterative method implementing the forward Reeds-Shepp car, while for $\cF_\ve$ the FM method was used.

We conclude that without the positive control constraint, small changes in tip orientation cause large variations in the traced geodesic in the metric space 
$(\mathbb{M},d_{\mathcal{F}_{\ve}})$, whereas the traced geodesic in the quasi-metric space 
$(\mathbb{M},d_{\mathcal{F}_{\ve}^+})$ is both more stable and more reasonable.

\begin{figure*}[t]
\centerline{
\includegraphics[width=0.6\hsize]{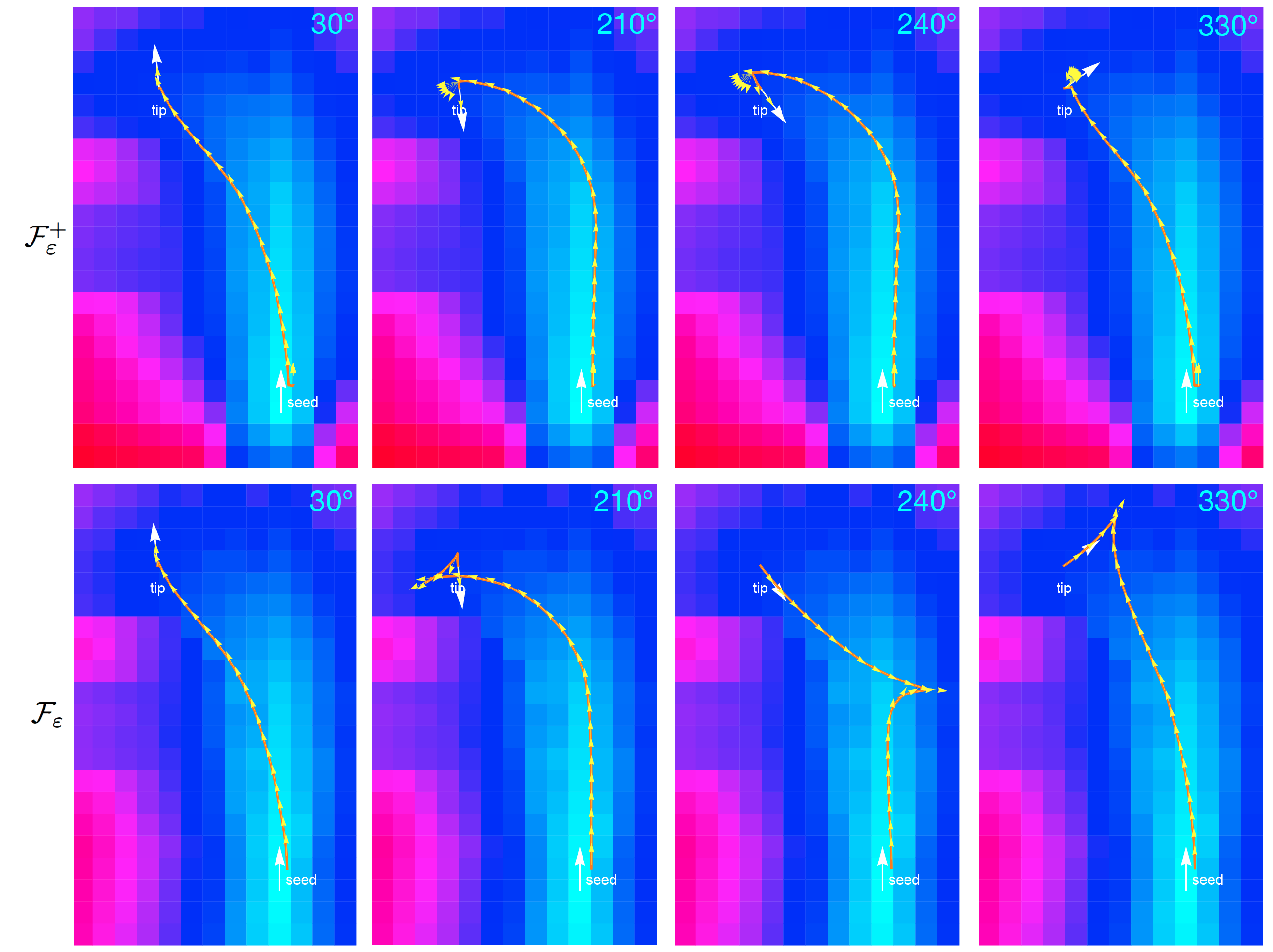}
}
\caption{Backtracking of minimizing geodesics of the model $(\mathbb{M},d_{\cF_\ve^+})$ without reverse gear (top) and the model with reverse gear
$(\mathbb{M},d_{\cF_\ve})$ (bottom) using the model parameters of configuration A ($\sigma=3.0$, $\xi=0.1$ and $\ve=0.1$) for various end conditions. }
\label{fig:results_2Dpluscomparison}
\end{figure*}

\subsubsection{Robustness to neighboring structures}
A pitfall of methods that provide globally minimizing curves using a dataterm is that dominant structures in the data attract many of the curves, much like the highway usually has the preference for cars rather than local roads. This phenomenon is to a certain extent unwanted in our applications, and we illustrate with the following example that it can be circumvented using a sub-Riemannian instead of Riemannian metric. We use the dataset as introduced in Fig. \ref{fig:introfig_challenges}. It consists of one bundle that has torsion (green), that crosses with another bundle (blue), and a third bundle (red) that is parallel with the first in one part. The cost in these bundles is constructed in the same way as above, but now the cost in the red bundle is twice as low as in the other bundles. A small part of the data is visualized on the left of Fig. \ref{fig:results_3Dbundles_complete}. This data is used to construct the cost function as explained above.

The resolution of the data is $N_x \times N_y \times N_z \times N_o = 32 \times 32 \times 32 \times 162$. Again we use $\cC_1 = \xi \cC_2 = \cC$, with $\xi = 0.1$. To have comparable parameters as in the previous experiment, despite increasing the amplitude in one of the bundles by a factor $2$, we choose to construct the cost using parameter $p = 3$, and $\sigma = 3 \cdot 2^p = 24$. From various positions inside the green, blue and red bundle, the shortest paths to the end of the bundles computed by the FM algorithm nicely follow the shape of the actual bundles, when we choose $\ve = .1$ small, corresponding to an almost SR geodesic. This is precisely what prevents the geodesic in the green bundle to drift into the (much cheaper) red bundle. We show on the right in Fig. \ref{fig:results_3Dbundles_complete} that choosing $\ve = 1$, corresponding to having an isotropic Riemannian metric, this unwanted behavior can easily occur.

We conclude that the SR geodesics in $(\bbM = \R^3 \times \bS^2, d_{\cF_\ve})$ with $\ve \ll 1$, are less attracted to parallel, dominant structures than isotropic Riemannian geodesics.

\begin{figure*}[t]
\centerline{
\includegraphics[width=\hsize]{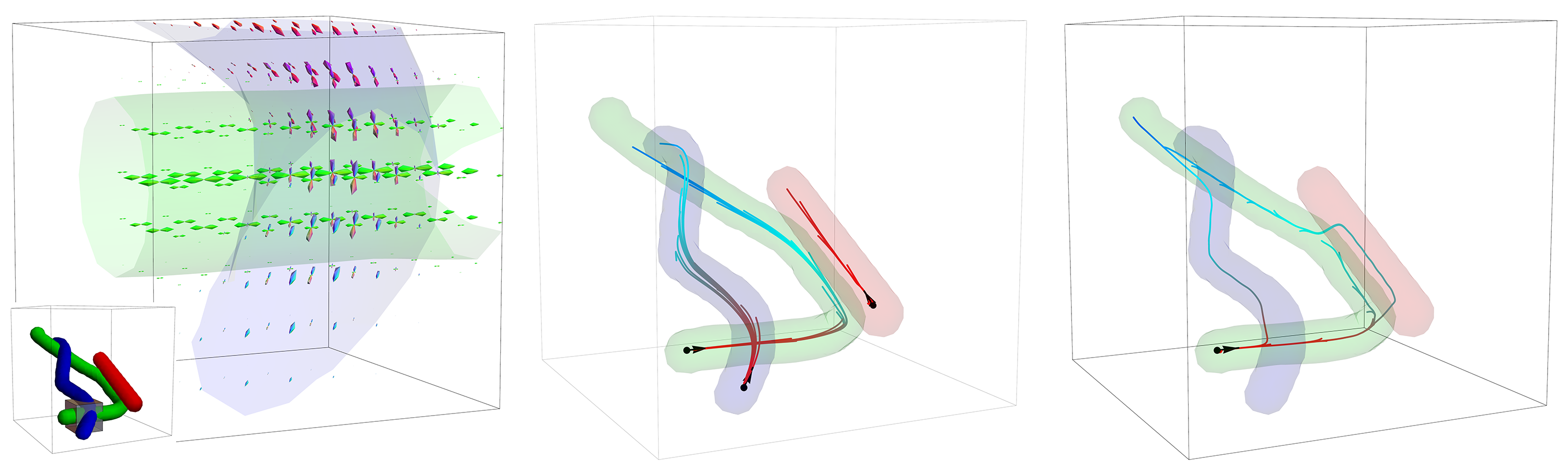}
}
\caption{Left: 3D configuration of bundles and a visualization of part of the synthetic dMRI data. Middle: backtracking of geodesics in $(\mathbb{M}, d_{\mathcal{F}_{\ve}})$ from several points inside the curves to end points of the bundle is successful when using $\ve = 0.1$. Right: when using $\ve = 1$, the dominant red bundle can cause the paths from the green bundle to deviate from the correct structure.}
\label{fig:results_3Dbundles_complete}
\end{figure*}

\section{Conclusion and Discussion}
\label{sec:Conclusion}

We have extended the existing methodology for modelling and solving the problem of finding optimal paths for a Reeds-Shepp car to 3D and to a case without reverse gear. We have shown that the use of the constrained model leads to more meaningful shortest paths in some cases and that the extension to 3D has opened up the possibility for tractography in dMRI data.

Instead of using a hard constraint on the curvature as in the original paper by Reeds and Shepp \cite{reeds_optimal_1990}, 
we used symmetric and asymmetric Finsler metrics. We have introduced these metrics, $\cF_0$ and $\cF_0^+$, for $d  = 2, 3$, such that they allow for curves that have a spatial displacement proportional to the orientation, with a positive proportionality constant in the case of $\cF_0^+$.

We have captured theoretically some of the nature of the distance maps and geodesics following from the new constrained model. We have shown in Thm.~\ref{th:Controllability} that both models are globally controllable, but only the unconstrained model is also locally controllable.

The sub-Riemannian and sub-Finslerian nature is difficult to capture numerically. To this end, we introduced approximating Finsler 
metrics $\cF_\ve$ and $\cF_\ve^+$, that do allow for numerical approaches. We have shown in Thm.~\ref{th:ReedSheppCV} that as $\ve \rightarrow 0$, the distance map converges pointwise and the geodesics converge uniformly, implying that for sufficiently small $\ve$ we indeed have a reasonable approximation of the $\ve = 0$ case.

We have analyzed cusps in the metric space 
$(\mathbb{M},d_{\mathcal{F}_{0}})$ and keypoints in the quasi-metric space 
$(\mathbb{M},d_{\mathcal{F}_{0}^+})$ which occur on the interface surface $\partial \mathbb{M}_{\pm}$ given by (\ref{transitionsurface}). The analysis, for uniform costs, is summarized in Thm.~\ref{th:CuspsAndRotations}. We have shown that cusps are absent in $(\mathbb{M},d_{\mathcal{F}_{\ve}})$ for $\ve>0$, that keypoints in $(\mathbb{M},d_{\mathcal{F}_{0}^+})$ occur only on the boundary, and we provided analysis on how this happens.
In Thm.~\ref{th:Backtracing} we have shown how minimizing geodesics in $(\bbM, d_{\cF_\ve})$ and $(\bbM, d_{\cF^+_\ve})$ can be obtained from the distance maps with an intrinsic gradient descent method.

To obtain solutions for the distance maps and optimal paths, we used a Fast-Marching method. By formulating an equivalent problem to the minimization problem for optimal paths in the form of an eikonal equation, the FM method can be used using specific discretization schemes. We briefly compared the numerical solutions using $\cF_\ve$ with $\ve \ll 1$ with the exact sub-Riemannian geodesics in SE(2) with uniform cost, which showed sufficient accuracy for not too extreme begin and end conditions.

To show the use of our method in image analysis, we have tested it on two 2D problems and two 3D problems. All four experiments confirm that the combination of the eikonal PDE formulation, the Fast-Marching method and the construction of the non-uniform cost from the images, results in geodesics that follow the desired paths. From the experiment on an image of Centre Pompidou, with constant, finite cost everywhere except for the walls, it followed that instead of having cusps when using the Finsler metric 
$\cF_\ve$, we get keypoints (in-place rotations) when using $\cF_\ve^+$. These keypoints turn out to be located on logical places in the image. On the 2D retinal image we showed that the Finsler 
 metric $\cF_\ve^+$ gives a new tool for tackling vessel tracking through bifurcations. We see that keypoints appear close to the bifurcation, leading to paths that more correctly follow the data.

The basic experiments on 3D
show advantages of the model $(\mathbb{M},d_{\mathcal{F}_{\ve}})$ with $0 < \ve \ll 1$ over the model $(\mathbb{M},d_{\mathcal{F}_{1}})$ in the sense that
the minimizing geodesics better follow the curvilinear structure and deal with crossings and nearby parallel bundles (even if torsion is present).
Furthermore, we have shown the advantage of model $(\mathbb{M},d_{\mathcal{F}_{\ve}^+})$ with $0 < \ve \ll 1$, compared to $(\mathbb{M},d_{\mathcal{F}_{\ve}})$ in terms of stability, with keypoints instead of cusps.

The strong performance of the Reeds-Shepp car model in 2D vessel tracking and positive first results on artificial dMRI data, encourages us to pursue a more quantitative assessment of the performance in both 3D vessel tracking problems and in actual dMRI data. Such 3D vessel tracking problems are encountered in for example Magnetic Resonance Angiography.
In future work we will elaborate on the implementation and evaluation of the fast-marching and the iterative PDE implementation of App. \ref{app:iterative}. Furthermore, we aim to integrate locally adaptive frames \cite{duits_locally_2016} into the Finsler metrics $\mathcal{F}_{\ve}$, $\mathcal{F}_{\ve}^+$, for
a more adaptive vessel/fiber tracking.

\section{Acknowledgements}

The authors gratefully acknowledge dr.~G.R.~Sanguinetti 
for fruitful discussion and ideas leading up to this article. We thank E.J. Bekkers for his assistance with and suggestions for Fig. \ref{fig:results_retina_complete}. The research leading to the results of this article has received funding from the European Research Council under the European Community’s 7th Framework Programme (FP7/20072014)/ERC grant agreement No. 335555 (Lie Analysis). This work was partly funded by ANR grant NS-LBR. ANR-13-JS01-0003-01.

\appendix

 \section{Well-posedness and convergence of the Reeds-Shepp models}
 \label{app:WellPosedness}

We introduce in \S \ref{subapp:ClosednessControllability} some general elements of control theory, which are specialized in \S \ref{subapp:ReedSheppSpecialization} to the Reeds-Shepp models and their approximations.

\subsection{Closedness of controllable paths}
\label{subapp:ClosednessControllability}

In this section, we introduce the notion of an admissible path $\gamma$ with respect to some controls $\gB$. We state in Theorem \ref{th:AdmissibleClosed} a closedness result, slightly generalizing the one from \cite{DaChen2016Thesis}, from which we deduce in Corollaries \ref{corol:ExistsMinOC} and \ref{corol:CVMinOC} an existence and a convergence result for a minimum time optimal control problem. The first ingredient of this approach is the notion of Hausdorff distance on a metric space.

\begin{definition}
Given a metric space $\bE$, we let $\cK(\bE)$ be the collection of non-empty compact subsets of $\bE$.
The distance function $d_A : \bE\to \bR_+$ and the Hausdorff distance $\cH(A,B)$, where $A,B \in \cK(\bE)$, are defined respectively by
	\begin{align*}
		d_A(x):= \inf_{y\in A} d(x,y), \
		\cH(A,B):= \max \{\sup_{x \in B} d_A(x), \sup_{x \in A} d_B(x)\}. 
	\end{align*}
\end{definition}

In the following, we fix a closed set $\bX$, contained in an Euclidean vector space $\bE$, or in a complete Riemannian manifold $\bM$.
In the applications considered in this paper, $\bX$ is of the form $\bX_0\times \bS^{d-1}$, where $\bX_0 \subset \bR^d$ is some image domain, see Fig.~\ref{fig:results_retina_complete}, or the set of accessible points in a map (which excludes the walls), see Fig.~\ref{fig:results_centrepompidou}.
The embedding space can be the vector space $\bE = \bR^d \times \bR^d$, which is an acceptable but rather extrinsic point of view, or the Riemannian manifold $\bM = \bR^d \times \bS^{d-1}$, equipped with the metric $\cG_\ve$  for some arbitrary but fixed $\ve>0$, see \eqref{importantmetric}.

We equip the collection of all Lipschitz paths $\Gamma := \Lip([0,1], \bX)$ with the topology of uniform convergence.
We will make use of Ascoli's lemma \cite{Ascoli,Arzela}, which states that any uniformly bounded and equicontinuous sequence of paths admits a converging sub-sequence. In our case the paths are Lipschitz with a common Lipschitz constant.

\begin{definition} \label{def:8}
	Given a normed vector space $V$, we denote by $\gC(V) \subset \cK(V)$ the collection of non-empty compact subsets of $V$, which are convex and contained in the unit ball.
\end{definition}

\begin{remark}\label{rem:convexsubsets}
The restriction to convex subsets is essential. For a uniformly converging sequence of Lipschitz functions $\gamma_n: [0,1] \to \bbM$ with limit $\gamma_*$, with $\dot{\gamma}_n(t) \in K$ for a.e. $t \in [0,1]$ and $K$ a compact set, we can deduce that $\dot{\gamma}_* \in \rm{Hull}(K)$, for a.e. $t \in [0,1]$. The convexity then guarantees that $\dot{\gamma}_* \in K = \rm{Hull}(K)$.
\end{remark}
\begin{definition}\label{def:famofcontrols}
	A family of controls $\cB$ on the set $\bX$ is an element of the set $\gB$ defined by 	\begin{itemize}
		\item If $\bX \subset \bE$ an Euclidean vector space, then $\gB := C^0(\bX,\gC(\bE))$.
		\item If $\bX\subset \bM$ a Riemannian manifold, then $\gB := \{\cB \in C^0(\bX, \cK(T\bM)) \; \vline \; \forall \bp \in \bX,\, \cB(\bp)\in \gC(T_\bp \bM)\}$.
	\end{itemize}
	In both cases, $\gB$ is equipped with the topology of locally uniform convergence.
\end{definition}

\begin{definition}
	A path $\gamma$ is $T\cB$-admissible, where $\gamma \in \Gamma$, $T\in \bR_+$ and $\cB\in \gB$, iff for almost every $t \in [0,1]$
	\begin{equation*}
		\dot \gamma(t)\in T \cB(\gamma(t)).
	\end{equation*}
\end{definition}
We denoted $T B := \{T \bv \; \vline \; \bv\in B\}$, where $T \in \bR_+$ and $B$ is a subset of a vector space. Note the potential conflict of notation with the tangent space $T \bM$ to the embedding manifold $\bM$, which should be clear from context.
If a path $\gamma$ is $T\cB$-admissible for some controls $\cB\in \gB$, then it must be $T$-Lipschitz.
The following result slightly extends, for our convenience, Corollary A.5 in \cite{DaChen2016Thesis}.

\begin{theorem}
\label{th:AdmissibleClosed}
	The set $\{ (\gamma,T,\cB)\in \Gamma \times \bR_+\times \gB \; \vline \; \\
	\hspace*{10em}\text{$\gamma$ is $T\cB$-admissible}\}$ is closed.
\end{theorem}
\begin{proof}
Let $(\gamma_n, T_n,\cB_n)$ be sequences of paths, times and controls converging to $(\gamma_\infty, T_\infty, \cB_\infty)$, and such that $\gamma_n$ is $T_n\cB_n$-admissible for all $n \geq 0$. Since the paths $\gamma_n$ are converging as $n \to \infty$, they lay in a common compact subset
$\bX'$ of the closed domain $\bX$, recall Remark \ref{rem:convexsubsets}. As a result, the restricted controls $\cB_n' :=
(\left.\cB_n\right|_{\bX'})$ are uniformly converging as $n \to \infty$.
In the case where $\bX \subset \bE$ a Euclidean space, applying Corollary A.5 in \cite{DaChen2016Thesis} to the sequence $(\gamma_n, T_n\cB'_n)$ we obtain that $\gamma_\infty$ is $T_\infty\cB_\infty$-admissible as announced.

In the case where $\bX \subset \bM$ a Riemannian manifold, an additional proof ingredient is required. Let $\bM'$ be an open neighborhood of $\bX'$ with compact closure in $\bM$, and let $\cI : \bM' \to \bE$ be an embedding (i.e.\ an injective immersion) with bounded distortion of the manifold $\bM'$ into a Euclidean space $\bE$ of sufficiently high dimension, which by Whitney's embedding theorem is known to exist. Define the set $\bX'' := \cI(\bX')$, the paths $\gamma''_n := \cI \circ \gamma_n$, and controls $\cB''_n(\cI(\bp)) := \diff\cI(\bp, \cB_n(\bp) )$ for all $\bp \in \bX'$ and $n\in \bN \cup \{\infty\}$. Applying again Corollary A.5 in \cite{DaChen2016Thesis} we obtain that $\gamma''_\infty$ is $T_\infty\cB''_\infty$ admissible, hence that $\gamma_\infty$ is $T_\infty\cB_\infty$-admissible as announced.
\end{proof}
In line with the identity \eqref{viewpoint}, we rely on the following definition where we rescale the time interval to $[0,1]$.
\begin{definition}
\label{def:MinTime}
	For any $\cB\in \gB$, $\bp, \bq\in \bX$, we let
	\begin{equation}
	\label{eqdef:MinTime}
	\begin{split}
		T_\cB(\bp, \bq) := \inf\{&T \geq 0 \; | \; \exists \gamma\in \Gamma, \, \gamma(0)=\bp,\, \gamma(1)=\bq, \\ &\hspace*{5em}\text{ and } \gamma \text{ is }T \cB \text{-admissible}\}.
	\end{split}
	\end{equation}
\end{definition}

\begin{corollary}
\label{corol:ExistsMinOC}
	If $\cB \in \gB$, $\bp,\bq\in \bX$ are such that $T_\cB(\bp, \bq) < \infty$, then the inf.\ \eqref{eqdef:MinTime} is attained.
\end{corollary}

\begin{proof}
	Let $T := T_\cB(\bp, \bq)$, and for each $0 < \ve\leq 1$ let $\gamma_\ve$ be a $(T+\ve)\cB$-admissible path from $\bp$ to $\bq$, which is thus $(T+1)$-Lipschitz. By Arzela-Ascoli's lemma \cite{Arzela,Ascoli} there exists a converging sequence of paths $\gamma_{\ve_n}\to \gamma_0$ as $n \to \infty$. The limit path $\gamma_0$ is $T\cB$-admissible by Theorem \ref{th:AdmissibleClosed}, and the result follows.
\end{proof}

\begin{corollary}
\label{corol:CVMinOC}
	For all $\ve \in [0,1]$ let $\cB_\ve \in \gB$. Assume that $\cB_\ve \to \cB_0$ as $\ve \to 0$, and that $\cB_\ve(\bp) \supset \cB_0(\bp)$ for all $\ve \geq 0$, $\bp\in \bX$. Then
	\begin{equation*}
		T_{\cB_\ve}(\bp,\bq) \to T_{\cB_0}(\bp,\bq), \quad \text{as } \ve \to 0.
	\end{equation*}
	Let $T_\ve := T_{\cB_\ve}(\bp,\bq)$ for each $\ve \geq 0$.
	Assume in addition that there exists a unique $T_0\cB_0$-admissible path $\gamma_0$ from $\bp$ to $\bq$, and for each $\ve>0$ denote by $\gamma_\ve$ an arbitrary path from $\bp$ to $\bq$ which is $(\ve+T_\ve)\cB_\ve$ admissible. Then $\gamma_\ve \to \gamma_0$ as $\ve \to 0$.
\end{corollary}

\begin{proof}
The inclusion $\cB_\ve(\bp) \subset  \cB_0(\bp)$, $\forall \bp \in \bM$, implies the inequality $T_\ve \leq T_0$, for all $\ve \geq 0$. Denoting $T_* := \limsup T_\ve$ as $\ve\to 0$, we thus observe that $T_* \leq T_0$. For the reverse inequality $T_* \geq T_0$, we apply Arzela-Ascoli lemma to the family of paths $(\gamma_\ve)_{0 < \ve \leq 1}$ which are $(T_0+1)$-Lipschitz \todo{E.4: typo} by construction, and obtain a converging subsequence of paths $\gamma_{\ve_n} \to \gamma_*$. Theorem \ref{th:AdmissibleClosed} implies the admissibility of $\gamma_*$ with respect to the controls $T_* \cB_0$. Thus $T_*\geq T_0$ but since $T_* \leq T_0$, we must have $T_*=T_0$, and $\gamma_* = \gamma_0$ by the uniqueness assumption. The result follows.
\end{proof}

More generally, if the infimum \eqref{eqdef:MinTime} is realized by a family $(\gamma_i)_{i \in I}$ of paths, then for any sequence $\ve_n\to 0$ one can find a subsequence such that $\gamma_{\ve_{\vp(n)}} \to \gamma_i$ as $n \to \infty$ for some $i \in I$.

\subsection{Specialization to the Reeds-Shepp models}
\label{subapp:ReedSheppSpecialization}

We begin this section by recalling, and slightly generalizing, the notion of Finsler metric introduced in \S \ref{subsec:Geometry}. We then prove that the Reeds-Shepp metrics $\cF_0$ and $\cF_0^+$ are indeed Finsler metrics in this sense.
\begin{definition}
\label{def:Metric}
	A metric on a complete Riemannian manifold $\bM$ is a map $\cF : T\bM \to [0,+\infty]$. With respect to the second variable, it must be $1$-homogeneous, convex, and bounded below by $\delta\|\cdot\|$, where $\delta$ is a positive constant. In terms of regularity, the sets $\cB_\cF(\bp) := \{\dot \bp \in T_\bp\bM \; | \; \cF(\bp, \dot \bp) \leq 1\}$ must be closed and depend continuously on $\bp \in \bM$ with respect to the Hausdorff distance on $T \bM$.
\end{definition}

The next proposition
is due to (\ref{viewpoint}).
\begin{proposition}
\label{prop:DistanceControlTime}
	With the notations of Definition \ref{def:Metric}, the sets $\bp\in \bM \mapsto \cB_\cF(\bp)$ form a family of controls on $(\bM, \delta\|\cdot\|)$. In addition for all $\bp,\bq\in \bM$
	\begin{equation*}
		d_\cF(\bp, \bq) = T_{\cB_\cF}(\bp,\bq).
	\end{equation*}
\end{proposition}

\begin{proposition}	
\label{prop:ReedSheppContinuous}
	The Reeds-Shepp metrics $(\cF_\ve)_{0\leq \ve \leq 1}$ and $(\cF_\ve^+)_{0\leq \ve \leq 1}$ are indeed metrics in the sense of Definition \ref{def:Metric}, for any $\ve \in [0,1]$. The associated controls $\cB_\ve:= \cB_{\mathcal{F}_{\ve}}$, $\cB_\ve^+:= \cB_{\mathcal{F}_{\ve}^+}$ depend continuously on the parameter $\ve\in [0,1]$, and satisfy the inclusions $\cB_\ve(\bp) \subset \cB_{\ve'}(\bp)$ and $\cB_\ve^+(\bp) \subset \cB_{\ve'}^+(\bp)$ for any $\bp \in \bM$ and  $0 \leq \ve \leq \ve' \leq 1$.
\end{proposition}

Proposition \ref{prop:ReedSheppContinuous} allows to apply the results of \S \ref{subapp:ClosednessControllability} to the Reeds-Shepp metrics. Theorem \ref{th:ReedSheppCV} then directly follows from Corollary~\ref{corol:CVMinOC}.
The only remaining non-trivial claim in Proposition \ref{prop:ReedSheppContinuous} is the continuity of the controls on $\bM$, recall Definitions \ref{def:famofcontrols}, and their convergence $\cB_\ve \to \cB_0$ as $\ve \to 0$, as required in Corollary \ref{corol:CVMinOC}.
These two properties are implied by the continuity on $[0,1]\times \bM$, that we next prove, of the following maps
\begin{equation}
\begin{array}{l}\,
 [0,1] \times \mathbb{M} \ni (\ve, \bp) \to \cB_\ve(\bp) \in \gothic{C}(T_{\bp} \mathbb{M}), \\ \,
 [0,1] \times \mathbb{M} \ni (\ve, \bp) \to \cB_\ve^+(\bp) \in \gothic{C}(T_{\bp} \mathbb{M}),
\end{array}
\end{equation}
with $\gothic{C}(T_{\bp} \mathbb{M})$ defined in Definition~\ref{def:8} and equipped with the Hausdorff distance.

\begin{lemma}
\label{lem:HaussMap}
	Let $B$ be a compact subset of a metric space $\bE$, and let $\vp \in C^0(B,\bE)$. Then
	\begin{equation*}
		\cH(B,\vp(B)) \leq \sup_{x \in B} d(x,\vp(x)).
	\end{equation*}
\end{lemma}

This basic lemma, stated without proof, is used in the next lemma to obtain an explicit estimate of the Hausdorff distance between the controls sets of the Reeds-Shepp models.

\begin{lemma}
\label{lem:ExplicitHaussDist}
	Let $\bn_1,\bn_2\in \bS^{d-1}$, let $a_1,a_2,b_1,b_2 \geq 1$, and let $\ve_1,\ve_2\in [0,1]$. For each $i\in \{1,2\}$, let $B_i$ be the collection of all $(\dot \bx,\dot \bn)\in \bR^d\times \bR^d$ obeying
	\begin{equation*}
		\begin{array}{c}
		\dot \bn \cdot \bn_i = 0, \\
		\left\{ \begin{array}{ll}
		a_i^2\|\dot \bn\|^2+b_i^2\left(|\dot \bx \cdot \bn_i|^2+\ve_i^{-2} \|\dot \bx \wedge \bn_i\|^2 \right) \leq 1, & \quad \varepsilon_i > 0 \\
		a_i^2\|\dot \bn\|^2+b_i^2|\dot \bx \cdot \bn_i|^2 \leq 1 \ \text{ and } \ \dot \bx \wedge \bn_i=0 , & \quad \varepsilon_i = 0 \ .
		\end{array} \right.
		\end{array}
	\end{equation*}
	\begin{equation}
	\begin{split}
\textrm{Then }
	\label{eq:ExplicitHaussDist}
		&\cH(B_1, B_2) \leq |a_1^{-1}-a_2^{-1}| + |b_1^{-1}-b_2^{-1}|  \\ &\hspace*{6em}+ \sqrt{2(1-\bn_1\cdot\bn_2)} + |\ve_1-\ve_2|.
	\end{split}
	\end{equation}
	The same estimate holds for the 
	 sets $B_i^+$, $i \in \{1,2\}$, defined by the inequalities
	\begin{equation*}
	\hspace*{-.5em}
	\begin{array}{c}
		\dot \bn \cdot \bn_i = 0, \\

		\left\{\begin{array}{ll}
		a_i^2\|\dot \bn\|^2+b_i^2\left((\dot \bx \cdot \bn_i)_+^2+\ve_i^{-2} (\|\dot \bx \wedge \bn_i\|^2 + (\dot \bx \cdot \bn_i)_-^2)\right) \leq 1, \\
		\hspace*{20.5em} \text{if }\ve_i>0, \vspace{.5em} \\

		a_i^2\|\dot \bn\|^2+b_i^2(\dot \bx \cdot \bn_i)_+^2  \leq 1
\ \text{ and } \ \dot \bx \wedge \bn_i=0, \quad  \dot \bx \cdot \bn_i\geq 0, \\
		\hspace*{20.5em} \text{if }\ve_i=0.
		\end{array}\right.

	\end{array}
	\end{equation*}
\end{lemma}

\begin{proof}
	It suffices to establish the announced estimate \eqref{eq:ExplicitHaussDist} when the tuples $(a_i,b_i,\bn_i,\ve_i)$, $i \in \{1,2\}$, differ by a \emph{single} element of the four, and then to use the subadditivity of the Hausdorff distance. In each case we apply Lemma \ref{lem:HaussMap} to a well chosen surjective map $\vp : B_1 \to B_2$ (resp $\vp^+ : B_1^+ \to B_2^+$).
	\begin{itemize}
		\item Case $a_1 \neq a_2$. Assume w.l.o.g.\ that $a_1 < \infty$, and observe that for all $(\dot \bx,\dot\bn)\in B_1$ one has  $a_1\|\dot \bx\| \leq 1$, hence $\|a_1 \dot \bx / a_2 -\dot \bx\|\leq |a_1^{-1}-a_2^{-1}|$. Choose $\vp(\dot \bx, \dot \bn) := (a_1 \dot \bx / a_2, \dot \bn)$.
		\item Case $b_1\neq b_2$. As above, with $\vp(\dot \bx, \dot \bn) := (\dot \bx, b_1 \dot \bn / b_2)$, yielding upper bound $|b_{1}^{-1}-b_{2}^{-1}|$.
		\item Case $\bn_1\neq \bn_2$. Let $R$ be the rotation of $\bR^d$ which maps $\bn_1$ onto $\bn_2$, in such a way that it maps the space orthogonal to the plane ${\rm Span}(\bn_1,\bn_2)$ onto itself. A simple calculation yields $\|R-\Id\| = 2 \sin[ \frac 1 2 \cos^{-1}(\bn_1 \cdot\bn_2)] = \sqrt{2(1-\bn_1\cdot\bn_2)}$. The result follows by choosing $\vp(\dot \bx, \dot \bn) := (R \dot \bx, R \dot \bn)$, so that
$\|\vp(\dot \bx, \dot \bn) - (\dot \bx, \dot \bn)\| \leq \|R-\Id\|\sqrt{\|\dot \bn\|^2+\|\dot \bx\|^2} \leq \|R-\Id\|$ for all $(\dot\bx,\dot \bn) \in B_1$ as announced.
		\item Case $\ve_1 \neq \ve_2$. Assume w.l.o.g.\ that $\ve_1 > 0$, and
		consider the orthogonal projections
		\begin{align*}
			P_1(\dot \bx) &:= (\dot \bx \cdot \bn_1) \bn_1 &
			P_1^\perp(\dot \bx) &:= (\text{Id}-P_1)(\dot \bx).
		\end{align*}
		Note that $P_1^\perp(\dot \bx) \leq \ve_1$ if $(\dot\bx,\dot\bn) \in B_1$, and that $\|\dot \bx\| \leq \ve_1$ if $(\dot\bx,\dot\bn) \in B_1^+$ and $\dot \bx \cdot \bn_1 \leq 0$. The result follows by choosing
		\begin{align*}
			\vp(\dot\bx,\dot\bn) &:= \left(P_1(\dot \bx)+ \frac {\ve_2}{\ve_1} P_1^\perp(\dot \bx), \dot \bn \right), \\
			\vp^+(\dot\bx,\dot\bn) &:=
			\begin{cases}
			\vp(\dot\bx,\dot\bn) & \text{if } \dot \bx \cdot \bn_1 \geq 0,\\
			( \frac {\ve_2}{\ve_1} \dot \bx,\dot \bn) & \text{otherwise.}
			\end{cases}
			\qedhere
		\end{align*}
	\end{itemize}
\end{proof}

\begin{proof}[Proof of Proposition \ref{prop:ReedSheppContinuous}]
Since working with Hausdorff distances on the abstract tangent bundle $T \bM$ is not very practical, we make use of the canonical embedding $\cI : \bR^d \times \bS^{d-1} \to \bR^d \times \bR^d$ of the manifold $\bM$ into the Euclidean vector space $\bR^{2d}$ given by
$(\ul{x},\ul{n}) \mapsto (\ul{x},\ul{n})$, which has bounded distortion. It suffices to prove the continuity of the image of the control sets $(\ve,\bp) \to \diff\cI (\bp,\cB_{\cF_\ve}(\bp))$ (resp.\ likewise with $\cF_\ve^+$) by the tangent maps to this embedding, which follows by Lemma~\ref{lem:ExplicitHaussDist}. Indeed the lemma shows that

\begin{align*}
( (\ve_1, \bp_1) \rightarrow (\ve_2,\bp_2) ) \implies (\mathcal{H}(B_{\cF_{\ve_1}},B_{\cF_{\ve_2}}) \rightarrow 0),
\end{align*}
and it includes the spherical constraint via the velocity constraint $\dot{\bn} \cdot \bn_i = \frac{d}{dt}(\bn(t) \cdot \bn(t))|_{t = 0} = 0$ for a smooth curve $\gamma(t) = (\bx(t),\bn(t))$ passing through $\gamma(0) = (\bx_i,\bn_i)$.
\end{proof}

\section{Iterative PDE procedure for solving the Eikonal Equation}
\label{app:iterative}

We compare the FM method with an iterative PDE method similar to the one used for the $\R^2 \times \bS^1$-case in \cite{bekkers_pde_2015}, in which the BVP is solved using an iterative procedure (with updating) inspired by mathematical morphology \cite{schmidt_morphological_2016}. To adhere, to the previous work \cite{bekkers_pde_2015},
and for notational convenience we constrain ourselves to the case where the external costs are equal, i.e. $\mathcal{C}_{1}=\mathcal{C}_{2}=\mathcal{C}$, where
of course the general case can be straightforwardly obtained from this special case by a simple position dependent rescaling in the PDE's.

They formulate an auxiliary initial value problem (IVP), for which the 3D analog in $\R^d \times \bS^{d-1}$ with $d \in \{2,3\}$, and with $\ve \geq 0$ is the following:
{\small
\begin{equation}\label{eq:IVP}
\hspace*{-.5em}
\left\{
\begin{aligned}
&\frac{\partial U_{n+1}}{\partial r}(\bp,r) =  \\ &\mathcal{C}^{-1}(\bp) \sqrt{
\begin{split}
\|\nabla_{S^{d\!-\!1}}U_{n+1}(\ul{p},r)\|^2 + \frac{\ve^{2}}{\xi^2}\|\nabla_{\R^{d}}U_{n+1}(\ul{p},r)\|^2 \\ +\frac{1-\ve^2}{\xi^2} |\, \ul{n} \cdot \nabla_{\R^{d}}U_{n+1}(\ul{p},r) \,|^2
\end{split}
}-1, \\
&U_{n+1}(\bp,r_n) = U_n(\bp,r_n), \qquad \text{for } \bp \neq \ul{e}, \\
&U_{n+1}(\ul{e},r_n) = 0, \\
&U_{n+1}(\bp,0) = \delta_{\ul{e}}^M(\bp),
\end{aligned}
\right.
\end{equation}
}
with source point $\ul{p}_{S}=\ul{e}:=(\mathbf{0},\ba)$, for $n \in \mathbb{N}$. Here $r_n = n \ve$ and $r \in [r_n, r_{n+1}]$ are artificial times of the IVP and $\delta_{(\mathbf{0},\ba)}^M$ is the morphological delta, given by

\begin{equation}
\delta_{(\mathbf{0},\ba)}^M(\bp) = \begin{cases}
0 \qquad & \bp = (\mathbf{0},\ba), \\
\infty \qquad & \text{else}.
\end{cases}
\end{equation}
Now the limit
\begin{equation} \label{limit}
U_{\infty}(\bp) := \lim_{\ve \rightarrow 0} \left( \lim_{n \rightarrow \infty} U_{n+1} (\bp, (n+1)\ve )\right)
\end{equation}
gives the viscosity solution $U_{\infty}(\bp)=d_{\mathcal{F}_{\ve}}(\bp, \ul{e})$ of the eikonal equation (\ref{eqdef:EikonalPDE}) for Finsler function $\mathcal{F}_{\ve}$ whose dual is given by (\ref{eq:FEpsPlusStar}).

We approximate the system (\ref{eq:IVP}) with first order, upwind finite differences for the gradients on the right-hand side, and central differences for the time derivate. We use the following stopping criterion:
\begin{equation}
\max_{\bp} |U_{n+1}(\bp,r_{n+1}) - U_n(\bp,r_n)| < \theta, \qquad \theta \in \R.
\end{equation}
The disadvantage of this method (compared to the single pass anisotropic fast-marching method) is the computational load.
The advantage of this PDE-method is a high accuracy near the origin, and that it is very easy to adapt to the (approximative) Reeds-shepp car model without reverse gear (i.e. 
the metric space $(\mathbb{M},\mathcal{F}_{\ve})$) as we explain next. 
Namely $W^+(\bp,e)=d_{\mathcal{F}_{\ve}^+}(\bp, \ul{e})$ is implemented by the same limiting procedure (\ref{limit}) but now applied to
{\small
\begin{equation*}
\left\{
\begin{aligned}\!\!
&\frac{\partial U_{n+1}^+}{\partial r}(\bp,r) = \\ &\mathcal{C}^{-1}(\bp) \sqrt{
\begin{split}
\|\nabla_{S^{d\!-\!1}}U^+_{n+1}(\ul{p},r)\|^2\!+\! \frac{\ve^{2}}{\xi^2}\|\nabla_{\R^{d}}U_{n+1}^+(\ul{p},r)\|^2\! \\ + \frac{1-\ve^2}{\xi^2} |\, (\, \ul{n} \cdot \nabla_{\R^{d}}U_{n+1}^+(\ul{p},r) )_+ \,|^2
\end{split}
}-1, \\
&U_{n+1}^+(\bp,r_n) = U_n^+(\bp,r_n), \qquad \text{for } \bp \neq \ul{e}, \\
&U_{n+1}^+(\ul{e},r_n) = 0, \\
&U_{n+1}^+(\bp,0) = \delta_{\ul{e}}^M(\bp)
\end{aligned}
\right.
\end{equation*}
}

\section{Backtracking of Geodesics in $(\mathbb{M},d_{\mathcal{F}})$}
\label{app:Backtracing}

This section is devoted to a generic ingredient in the proof of Theorem~\ref{th:Backtracing}, regarding
backtracking of Geodesics in the (quasi)-Metric Space $(\mathbb{M},d_{\mathcal{F}})$ in general.
Although, these results are standard in Finsler Geometry, we aim to provide a concise overview.
\begin{lemma}
\label{lem:DualNormDiff}
	Let $F$ be an asymmetric norm on a vector space $\bE$, and assume that $F^*$ is differentiable at $\cot \bp \in \bE^*$. Then
	\begin{align*}
		F(\diff F^*(\cot \bp)) &= 1, &
		\<\cot \bp, \diff F^*(\cot \bp)\> &= F^*(\cot \bp).
	\end{align*}
\end{lemma}

\begin{proof}
	The 1st claim follows by differentiation of $F^*$
	\begin{equation*}
		F^*(\cot \bp) = \sup_{\dot \bp \in \bE\sm \{0\}} \frac {\<\cot \bp, \dot \bp\>}{F(\dot \bp)} = \max_{F(\dot \bp)=1} \<\cot \bp, \dot \bp\>.
	\end{equation*}
	The 2nd claim is Euler's formula for homogeneous functions.
\end{proof}

\begin{proposition}
\label{prop:GeodesicODE}
	Let $\pSource,\pTarget \in \bM$, let $\gamma$ be a minimizing \todo{R5.1: terminology} geodesic from $\pSource$ to $\pTarget$ w.r.t. a continuous metric $\cF$, and let $t \in [0,1]$.
	Assume that the distance map $U$ from $\pSource$ is differentiable at $\gamma(t)$, and that the dual metric $\cF^*$ is differentiable w.r.t.\ the second variable at $(\gamma(t), \diff U(\gamma(t)))$.
	Then $\gamma$ is differentiable at time $t$ and with $L := d_\cF(\pSource, \pTarget)$ 
	\begin{equation}
	\label{eq:appendixGeodesicODE}
		\dot \gamma(t) = L \ \diff_{\cot \bp} \cF^*(\gamma(t), \diff U(\gamma(t))), \; \; \gamma(0) = \bp_S, \gamma(1) = \bp_T.
	\end{equation}
\end{proposition}

\begin{proof} The path $\gamma$ has constant speed $L$, and $t\mapsto U(\gamma(t))$ increases linearly from $0$ to $L$ on it. Let $t \in [0,1]$ be as in the statement of the proposition, and let \[
\dot \Gamma(t) :=  \lim \limits_{n \to \infty}(\gamma(t+\ve_n)-\gamma(t))/\ve_n\] for some sequence $\ve_n\to 0$. Then 
	\begin{align*}
		\cF(\gamma(t), \dot \Gamma(t)) &= L & \textrm{ and }
		\<\diff U(\gamma(t)), \dot \Gamma(t)\> &= L.
	\end{align*}
	For typographic simplicity let us denote $\bp := \gamma(t)$, $\dot \bp := \dot \Gamma(t)$, $F = \cF(\bp, \cdot)$ and $F^* := \cF^*(\bp,\cdot)$. By Lemma \ref{lem:DualNormDiff} and the eikonal equation \eqref{eqdef:EikonalPDE}, the vector $\dot \bq = \diff F^*( \diff U( \bp))$ obeys
	\begin{align*}
		F(\dot \bq) &= F(\diff F^*(\diff U(\bp))) = 1, \\
		\<\diff U(\bp), \dot \bq\> &=  \<\diff U(\bp), \diff F^*(\diff U(\bp))\> = F^*(\diff U(\bp) ) = 1.
	\end{align*}
	Note that the duality-bracket/norm inequality is saturated by $\<\diff U(\bp),\dot \bq\> = 1  =\allowbreak F^*(\diff U(\bp))F(\dot \bq)$, and that the assumed differentiability of the dual norm $F^*$ at the point $\cot \bp = \diff U(\bp)$ implies the strict convexity of the primal norm $F$ (up to $1$-homogeneity)
	at the point $\diff F^*(\cot \bp) = \dot \bq$. Hence $\dot \bq$ is the unique solution to the system ``$F^*(\dot \bq)=1$ and $\<\diff U(\bp), \dot \bq\>=1$'', and therefore  $\dot \Gamma = L \dot \bq$. This implies the differentiability of $\gamma$ at time $t$, and the announced equality \eqref{eq:appendixGeodesicODE}.
\end{proof}

\begin{remark}[Lagrangians and Hamiltonians]
\label{rem:LagrangianHamiltonian}
Given an arbitrary Finsler metric $\cF$ on $\bM$, its half-square $\gL := \frac 1 2 \cF^2 : T(\bM) \to [0,+\infty]$ is usually called the Lagrangian.
The shortest path problem \eqref{eqdef:dF} can be reformulated in terms of the Lagrangian, thanks to the Cauchy-Schwartz's inequality which gives
\begin{align}
\label{eqdef:Ef}
	d_\cF(\bp,\bq)^2 = \inf\{ \int_0^1 \cF(\gamma(t), \dot \gamma(t))^2 dt \, | \, \gamma \in \Lip([0,1], \bM) \nonumber \\, \gamma(0)=\bp, \gamma(1)=\bq\}.
\end{align}
A path $\gamma$ is a minimizer of \eqref{eqdef:Ef} iff it is simultaneously  normalized and a minimizer of \eqref{eqdef:dF}.
The Hamiltonian $\gH$ is the Legendre-Fenchel \todo{E.4: typo} transform of its Lagrangian $\gL$ w.r.t. the second variable, hence $\gH = \frac 1 2 (\cF^*)^2$
(for details see \cite[ch.14.8]{bao})
The eikonal equation can thus be rephrased in terms of the Hamiltonian: 
	\begin{equation*}
		\cF^*(\bp, \diff U(\bp)) = 1 \ \Leftrightarrow \ \gH(\bp, \diff U(\bp))=\frac 1 2.
	\end{equation*}
The Hamiltonian can also be used to reformulate the backtracking ODE of geodesics, thanks to the following identity which follows from the eikonal equation: for any $\bp \in \bM$
\begin{equation}
\label{eq:HamiltonBacktracing}
\begin{split}
	\diff_{\cot \bp} \gH(\bp, \diff U(\bp))
	=& \cF^*(\bp, \diff U(\bp)) \, \diff_{\cot \bp} \cF^*(\bp, \diff U(\bp)) \\
	=& \diff_{\cot \bp} \cF^*(\bp, \diff U(\bp)).
\end{split}
\end{equation}
In geometric control theory this Hamiltonian is often referred to the `fixed time Hamiltonian of the action functional', cf.~\cite{agrachev_control_2004,bekkers_pde_2015,sachkov_cut_2011},
and is typically used \cite{moiseev_maxwell_2010} in the Pontryagin maximum principle \cite{agrachev_control_2004} for (sub-)Riemannian geodesics.
\end{remark}

\section{Characterization of Cusps: Proof of Lemma~\ref{lem:cusp}}
\label{app:cusp}

Consider Lemma~\ref{lem:cusp}. The structure of this lemma is
$a \desda b \desda c$.
The implication
$a \Rightarrow b$ is trivial. The equivalence $b \Leftrightarrow c$ follows by Theorems~\ref{th:Backtracing},~\ref{th:ReedSheppCV}. The implication $b \Rightarrow a$ remains.

Suppose the $d$-th spatial control aligned with $\ul{n}(t_0)$, recall \eqref{def:utilde}, vanishes: $\tilde{u}(t_0)=0$. Now we show by contradiction that in this case $\dot{\tilde{u}}(t_0) \neq 0$.
Suppose $\tilde{u}(t_0)=\dot{\tilde{u}}(t_0)=0$.

Then by application of the PMP (Pontryagin Maximum Principle), similar to \cite[App.A]{bekkers_pde_2015}, \cite{duits_cuspless_2014}) and coercivity/invertibility of the SR-metric tensor $\left.\mathcal{G}_{0}\right|_{\gamma(t_0)}$, recall (\ref{importantmetric}), constrained to the horizontal part of the tangent space $\left.\Delta \right|_{\gamma(t)}=\{(\bp_0=(\ul{x}_0,\ul{n}_{0}), \dot{\bp}_0=(\dot{\ul{x}}_0,\dot{\ul{n}}_{0})) \in T(\mathbb{M}) \; |\;
\ul{n}_{0} \equiv \dot{\bx}_0\}$, that the (analytic) spatial control variable $\tilde{u}= \mathcal{C}_{1}^{-2}\tilde{\lambda}$ vanishes for all times (for $d=2$ this is directly deduced from the pendulum phase portrait \cite{moiseev_maxwell_2010} in momentum space).
This leaves only purely angular momentum and motion, contradicting $\dot{\ul{x}}(\cdot) \neq \ul{0}$ in Lemma~\ref{lem:cusp}.

Next we verify 
$
\tilde{u}(t_0)=\dot{\tilde{u}}(t_0)=0 \Rightarrow  \dot{\tilde{\lambda}}(t_0)=0 =\tilde{\lambda}(t_0)$.
By the chain rule for differentiation (applied to the $d$-th spatial momentum component
$\tilde{\lambda}(t)= \langle \lambda(t) , (\ul{n}(t), \ul{0})\rangle$):
\[\begin{array}{ll}
\left.\frac{d}{dt} \tilde{\lambda}(t)\right|_{t=t_0} &= \left. \frac{d}{dt} (\mathcal{C}_{1}(\gamma(t)))^{-2} \tilde{u}(t)\right|_{t=t_0} \\ &
 =\left. \frac{d}{dt} (\mathcal{C}_{1}(\gamma(t)))^{-2}\right|_{t=t_0}\, \tilde{u}(t_0)+ \\ &\hspace*{7em}\left. \frac{d}{dt} (\mathcal{C}_{1}(\gamma(t)))^{-2}\right|_{t=t_0}\, \dot{\tilde{u}}(t_0)=0.
 \end{array}
\]
We deduce from PMP's Hamiltonian equations (cf.~\cite{duits_cuspless_2014}) that
\[
\dot{\tilde{\lambda}}(t_0)= \tilde{\lambda}(t_0)=0 \Rightarrow  \tilde{\lambda}(\cdot)=0 \Rightarrow \tilde{u}(\cdot)=0.  
\]

\section{On the Hamiltonian discretization}\label{app:ham}

This appendix is devoted to the rigorous formulation and proof of \eqref{eq:CausalFiniteDiff}. This particular result does not appear in the journal version of this paper, because it makes more sense within a complete convergence analysis for this discretization, to appear soon.
\begin{proposition}
Let $\bn \in \bS^{d-1}$, and let $ \bw_1, \cdots, \bw_k \in \bR^d$ and $\rho_1,\cdots, \rho_d\in \bR_+^d$ be such that 
\begin{align}
\label{eq:TensorDecomposition}
&\forall \bv \in \bR^d,\ \sum_{i=1}^k \rho_i |\bw_i\cdot\bv|^2 = |\bn\cdot\bv|^2  +\ve^2 \|\bn \times \bv\|^2. \nonumber 
\\ &\hspace*{4em} \forall 1 \leq i \leq k, \quad (\bn \cdot \bw_i) \geq 0.
\end{align}
Then $\forall \bv \in \bR^d$ the positive part of the scalar product $\bn\cdot \bv$ can be approximated as follows
\begin{equation}
\label{eq:PositiveProductFramed}
  \ (\bn\cdot\bv)_+^2 \leq  \sum_{i=1}^k \rho_i (\bw_i\cdot\bv)_+^2 \leq (\bn\cdot\bv)_+^2 +\ve^2 \|\bn \times \bv\|^2.
\end{equation}
\end{proposition}
\begin{proof}
We may assume that $\rho_i=1$, for all $1 \leq i \leq k$, up to replacing $\bw_i$ with $\sqrt{\rho_i} \bw_i$.
Denote by $\bw_i^\perp := \bw_i - \<\bw_i,\bn\>\bn$ the orthogonal projection of $\bw_i$ on the hyperplane orthogonal to $\bn$. Then by \eqref{eq:TensorDecomposition} 
\begin{align*}
  \sum_{1\leq i \leq k}  |\bn\cdot\bw_i|^2 &= 1, &
  \sum_{1 \leq i \leq k} \bw_i^\perp \otimes \bw_i^\perp &= \ve^2 (\Id-\bn \otimes \bn).
\end{align*}
The proof of \eqref{eq:PositiveProductFramed} is split into two parts, depending on the sign of $(\bn\cdot\bv)$.
  If $(\bn\cdot\bv) \leq 0$, then $(\bw_i\cdot\bv) \leq (\bw_i^\perp\cdot \bw)$ for all $1 \leq i \leq k$, thus as announced
  \begin{equation*}
    \begin{split}
    \sum_{1 \leq i \leq k} (\bw_i\cdot\bv)_+^2 
    \leq  \sum_{1 \leq i \leq k} (\bw_i^\perp\cdot\bv)_+^2 
    \leq  \sum_{1 \leq i \leq k} |\bw_i^\perp\cdot\bv|^2 \\ 
    = \ve^2\|\bn \wedge \bv\|^2.
    \end{split}
  \end{equation*}
  In contrary if $(\bn\cdot\bv) \geq 0$, then the RHS of \eqref{eq:PositiveProductFramed} is immediate, and in addition $(\bw_i\cdot\bv)^2_+ \geq |\bw_i\cdot\bv|^2-|\bw_i^\perp\cdot\bv|^2$ for any  $1 \leq i \leq k$. (Indeed, if $(\bw_i\cdot\bv)\geq 0$ then $(\bw_i\cdot\bv)^2_+ = |\bw_i\cdot\bv|^2 \geq  |\bw_i\cdot\bv|^2 -|\bw_i^\perp\cdot\bv|^2$, and in contrary if $(\bw_i\cdot\bv)\leq 0$ we get $(\bw_i\cdot\bv)^2_+ = 0 \geq |\bw_i\cdot\bv|^2-|\bw_i^\perp\cdot\bv|^2$.) Hence, we conclude 
  \begin{equation*}
  \hspace{-.5em}
    \sum_{1 \leq i \leq k} (\bw_i\cdot\bv)_+^2 \geq 
\sum_{1 \leq i \leq k} |\bw_i\cdot \bv|^2 - |\bw_i^\perp\cdot\bv|^2 
= |\bn\cdot\bv|^2.
\qedhere
  \end{equation*}
\end{proof}

\begin{table*}
\renewcommand{\arraystretch}{1.}
\section{Table of Notations \label{app:TableOfNotations}} 

\begin{tabular}{|l|p{7cm}|p{5.4cm}|}
\hline
\textbf{Symbol}
& \textbf{Explanation}
& \textbf{Reference}
\\ \hline \hline
\multicolumn{3}{c}{} \vspace{-3mm}
\\ \hline
\rule{0pt}{8pt} $\R^d$, $\bx$ & Position space with vectors $\bx = (x^1, \dots, x^d)^T$. & Sect. \ref{sec:introdistance}, Sect. \ref{subsec:introgeometry}, $\dots$ \\
\hline
\rule{0pt}{8pt} $\bbS^{d-1}$,$\bn$ & Angular space, $\bbS^{d-1} = \{ \bn \in \R^d \; | \; ||\bn|| = 1 \}$. & Sect. \ref{sec:introdistance}, Sect. \ref{subsec:introgeometry}, $\dots$\\
\hline
\rule{0pt}{8pt} $\ba$ & Reference axis. For $d = 2$, $\ba = (1,0)^T$, for $d = 3$, $\ba = (0,0,1)^T$. & Eq. \eqref{origin}, Remark \ref{rem:globalmingeodesic}\\
\hline
\rule{0pt}{8pt} $\bbM$, $\bp$ & Manifold $\bbM = \R^d \times \bbS^{d-1}$, with $\bp = (\bx,\bn) \in \bbM$ & Sect. \ref{sec:introdistance} $\dots$\\
\hline
\rule{0pt}{8pt} $T(\bbM)$, $T^*(\bbM)$, $T_\bp(\bbM)$ & Tangent bundle $T(\bbM) = \{ (\bp,\dot{\bp}) \; |\; \bp \in \bbM, \dot{\bp} \in T_{\bp}(\bbM) \}$, and cotangent bundle $T^*(\bbM)$, with tangent space $T_\bp(\bbM)$. & Sect. \ref{sec:introdistance}, Sect. \ref{subsec:eikonal}, Sect. \ref{subsec:introgeometry}, $\dots$\\
\hline
\rule{0pt}{8pt} $\Gamma$, $\bgamma$ & Space $\Gamma = \Lip([0,1],\bbM)$ of admissible curves, with $t \mapsto \bgamma(t) = (\bx(t),\bn(t))$. & Eq. \eqref{eqdef:dF}, $\dots$\\
\hline
\rule{0pt}{8pt} $\cF$, $\cF^*$, $\cF_0$, $\cF_0^+$, $\cF_\ve$, $\cF_\ve^+$, & & \\$(\cF_\ve)^*$, $(\cF_\ve^+)^*$ & \vspace{-5.5mm} Finsler metric $\cF$ defined on $\mathbb{M}$, its dual $\cF^*: T^*(\bbM) \rightarrow \R$ the models with and without reverse gear $\cF_0$, $\cF_0^+$, their approximations $\cF_\ve$, $\cF_\ve^+$ and their duals. & \vspace{-5.5mm} Sect. \ref{sec:introdistance}, Eqs. \eqref{eqdef:ReedSheppMetric}, \eqref{eqdef:ReedSheppForwardMetric}, \eqref{eqdef:EikonalPDE}, \eqref{eqdef:Metric}, \eqref{importantFs1}, \eqref{importantFs2}, Prop. \ref{prop:DualMetric}, $\dots$\\
\hline
\rule{0pt}{8pt} $d_\cF$, $U$ & Distance function $d_\cF(\bp,\bq)$ for $\bp,\bq \in \bbM$, and $U(\bp) = d_\cF(\bp_S,\bp)$ for a fixed source $\bp_S \in \bbM$& Eqs. \eqref{eqdef:dF}, \eqref{eqdef:U}, $\dots$ \\
\hline
\rule{0pt}{8pt} $\ve$ & Anisotropy parameter in the metric, $\ve = 0$ corresponds to the sub-Riemannian manifold case. & Eqs. \eqref{importantFs1}, \eqref{importantFs2}, Fig. \ref{fig:approximatemetric}, \dots\\
\hline
\rule{0pt}{8pt} $\propto$ & We write $\dot{\bx} \propto \bn$ when $\dot{\bx} = \lambda \bn$ for some $\lambda \in \R$ & Eqs. \eqref{eqdef:ReedSheppMetric}, \eqref{eqdef:ReedSheppForwardMetric}, Sect. \ref{subsec:ApproximateReedShepp}, Thm. \ref{th:Controllability}, \\
\hline
\rule{0pt}{8pt} $\cC_1$, $\cC_2$, $\xi$ & External cost $\cC_i: \bbM \rightarrow \bR^+$, analytic and strictly bounded from below, and $\xi > 0$ to balance the cost of spatial motion relative to angular motion, when we choose $\cC_1 = \xi \cC_2$ & Sect. \ref{subsec:introgeometry}, $\dots$
\\
\hline
\rule{0pt}{8pt} $\gB$, $\cB_\cF$ & Set of controls $\gB$, and the set of admissible controls $\cB_\cF(\bp) = \{\dot{\bp} \in T_\bp(\bbM) | \cF(\bp,\dot{\bp}) \leq 1) \}$ & Fig. \ref{fig:freevspositive_introfig_full}, Eq. \eqref{controlset}, \eqref{viewpoint}, Appendix \ref{app:WellPosedness}\\
\hline
\rule{0pt}{8pt} $\ba$ & Reference axis. For $d = 2$, $\ba = (1,0)^T$, for $d = 3$, $\ba = (0,0,1)^T$. & Eq. \eqref{origin}, Remark \ref{rem:globalmingeodesic}\\
\hline
\rule{0pt}{8pt} $(\cdot)_-$, $(\cdot)_+$ & $(\cdot)_- = \min (\cdot, 0)$, $(\cdot)_+ = \max (\cdot, 0)$ & Eqs. \eqref{eq:wedgeminmax}, \dots\\
\hline
\rule{0pt}{8pt} $\gothic{R}$, $\bar{\gothic{R}}$, $\bar{\gothic{R}}^c$ & Subset $\gothic{R} \in \bbM$ of end-points that are reached by cuspless geodesics, the closure $\bar{\gothic{R}}$ and its complement $\bar{\gothic{R}}^c$ & Def. \ref{def:cusplessset}, Remark \ref{rem:globalmingeodesic}, Thm. \ref{th:CuspsAndRotations}, Sect. \ref{ch:proofcuspskeypoints}.\\
\hline
\rule{0pt}{8pt} $\mathcal{A}_i$, $\omega^i$ & Left-invariant frame $\mathcal{A}_i$ and the dual frame $\omega^i$. & Sect. \ref{ch:proofcuspskeypoints}, Eqs. \eqref{localframe}, \eqref{eq:dualframe}, Remark \ref{rem:dualnormmovingframe}.\\
\hline
\rule{0pt}{8pt} $u^i$, $\hat{p}_i$, $\tilde{u}$ & Controls (velocity components) $u^i$, momentum components $\hat{p}_i$ and the special spatial $\tilde{u}$ & Def. \ref{def:cusp}, \eqref{eq:speedmomentuminframe}, \dots \\
\hline
\rule{0pt}{10pt} $\mathcal{G}_{\bp,\ve}$, $\tilde{\mathcal{G}}_{\bp,\ve}$ & Metric tensors  $\mathcal{G}_{\bp,\ve},  \tilde{\mathcal{G}}_{\bp,\ve}: T_\bp(\bbM) \times T_\bp(\bbM) \rightarrow \R^+$ & Eq. \eqref{importantmetric}, \eqref{Gtilde} \\
\hline
\rule{0pt}{10pt} $\nabla$, $\mathcal{G}_{\bp,\ve}^{-1}{\rm d}$, $\tilde{\mathcal{G}}^{-1}_{\bp,\ve}{\rm d}$ & Standard gradient $\nabla = (\nabla_{\bR^d}, \nabla_{\bbS^{d-1}})$, the intrinsic gradient $\mathcal{G}_{\bp, \ve}^{-1}{\rm d}$ of the manifold $(\mathbb{M}_+,d_{\cF_{\ve}})$ and $\tilde{\mathcal{G}}_{\bp, \ve}^{-1}{\rm d}$ the intrinsic gradient of $(\mathbb{M}_-,d_{\cF_{\ve}^+})$ & Cor. \ref{cor:eik}, Thm. \ref{th:Backtracing}, Remark \ref{simple2}, \\
\hline
\rule{0pt}{8pt} $F_{M,\bw}$, $F^*_{\hat{M},\hat{\bw}}$ & Norm $F_{M,\bw} : \R^n \rightarrow \R^+$ and dual norm $F^*_{\hat{M},\hat{\bw}} : (\R^n)^* \rightarrow \R^+$  & Lemma \ref{le:normanddualnorm} \\
\hline
\rule{0pt}{8pt} $X$, $\mathbb{X}$  & Discrete subset $X$ of $\R^d$, and image support $\mathbb{X} \subset \bM$. & Sect. \ref{sec:Implementation}, Appendix \ref{app:WellPosedness}\\
\hline
\rule{0pt}{8pt} $D_\bn^\ve$ & Symmetric positive definite matrix $D_\bn^\ve = \bn \otimes \bn + \ve^2 (\Id-\bn\otimes\bn)$ & Eq. \eqref{Dnve}, \eqref{btsimple}, \dots \\
\hline
\rule{0pt}{8pt} $\bbM_+$, $\bbM_-$, $\partial \bbM_\pm$ & $\mathbb{M}_{+}=\{\bp \in \mathbb{M} \;|\;  \langle {\rm d}U^+(\bp), \ul{n} \rangle > 0 \}$, $\mathbb{M}_{-}=\{\bp \in \mathbb{M} \;|\;  \langle {\rm d}U^+(\bp), \ul{n} \rangle < 0 \}$ and their boundary & Cor. \ref{cor:eik}, Thm. \ref{th:Backtracing}\\
\hline
\rule{0pt}{8pt} $N_x, N_y, N_z, N_o$ & Resolution in spatial/angular coordinates  & Sect. \ref{sec:Applications}\\
\hline
\rule{0pt}{8pt} $\sigma, p$ & Parameters $\sigma > 0$, $p \in \mathbb{N}$ of the cost function $\cC$  & Sect. \ref{sec:Applications}\\
\hline
\end{tabular}\caption{Symbols used throughout the paper, their brief explanation and references to where they appear, or where they are defined/first appear. The dots in the Reference column indicate that they are used frequently.}
\end{table*}
\todo{E.2: spacing table}

\bibliographystyle{plain}

\end{document}